\documentclass[oneside,english]{amsart}
\usepackage{ae,aecompl}
\usepackage[T1]{fontenc}
\usepackage[latin9]{inputenc}
\usepackage{geometry}
\usepackage{babel}
\usepackage{amsbsy}
\usepackage{amstext}
\usepackage{amsthm}
\usepackage{amssymb}
\usepackage{setspace}
\usepackage[authoryear]{natbib}
\usepackage[unicode=true,pdfusetitle,
 bookmarks=true,bookmarksnumbered=false,bookmarksopen=false,
 breaklinks=false,pdfborder={0 0 1},backref=false,colorlinks=false]
 {hyperref}

\makeatletter
\numberwithin{equation}{section}
\numberwithin{figure}{section}
\theoremstyle{plain}
\newtheorem{thm}{\protect\theoremname}
\theoremstyle{definition}
\newtheorem{defn}[thm]{\protect\definitionname}
\theoremstyle{plain}
\newtheorem{prop}[thm]{\protect\propositionname}
\newenvironment{lyxlist}[1]
	{\begin{list}{}
		{\settowidth{\labelwidth}{#1}
		 \setlength{\leftmargin}{\labelwidth}
		 \addtolength{\leftmargin}{\labelsep}
		 }}
	{\end{list}}
\theoremstyle{plain}
\newtheorem{lem}[thm]{\protect\lemmaname}
\theoremstyle{plain}
\newtheorem{fact}[thm]{\protect\factname}
\theoremstyle{definition}
\newtheorem{problem}[thm]{\protect\problemname}

\usepackage[noload]{qtree}
\usepackage{amsfonts}
\usepackage{bussproofs}
\usepackage{pdflscape}

\usepackage{tikz}
\usepackage{tikz-cd}
\usetikzlibrary{matrix,arrows,decorations.pathmorphing,arrows.meta}
\usetikzlibrary{positioning}

\usetikzlibrary{shapes.geometric}

\makeatother

\providecommand{\definitionname}{Definition}
\providecommand{\factname}{Fact}
\providecommand{\lemmaname}{Lemma}
\providecommand{\problemname}{Problem}
\providecommand{\propositionname}{Proposition}
\providecommand{\theoremname}{Theorem}

\begin{document}
\title{Found in Translation\\
\begin{small}
At the limits of the Hudetz program
\end{small}}
\author{Toby Meadows}
\thanks{I'd like to thank Jim Weatherall for all his support in the development
of this project. I'm grateful to the LPS logic seminar for letting
me speak for a three hour (!) session on this material. Thanks are
also due to the Society for Exact Philosophy for allowing me to present
a condensed version of this material in an exhilarating forty-five
minute talk. I'm also very thankful to the graduate students I convinced
to attend a long talk on nascent chunks of this material in the Summer
of 2024. I'd particularly like to thank Laurenz Hudetz for letting
me use his name in the title and for generating many of the ideas
that permeate this paper. I'm also grateful to Bokai Yao for his patient
and well-argued defense of $ZFCU$ over $ZFCA$.}
\begin{abstract}
This paper aims to provide an analysis of what it means when we say
that a pair of theories, very generously construed, are equivalent
in the sense that they are interdefinable. With regard to theories
articulated in first order logic, we already have a natural and well-understood
device for addressing this problem: the theory of relative interpretability
as based on translation. However, many important theories in the sciences
and mathematics (and, in particular, physics) are precisely formulated
but are not naturally articulated in first order logic or any obvious
language at all. In this paper, we plan to generalize the ordinary
theory of interpretation to accommodate such theories by offering
an account where definability does not mean definability relative
to a particular structure, but rather definability without such reservations:
definable in the language of mathematics.
\end{abstract}

\maketitle
\newpage{}

\tableofcontents{}

\newpage{}

Our goal in this article is to find a foothold in a hopelessly general
question:
\begin{quote}
What could we mean when we say that two mathematical objects, or theories
thereof, are interdefinable?
\end{quote}
Textbooks are riddled with such statements. Sometimes the underlying
idea is made precise by saying that, for example, the relevant objects
can compute, construct or interpret each other.\footnote{For classics in each case respectively, see \citep{RogRecT}, \citep{DevlinC}
and \citep{Visser2006}.} But more often that one might like, we find a statement supported
by some kind of mathematical argument, but without any framework in
which to situate the claim. The reader is just supposed to understand.
And quite often, when such an equivalence holds, the argument and
its conclusion do appear to be quite easy to understand. But what
about the inequivalences? What does it mean to say two mathematical
objects \emph{aren't }interdefinable? These questions are more difficult
to answer and they place pressure on us to be more precise in our
formalization of these ideas. Our goal in this paper is to hazard
a very general answer to these questions. They will be the recurring
theme of our discussion. 

Despite the generality of our topic, our entry point into these questions
will come from a specific set of problems in the philosophy of physics.
I anticipate, nonetheless, that the proposed framework poses an interesting,
and perhaps novel, type of problem that will be of interest to mathematical
logicians and mathematicians more broadly. To motivate matters, let's
begin in some well-trodden and stable ground. If we want to compare
two mathematical theories, it is natural to ask whether one theory
can be translated into another; and whether anything is lost in the
process. If those theories are articulated in first order logic, then
the theory of relative interpretation provides compelling criteria
for interdefinability \citep{Visser2006,VisserFriedBitoSyn}. If we
can translate from one theory into the other and back again; and if
we end up saying the same thing, then it seems like nothing has been
lost. Moreover, we might argue that these theories just give us alternative
ways to say the same thing. This is the informal idea behind what
is known as \emph{definitional} \emph{equivalence}. But in this scenario,
observe the crucial roles played by first order logic and translation.
It is not so clear how to do the latter without the former.

In general, theories in physics are not articulated in first order
logic: broadly speaking, physical theories tend to be collections
of mathematical structures. For example, we might think of the theory
of general relativity as being naturally represented by a collection
of mathematical structures that are spacetimes understood as manifolds
endowed with a particular kind of metric. Like a model in first order
logic, these structures are based on a domain with some kind of structure
placed upon it. As with first order logic, the structure consists
of things like functions and relations. But unlike first order logic,
there is no thought that these functions and relations have finite
arity. There is no requirement that they be subsets of products of
the domain. We do not use them to inductively define a set of formulae.
As such, there is no thought that the collection of structures representing
the theory can be isolated by describing a set of sentences that is
true in exactly those structures. We're not in Kansas anymore.

Let us make things a little more concrete with a couple of examples
to illustrate. First, let's note that even when part of a physical
theory can be articulated in first order logic, the intended structure
will often be beyond its expressive range. For example, a theory in
physics will generally need to use the real numbers, and it is possible
to offer a serviceable theory of analysis using first order logic.
Nonetheless in physics, the intended interpretation of that language
is unique up to isomorphism. In contrast the corresponding first order
theory, we do not intend that models with nonstandard numbers and
domains of the wrong cardinality to be included in the structures
instantiating our physical theories. We might say that theories in
physics are defined analytically, from the ``outside,'' rather than
synthetically, from the ``inside.''\footnote{See the remarks at the beginning of Chapter 2 in \citep{hottbook}
for some interesting discussion of this distinction.} As we shall see later, this problem doesn't present an insurmountable
hurdle for translation and interpretation. However, our second problem
is more difficult. If we want to talk about manifolds, then we need
to talk about topology. Leaving aside the question of whether there
is a first order theory of topology, it is not even clear that there
is a first order language in which we could articulate it. In particular,
when we say that a topology on some set is closed under arbitrary
unions, it could appear that the union function -- unlike a function
in first order logic, or indeed, English -- is a function that takes
infinitely many arguments. We'll revisit this case soon enough and
eventually offer a satisfying response to this problem. But for now,
I just want to push the following general point: there is a wide class
of things that we have good reason to call ``theories,'' but which
do not appear to be amenable to articulation in first order logic.
As such, translation seems to be off the menu.

This raises an interesting question: if linguistic tools are unavailable,
how can physical theories be compared? When can we say that they are
equivalent? This is an important question in physics and quite recently
a promising answer has emerged from category theory \citep{WeathWhyNotCats}.
We've already mentioned that theories in physics tend to be understood
as collections of mathematical structures. Recalling that that a category
consists of objects and arrows, we may form a \emph{theory category
}by letting those mathematical structures be the objects and letting
order-preserving maps between those structures be the arrows. To assess
whether two theories are equivalent, we then ask whether their associated
theory categories are isomorphic or equivalent, which roughly means
that they possess the same arrow structure. Such equivalences are
then witnessed by functors that take us back and forth from those
categories and return us to exactly where we started, or at least,
very close to where we started. This move into category theory allows
us to ignore the problem of finding a language to describe mathematical
structures. Rather than using interpretations to define a new model
inside an existing one, we replace them with functors taking structures
to structures. Indeed, this replacement is quite a natural one. For
example, definitional equivalence can be seen as a special case of
isomorphism between theory categories where the required functors
are straightforwardly generated by interpretations.\footnote{See \citep{meadowsBLI} for further discussion of this relationship. } 

Category theory, thus, gives us a way of comparing theories and structures
in physics that avoids the language problem and yet, is quite closely
related to the mathematics and logic of interpretation. However, a
great deal hangs on the arrows we employ in these categories and the
functors we use to relate them. With regard to the arrows, there are
frequently many different order-preserving maps available; for just
a few examples, we might consider homomorphisms, isomorphisms, embeddings,
homeomorphisms or homotopies. Different choices of arrow will change
the structure of the theory category and this can cause an equivalence
to turn into an inequivalence.\footnote{See \citep{BarrettEAI} for an intriguing investigation of this in
relation to classical mechanics.} What makes one kind of arrow the right choice in a theory category?
There is no counterpart to this question in the linguistic realm of
interpretation. With regard to functors, the ability to use arbitrary
functors rather than those determined by interpretations makes it
easier to obtain equivalence results, but it also makes it difficult
to mount philosophical arguments regarding their significance. To
illustrate this note that in contrast, when two theories in first
order logic are found to be definitionally equivalent, we have a simple
story about translation to tell. Anything you can say in one language
has an equivalent in the other as witnessed by the fact that translating
back and forth plausibly preserves the meanings of expressions. It
is a generalization of that moment in a definition-debate where you
realize that you and your interlocutor are arguing past each other.
It is not at all obvious that this story generalizes to the world
of functors and theory categories. As such, while this category theoretic
equivalences provide a significant step forward with these problems,
their philosophical significance is on shakier ground than their translational
counterparts. Without some notion of inter-definability, it is not
so clear what categorical equivalence between theories is telling
us.

In the context of physics, this seems to leave us with an unpleasant
dilemma. If we restrict our attention to theories in first order logic,
we get a straightforward story about equivalence based on interdefinability,
but we aren't able to naturally accommodate many physical theories.
If we move beyond first order logic to class of structures, then physical
theories are easily dealt with, but the philosophical story about
their equivalence becomes more difficult. Indeed, this problem has
been observed before and forms the basis of a distinction between
physical theories promoted by David Wallace \citeyearpar{WallaceMath1st}.
On the one hand, we have the \emph{language-first }approach in which,
roughly, theories are sets of sentences that can be compared using
translation. On the other, we have the \emph{mathematics-first} approach
whereby theories are collections of mathematical structures that are
perhaps best compared using category theory.\footnote{I should say that Wallace does not commit himself to the use of category
theory, although it is -- as we've seen -- a promising approach
to that problem.} Being a philosopher of physics, Wallace is rightfully concerned with
being faithful to his subject matter and its practitioners. As such,
he favors the \emph{mathematics-first }approach to theories in physics,
with regard to which he makes the following intriguing remark: 
\begin{quote}
\emph{This is a conception of theories not as collections of sentences,
but as collections of mathematical models. Of course, I used language
to describe those models to you. (How else could we have communicated?
I'm not telepathic.) \citep{WallaceMath1st}}
\end{quote}
Setting aside Wallace's apparent ineligibility to join the X-men,
something very interesting is being said here. Even when we treat
physical theories as collections of mathematical structures, we still
describe those structures in a language. This observation provides
the fundamental hint for the framework proposed in this paper and
a path around our unpleasant dilemma. We are going to develop an account
that allows us to compare collections of mathematical structures using
translation and interdefinability in the language that we use to describe
those structures. And what language is this? The language of mathematics
itself. 

We now risk opening up another very general question about the foundations
of mathematics, so let's nip this one in the bud. We take it that
a set theory based on $ZFC$ provides a general foundation for mathematics
that will be suitable for the purposes of this paper.\footnote{See \citep{MaddySTF} and \citep{Maddy2019} for some defense of this
claim.} In particular, it is sufficient to prove the existence of the mathematical
structures that we are interested in and perform the manipulations
of those structures that we require. Nonetheless, I think it is very
likely that other foundational systems can be used to furnish us with
frameworks analogous to that delivered in the paper. Indeed, this
is an appropriate moment to acknowledge that Laurenz \citet{Hudetz2017}
was here before us and did exactly that. Noting that ``working scientists
do not restrict themselves to first-order constructions,'' he developed
a framework for theory comparison based on type theory and the Bourbaki's
notion of an echelon. Like the system developed in this paper, it
is not restricted to first order logic. We have opted for a set theoretic
approach for a couple of reasons. First, we think it offers a simpler
framework that is easier to use and closer the practice of working
mathematicians. This will be a controversial point that we do not
aim to unpack or defend. However, the sense in which we make this
claim can be explained by putting forward one of our primary goals.
Our framework should be a silent partner. When we show that classic
equivalences fit into our framework, we shall see that the informal
textbook proof is the same as ours. So even if a mathematician doesn't
know our framework, we can still understand them as working within
it. Like $ZFC$ itself, our framework is intended to be compatible
with mathematical practice and furthermore a reasonable model of it.
Second, while Hudetz's type theoretic framework delivers a notion
of definability that is much more powerful than first order logic,
it is still circumscribed within that background mathematics of its
users.\footnote{We'll discuss this further later, but the easiest way to see this
point is to observe that a model-theoretic semantics for the type
theory is given in \citep{Hudetz2017}.} This is essentially because the type theory is not used as a model
of our background foundational mathematics, but rather formulated
within that milieu. In contrast, our goal is to take Hudetz's idea
to its natural limit: hence, the subtitle of this paper. Our goal
is to develop a framework that models definability as that which can
be defined using the entirety of our background mathematics. 

The paper is divided into three main parts. In the first part, Section
\ref{sec:1}, we set the scene by considering some naive approaches
to our problem. While we'll quickly see that they do not work, the
hurdles we encounter will inform our proposed solution and make clearer
the shape of the problem we are trying to address. In the second part,
Section \ref{sec:V*}, we provide the technical exposition of our
proposed solution, which we call the $\mathcal{V}^{*}$-framework.
While this is the shortest part of the paper, it contains a lot of
technical details that may take a little while to fully digest. Nonetheless,
the final outcome of this work will be generalized characterizations
of definitional equivalence and bi-interpretability that will look
very familiar to students of interpretability. The final part, Section
\ref{sec:Applications}, takes the framework out for a test drive.
We consider some elementary equivalence and inequivalence results
across a variety of theories and then conclude with a discussion of
the limitation of our framework.

The reader will have noted by now that this is not a short paper.
As such, it only seems fair to provide a little guidance on how it
might be read. First of all, while long, I think (or hope) this is
a relatively smooth paper to read. Most of the proofs are easy or
well-known by the folk. As such, I've omitted many proofs of obvious
facts and aimed to include proofs only when I think they are part
of the story.\footnote{And I've tried to provide relatively comprehensive footnotes pointing
to resources that might help the less experienced reader.} By this, I mean those proofs that illustrate an important technique
in the underlying framework. These will often be proofs of relatively
easy claims, but I have tried to curate a collection of such claims
that illustrate the basic toolkit that will be required for addressing
more complex problems. So even though it's long, I think this paper
can be read left-to-right with a reasonable expectation on the part
of the reader that they will not get stuck. Nonetheless, life is short.
For the more experienced reader with less time or patience, I think
it would be quite reasonable to jump ahead to Section \ref{subsec:Definitional-equivalence-and}
after reading this paragraph. There, they will find the core definitions
of the paper. They might then backtrack a little into Section \ref{sec:V*}
in order to fill in some of the gaps. Then they might dive into the
applications offered in Section \ref{sec:Applications} to see how
the framework plays out in practice. This may still leave the reader
wondering about some of the idiosyncrasies of the $\mathcal{V}^{*}$-framework
are the way they are. At that moment, such a reader might return to
Section \ref{sec:1} to better understand the problems that motivated
the definitions of Section \ref{subsec:Definitional-equivalence-and}.

\section{Things that don't work well\label{sec:1}}

Our first goal is to develop a better understanding of what it would
mean to produce a good account of the background mathematical language
we use to described mathematical structures and classes thereof. To
achieve this, I'll begin negatively, by discussing some simple ideas
that work quite badly. This will give us a clearer idea of the nature
and magnitude of our problem since the failures of these naive proposals
will highlight some criteria that a successful proposal ought to satisfy.

\subsection{Straw equivalence}

So let's dive in and try something. Suppose we are modeling our mathematical
practice in $ZFC$ and we ask ourselves the question: when are two
theories -- understood as collections of mathematical structures
-- interdefinable? In order to do anything with this question, we
need to first answer another one. What do we mean by a collection
of mathematical structures? Given that we are working in set theory,
the obvious answer would be: some kind of set. However for well-known
reasons, this turns out to be too restrictive. For example, even in
the simple context of first order logic there is no set of all groups.
There is, however, a definable (proper) class of groups. The fact
that such a class is definable is also important. Following Wallace's
hint, our goal is to model the background language of mathematics
and the way in which we talk about theories. If this is to make sense,
there must be some some way to define those structures. With this
in mind, let's say -- for now -- that a theory is a class that is
definable without the use of parameters. This last bit is important,
since a parameter, like a real number, might be used to smuggle in
infinitary information that we are unable to communicate. Unpacking
things we see that since a theory is a collection of mathematical
structures, a mathematical structure will simply be a set that is
a member of the associated class. There are many reasons to be unsatisfied
with this characterization, but they are not relevant to our first
hurdle.

In order to say when two such theories are intuitively equivalent,
we want to say that there is some sense in which the structures contained
within them are interdefinable. A natural way of doing this occurs
with definitional equivalence between first order theories. We recall
how this works. Let $T$ and $S$ be theories articulated in $\mathcal{L}_{T}$
and $\mathcal{L}_{S}$ respectively; and let $mod(T)$ and $mod(S)$
be the classes of models that satisfy $T$ and $S$. Suppose then
that there are functions $t:mod(T)\to mod(S)$ and $s:mod(S)\to mod(T)$
determined by translations between their languages. By this, we mean
that every article of $\mathcal{L}_{S}$'s vocabulary has a translation
into $\mathcal{L}_{T}$. This allows us to define an $\mathcal{L}_{S}$
model within any model of $T$.\footnote{See Theorem 3.2 in \citep{meadowsBLI} for a description of how to
turn an translation into such a function. Essentially, we just use
the translation of items of vocabulary of $\mathcal{L}_{S}$ into
$\mathcal{L}_{T}$ in order to define a model of $S$ within a model
of $T$.} Let us abuse notation and write $t:mod(T)\leftrightarrow mod(S):s$
for this situation. This means that $T$ and $S$ are mutually interpretable.
We then say that $T$ and $S$ are \emph{definitionally} \emph{equivalent}
if:\footnote{See \citep{Visser2006,VisserFriedBitoSyn} or \citep{meadowsBLI}
for more detailed definitions and discussion. See \citep{LefeverDefEqNonDisjLang}
for alternative characterizations of definitional equivalence and
some traps for young players.}
\begin{enumerate}
\item $s\circ t(\mathcal{A})=\mathcal{A}$ for all models $\mathcal{A}\models T$;
and
\item $t\circ s(\mathcal{B})=\mathcal{B}$ for all models $\mathcal{B}\models S$.
\end{enumerate}
Informally speaking, this tells us the following. Within any model
$\mathcal{A}$ of $T$ we may define a model of $S$ and within any
model $\mathcal{B}$ of $S$ we may define a model of $T$. Moreover,
if we consider the model $t(\mathcal{A})$ of $S$ defined within
$\mathcal{A}$ and then the model $s\circ t(\mathcal{A})$ of $T$
within it, we get exactly the same model. Similarly, when we start
with a model $\mathcal{B}$ of $S$. Thus, we might say that $T$
and $S$ are \emph{interdefinable} since the back-and-forth process
of defining one model in another takes back to exactly where we started.

How might we apply this idea in our vastly generalized setting? In
the context of first order logic, we used definability relative to
a particular model in order to define a new structure. But in our
generalized context, we do not want our notion of definability to
be relativized, unless it is vacuously relativized to the entire universe.
With this in mind, we might modify our characterization of definitional
equivalence to give the following (terrible) definition. 
\begin{defn}
Say that two theories $T$ and $S$ (construed as definable classes)
are \emph{straw equivalent} if there are class functions $t$ and
$s$ such that:\footnote{Strictly speaking, this definition should probably be articulated
in a class theory like $GBN$, but the formulation above should suffice
for the purposes of illustrating the problem we have in mind.} 
\begin{enumerate}
\item Any universe containing some $\mathcal{A}\in T$ is identical to the
universe containing $s\circ t(\mathcal{A})$; and
\item Any universe containing some $\mathcal{B}\in S$ is identical to the
universe containing $t\circ s(\mathcal{B})$.
\end{enumerate}
\end{defn}

Since (in the context of $ZFC$) there only one universe, $V$, it
should be clear that this equivalence relation on theories is all
but trivial. 
\begin{prop}
If $T$ and $S$ are theories (as definable classes) that each have
at least one definable structure within them, then $T$ and $S$ are
straw equivalent.
\end{prop}

Thus for example, the theory of groups is straw equivalent to the
theory of topologies. Of course, there is nothing logically wrong
with the definition. It just fails to correspond well with our intuitions
on these matters. We tend to think that groups and topologies are
importantly distinct and so it is clear that straw equivalence does
not detect this distinction. It's not, however, difficult to place
at least some of the blame: we should be focused on comparing structures,
not universes. In the context of first order logic, we were able to
compare models using definability over those models. Now the means
of definability and their targets come apart. While we want the full
resources of the background universe available to define a new structure,
the targets of our comparison are the structures, not the universes.
This will motivate our next, slightly less bad, definition. 

\subsection{Sticks equivalence}

We'll now give another flawed characterization of interdefinability.
However, this one will be sufficiently improved that it will allow
us to draw out a pair of deeper problems that we'll discuss in the
next section. Recall that this time our goal is to compare structures
rather than universes while exploiting the full definability resources
of the background universe. Given that we started with straw, let
us follow the story of the the \emph{Three Little Pigs} for this next
iteration. 
\begin{defn}
Say that theory $T$ and $S$ (construed as class of sets) are \emph{sticks
equivalent} if there are class functions $t$ and $s$ such that:\label{def:Sticks}
\begin{enumerate}
\item $\mathcal{A}=s\circ t(\mathcal{A})$ for all elements $\mathcal{A}\in T$;
and
\item $\mathcal{B}=t\circ s(\mathcal{B})$ for all elements $\mathcal{B}\in S$. 
\end{enumerate}
\end{defn}

The key difference is that we are now concerned with whether our functions
return us to the same structures (as set) rather than universe. To
see how this definition performs let's try it out on a pair of different
theories of topology: $Top$ and $Nei$. We'll return to these theories
repeatedly throughout this paper so it will be worthwhile giving proper
definitions of them. The first is the familiar definition of topology
favored in textbooks today. 
\begin{defn}
We let $Top$ be the theory (as class) of sets of the form $\langle X,\mathcal{T}\rangle$
where $\mathcal{T}\subseteq\mathcal{P}(X)$ is such that: $\emptyset,X\in\mathcal{T}$;
if $X,Y\in\mathcal{T}$, then $X\cap Y\in\mathcal{T}$; and if $\mathcal{Z}\subseteq\mathcal{T}$,
then $\bigcup\mathcal{Z}\in\mathcal{T}$.\label{def:Top}
\end{defn}

The second definition is based on a neighborhood function which is
intended to take points from the domain and return the set of neighborhoods
containing that point. 
\begin{defn}
We let $Nei$ be the theory (as class) of set of the form $\langle X,\mathcal{N}\rangle$
where $\mathcal{N}:X\to\mathcal{P}\mathcal{P}(X)$ such that: if $Z\in\mathcal{N}(y)$,
$y\in Z$; if $Y\in\mathcal{N}(z)$ and $Y\subseteq W\subseteq X$,
then $W\in\mathcal{N}(z)$; if $Y,Z\in\mathcal{N}(w)$ then $Y\cap Z\in\mathcal{N}(w)$;
and if $y\in X$, there is some $Z\in\mathcal{N}(y)$ such that for
all $w\in Z$, $Z\in\mathcal{N}(w)$.\label{def:Nei}
\end{defn}

Note that while neither of these are theories in the sense of first
order logic, it is easy to define these classes -- from the outside
-- using the resources of set theory. We can the put the equivalence
relation above to work as follows:
\begin{prop}
$Top$ and $Nei$ are sticks equivalent.\label{prop:TopNei}
\end{prop}

\begin{proof}
(\emph{Sketch only}) We need to define class functions $t:Top\leftrightarrow Nei:s$
meeting the requirements of Definition \ref{def:Sticks}. Given $\langle X,\mathcal{T}\rangle\in Top$,
we want to define $t(\langle X,\mathcal{T}\rangle)=\langle X,\mathcal{N}\rangle$
where $\mathcal{N}$ is the neighborhood function corresponding to
$\mathcal{T}$. For this purpose, we let $\mathcal{N}$ be such that
for all $y\in X$
\[
\mathcal{N}(y)=\{Z\subseteq X\ |\ \exists W\in\mathcal{T}\ y\in W\subseteq Z\}.
\]
In other words, we let the neighborhoods of $y$ be those subsets
of $X$ that extend an open set containing $y$. In the other direction,
we start with $\langle X,\mathcal{N}\rangle$ and aim to define $s(\langle X,\mathcal{N}\rangle)=\langle X,\mathcal{T}\rangle$
where $\mathcal{T}$ is the topology corresponding to $\mathcal{N}$.
For this purpose, we exploit the final clause in \ref{def:Nei} which
says that every point has a neighborhood that is a neighborhood of
all its points; i.e., an open set. Thus we let 
\[
\mathcal{T}=\{Y\subseteq X\ |\ \forall z\in Y\ Y\in\mathcal{N}(z)\}.
\]
We leave it to the reader to verify that $t$ and $s$ witness clauses
(1) and (2) of Definition \ref{def:Sticks}.
\end{proof}
A point in favor of sticks equivalence is that it detects the intuitive
equivalence of $Top$ and $Nei$. Moreover, the proof strategy lines
up with what we'd expect from an informal equivalence claim made in
a textbook or classroom. I submit that the argument strategy above
is what people generally have in mind when they talk about interdefinability
in these cases. 

Let us now try a different pair of theories that are often regarded
as equivalent: Boolean algebras and Stone spaces. As with $Top$ and
$Nei$, we shall revisit these theories again in this paper so it
will pay to briefly describe them here. 
\begin{defn}
Let $Bool$ be the class of sets of the form $\langle B,\wedge,\vee,\neg,\top,\bot\rangle$
that deliver a uniquely complemented, distributive lattice.\footnote{See \citep{Givant2009} for a comprehensive introduction to Boolean
algebras. This definition comes from \citep{Bell1969}.}
\end{defn}

\begin{defn}
Let $Stone$ be the theory of structure of the form $\langle X,\mathcal{T}\rangle\in Top$
that are compact, Hausdorff and totally disconnected.\footnote{See Section 3.6 and, in particular, Definition 3.6.32 in \citep{HalvLogPhilSci}
for more detailed information.}
\end{defn}

The standard argument for the equivalence of these theories is known
as the Stone duality theorem. However, this is not sufficient to give
us sticks equivalence. The functions we use to define one structure
in terms of the other do not takes us back-and-forth and return us
to the same structure. Rather, they merely return us to a structure
that is isomorphic to the one we started with. With this in mind,
let us say that theories (as classes) are \emph{almost sticks equivalent}
if we revise clauses (1) and (2) in Definition \ref{def:Sticks} so
that $\mathcal{A}\cong s\circ t(\mathcal{A})$ and $\mathcal{B}\cong t\circ s(\mathcal{B})$.
Since we are using sets to represent structures, the natural notion
of isomorphism is just bijection. However, this is too weak for our
current purposes. As such, we shall also require that the intended
notion of isomorphism between structures is built in (by hand) into
the statement that two theories are almost sticks equivalent.\footnote{This is obviously a weakness of this characterization, but we will
defer addressing it until we get to our main proposal.} This is easier to explain by just stating our rough version of what
is know as the Stone duality.
\begin{thm}
$Bool$ and $Stone$ are almost sticks equivalent, where isomorphism
in $Bool$ is algebraic isomorphism and isomorphism in $Stone$ is
homeomorphism.\label{thm:BoolStone}
\end{thm}

\begin{proof}
(\emph{Sketch only}) We want to define functions $t:Bool\leftrightarrow Stone:s$
witnessing almost sticks equivalence. Let's define $t(\mathbb{B})=\langle X,\mathcal{T}\rangle$
for $\mathbb{B}\in Bool$ first. We start by letting $X$ be the set
of ultrafilters on $\mathbb{B}$. And then we define the natural topology
$\mathcal{T}$ on $X$ by using the sets 
\[
\{U\in X\ |\ b\in U\}
\]
for $b\in B$ be the basic closed sets from which we may generate
$\mathcal{T}$. In the other direction, let $s(\langle X,\mathcal{T}\rangle)=\langle B,\wedge,\vee,\neg,\top,\bot\rangle$
for $\langle X,\mathcal{T}\rangle\in Stone$ be formed by letting
$B$ be the set of clopen elements of $\mathcal{T}$ and letting $\wedge,\vee,\neg,\top$
and $\bot$ be $\cap,\cup,\cdot^{c},X$ and $\emptyset$ respectively.
Once again, we leave it to the reader to fill in the gaps.\footnote{See Section 3.7 \citep{HalvLogPhilSci} for a particularly elegant
and patient delivery of the proof of the Stone duality theorem. }
\end{proof}
While the property is weaker than the one we obtained in Proposition
\ref{prop:TopNei}, we still have another interdefinability result
that follows the standard proof of equivalence that we find in textbooks.
Moreover, the proof sketch above is also usually used to show that
natural theory category associated with $Bool$ and the opposite category
for $Stone$ are equivalent as categories.\footnote{The category for $Bool$ uses homomorphisms as arrows and the category
for $Stone$ use continuous maps. The opposite category is obtained
by reversing the arrows in a category. In other words, $Bool$ and
$Stone$ are duals as categories.} Thus, we have a pleasing link between sticks equivalence and category
theory. Later, when we have our intended proposal on the table, we
shall see that there are many more results like this one. 

Taking a little stock, we see that sticks equivalence seems to perform
quite well when it comes to replicating naive interdefinability arguments.
However, we said right at the beginning that this would be the case.
The hard part is proving inequivalences. And more than that, proving
inequivalences that line up passably well with our intuitions on these
matters. This is the subject of the next section. 

\subsection{Triviality}

Before we start picking on sticks equivalence, let's first add another
point its favor. In particular, let us show that sticks equivalence
isn't entirely trivial; i.e., there are pairs of theories that are
not sticks equivalent. To this end, let $Nat$ be the set of natural
numbers $\mathbb{N}$ and let $Real$ be the set of real numbers $\mathbb{R}$.
According to our crude definition, $Nat$ and $Real$ are theories
since they are classes of sets. Moreover, we see that: 
\begin{prop}
$Nat$ and $Real$ are not sticks equivalent. 
\end{prop}

\begin{proof}
If there were functions $t:Nat\leftrightarrow Real:s$ witnessing
definitional equivalence, then $t$ would be a bijection from $\mathbb{N}$
onto $\mathbb{R}$, which is is impossible. 
\end{proof}
Aside from the fact that $Nat$ and $Real$ are quite unnatural as
theories, we see that the proof of their inequivalence is very crude:
it rest on a cardinality fact. We'd also like a tool that can distinguish
theories that have the same cardinality of structures. Moreover, it
doesn't seem unreasonable to expect that definability considerations
could achieve this. This brings us to the first of two intuitive triviality
problems. 

\subsubsection{When $V=HOD$\label{subsec:WhenV=00003DHOD}}

We are going to show that if a well-known axiom is added to $ZFC$,
sticks equivalence becomes all but trivial. First, we introduce the
axiom. Recall that in the context of $ZFC$, the statement $V=HOD$
says that every sets is hereditarily ordinal definable. A set $x$
is \emph{ordinal} \emph{definable} if there is some finite sequence
$\alpha_{0},...,\alpha_{n}$ of ordinals and a formula $\varphi(y,\alpha_{0},...,\alpha_{n})$
of $\mathcal{L}_{\in}$ such that 
\[
\varphi(y,\alpha_{0},...,\alpha_{n})\leftrightarrow y=x.
\]
A set $x$ is \emph{hereditarily} \emph{ordinal} \emph{definable}
if it and every set in its transitive closure is ordinal definable.\footnote{See Chapter V of \citep{KunenST} or Chapter 13 of \citep{JechST}
for more details.} However, for our purposes the following observation is the important
thing to note. In the context of $ZFC$, $V=HOD$ is true iff there
is a definable well-ordering of the universe; i.e., there is a formula
defining a relation $\prec$ that is linearly ordered and such that
every set has a $\prec$-least element \citep{MyhillScott}. Moreover,
the order type of this well-ordering is $Ord$.\footnote{This is somewhat helpful since, for example, it allows us to speak
of the $\alpha^{th}$ element of the well-ordering. Recall that is
trivial to define well-orderings longer than $Ord$. For example,
we might switch $0$ from being the least ordinal to being greater
than all other ordinals giving an ordering we might naturally denote
as $Ord+1$. } In the context of $V=HOD$, we then see: 
\begin{prop}
If $V=HOD$, then every pair of theories (construed as classes) that
have the same cardinality\footnote{I'm abusing terminology here by thinking of proper classes as having
cardinality $Ord$, when strictly, a cardinal should be a set. } are sticks equivalent.\label{prop:HODsticks}
\end{prop}

\begin{proof}
Let $T$ and $S$ be theories with the same cardinality. We define
class functions $t:T\leftrightarrow S:s$ as follows. First fix a
definable well-ordering $\prec$ of the universe. For technical reasons,
we break the proof into two cases according to whether $T$ and $S$
are proper classes or sets. Next, suppose first that $T$ and $S$
are proper classes. For $\mathcal{A}\in T$, we fix $\alpha$ such
that $\mathcal{A}$ is the $\alpha^{th}$ element of $T$ in the $\prec$-order.
We then let $t(\mathcal{A})$ be the $\alpha^{th}$ element of $S$
in the order. Then let $s:S\to T$ be $t^{-1}$.  Suppose $T$ and
$S$ are sets with the same cardinality. Let $\sigma_{t}$ be the
$\prec$-least bijection from $T$ onto its cardinality; and let $\sigma_{s}$
be the analogous bijection for $S$. We then let $t=\sigma_{s}^{-1}\circ\sigma_{t}$
and let $s=t^{-1}$. 
\end{proof}
Recalling that there is a proper class of groups and a proper class
of topologies, we see that the theory $Group$ of groups and $Top$
are sticks equivalent if $V=HOD$. We thought it was bad when they
were straw equivalent and it's still seems wrong to see that they
are sticks equivalent. Of course, in this case we have an assumption
to blame: $V=HOD$. But even if we set that aside, we are still left
with a problem. 
\begin{prop}
If $ZFC$ proves that some pair of theories have the same cardinality,
then it cannot prove that they are not sticks equivalent; assuming
that $ZFC$ is consistent.
\end{prop}

\begin{proof}
Let $T$ and $S$ be definable classes that $ZFC$ proves have the
same cardinality. Let $M$ be model satisfying $ZFC$ and $V=HOD$;
for example, start with a model of $ZFC$ and add a Cohen real. Then
in $M$, $T^{M}$ and $T^{S}$ have the same cardinality and are thus,
sticks equivalent.
\end{proof}
While we've take a step back from the brink of triviality, we see
that even without assuming $V=HOD$, something undesirable is occurring.
As we noted at the beginning of this paper, the hard task will be
proving that theories are inequivalent. With sticks equivalence it
will be almost impossible to do so. 

But beyond these mathematical difficulties, the proof of Proposition
\ref{prop:HODsticks} highlights something odd about the relationship
established by sticks equivalence in the context of $V=HOD$. To see
this, I'll start by claiming there is a common thought that when we
define one structure in terms of another, we are somehow \emph{transforming}
the objects of one theory into that of the other. We should be using
features of the structures of one theory in order to obtain the features
of the structures in the other. But the proof above doesn't live up
to this intuition at all. Rather, we might say that our assumption
that $V=HOD$ gives us a \emph{universal} \emph{lookup} \emph{table
}that assigns every structure in every theory a position in an ordering.
This can then be used to send structures in one theory to those in
another. But -- at least intuitively -- this process seems to make
no use of the particular features of those structures beyond their
externally assigned position. 

I realize that these observations are vague and imprecise, but nonetheless,
I still suspect that they touch on something close to our reasons
for thinking that something undesirable is occurring. Moreover, the
failure to exploit structural features may provide some explanation
at to why we come so close to triviality. As a spoiler, I will say
that this is a problem that our preferred solution doesn't entirely
avoid either, although we will be able to mitigate its effects using
forcing. A part of the problem is that we are trying to entertain
an extremely generous notion of definability that exploits the full
power of our background mathematics. This means that definability
will be sensitive not just to features of the structures we are considering
but to the universe itself. While sticks equivalence is exceedingly
simplistic, this simplicity makes it easy to illustrate such problems.
Moreover, this highlights the most common challenge for our project:
avoiding triviality.

\subsubsection{Counterintuitive Cantor-Bernstein\label{subsec:Counterintuitive-Cantor-Bernstei}}

Our next problem is more damaging for the framework above. However
once exposited, it will also provide a guiding light toward the our
preferred approach to theory comparison. First, recall that the Cantor-Bernstein
theorem tells us that whenever we have injections $f:A\to B$ and
$g:B\to A$, then there is a bijection $h:A\to B$. Moreover, the
delivery of $h$ is quite constructive; for example, the Axiom of
Choice is not required. With a little work, this result can then be
generalized to the world of definable classes and functions. 
\begin{thm}
If $T$ and $S$ are theories (as definable classes) and $t:T\to S$
and $s:S\to T$ are definable injections, then there is a definable
bijection $u:T\to S$.\label{Cantor-Bernstien}
\end{thm}

\begin{proof}
(\emph{Sketch only}\footnote{A full version of this argument in the context of sets rather than
classes can be found in Theorem 1.4 in \citep{schindler2014set}.}) Given some $\mathcal{A}_{0}\in T$ we ask if there is some $\mathcal{B}_{0}\in S$
such that $s(\mathcal{B}_{0})=\mathcal{A}_{0}$. If so, we then ask
if there is some $\mathcal{A}_{1}\in T$ such that $t(\mathcal{A}_{1})=\mathcal{B}_{0}$.
If so, we ask if there is some $\mathcal{B}_{1}\in S$ such that $s(\mathcal{B}_{1})=\mathcal{A}_{1}$.
We then repeat this process going back into the history of $\mathcal{A}_{0}$
via $s$ and $t$. 

Three things can happen. First, we might find some $\mathcal{A}_{n}$
for which there is no $\mathcal{B}_{n}$ such that $s(\mathcal{B}_{n})=\mathcal{A}_{n}$.
In this case, we let $u(\mathcal{A}_{0})=t(\mathcal{A}_{0})$. Second,
we might find some $\mathcal{B}_{n}$ for which there is no $\mathcal{A}_{n+1}$
such that $t(\mathcal{A}_{n+1})=\mathcal{B}_{n}$. In this case, let
$u(\mathcal{A}_{0})=s^{-1}(\mathcal{A}_{0})$. And finally, the process
might not stop. In this case, our choice doesn't matter, so let $u(\mathcal{A}_{0})=t(\mathcal{A}_{0})$. 
\end{proof}
For our purposes, the upshot of this is that a sufficient condition
for the sticks equivalence of $T$ and $S$ is the provision of definable
injections between them. The bijection then comes along for free.
While this sounds helpful, the proof above also draws out something
odd and perhaps counterintuitive. Although choice and well-ordering
are absent, we still seem distant from our intuition about transformation
of structures. We aren't just looking at features of the structures
themselves, but also history of the ways in which those structures
may have already been defined. We shall also see something similar
to this phenomenon with our preferred solution. However, in the context
of sticks equivalence there are more serious problems. 

Let $Set1$ be the theory consisting of set containing exactly one
object. Let $Set2$ be the class of sets that contain exactly two
objects. Intuitively, these appear to be quite different theories.
I don't think we would think of them as being interdefinable. Nonetheless,
we have the following: 
\begin{prop}
$Set1$ and $Set2$ are sticks equivalent.\label{prop:Set1Set2}
\end{prop}

\begin{proof}
By Theorem \ref{Cantor-Bernstien}, it will suffice to show that there
are definable injections $t:Set1\leftrightarrow Set2:s$. Given $\{a\}\in Set1$,
we let $t(\{a\})=\{a,\{a\}\}\in Set2$. This is clearly an injection.
Given $\{b,c\}\in Set2$, we let $s(\{b,c\})=\{\{b,c\}\}\in Set1$.
This is also clearly injective. This is all we need.
\end{proof}
The first thing to note is that impoverished representation of mathematical
structures as sets is not to blame here. $Set1$ and $Set2$ are theories
whose structures are, so to speak, empty. There are no functions,
relations, topologies etc. We just have raw, unadulterated domains.
Moreover, we might regard these theories as categorical in the sense
that each theory only possesses one structure up to isomorphism, which
in this case, is just bijection. 

So what is going wrong? While it is not visible in the proof of Proposition
\ref{prop:Set1Set2}, the appeal to Theorem \ref{Cantor-Bernstien}
does something quite strange. Rather than just taking an element $\mathcal{A}$
of $Set1$ as a set with no distinguishing features other that its
being a singleton, we also look at the way in which $\mathcal{A}$
is built from other sets in the universe. In particular, we ask whether
$\mathcal{A}$ is a singleton of a pair $\mathcal{B}$ in $Set2$.
And we then ask whether $\mathcal{B}$ is such that one element of
that pair is the unique member of the other. And so on. We might say
that that we are making use of information that shouldn't be available;
that goes beyond the ``structural information'' contained in $\mathcal{A}$.
We might say that allowing access to such information violates an
obvious norm of a structuralist approach to mathematics: 
\begin{quote}
We should be able to freely modify the domains of structures without
changing the properties of those structures that we care about. 
\end{quote}
To see how this usually works in the context of definability, consider
a model $\mathcal{M}=\langle M,\sigma^{\mathcal{M}}\rangle$ of some
first order language. Our structuralist ideal is upheld by the fact
that we can replace the domain $M$ of $\mathcal{M}$ with any set
of the same cardinality to obtain a model $\mathcal{N}$ that is isomorphic
to $\mathcal{M}$. Moreover, for any structure $\mathcal{A}$ that
can be defined \emph{in $\mathcal{M}$ }through an interpretation,
there will be a corresponding structure $\mathcal{B}$ that is defined
over $\mathcal{N}$ by the same definition that isomorphic to $\mathcal{A}$.
Changing the domains of models of first order logic has no effect
on the properties we care about. 

The case of $Set1$ and $Set2$ is very different. Here the domains
of structures in $Set1$ or $Set2$ are just the structures (as sets)
themselves. Now consider $\mathcal{B}=\{b,\{b\}\}$ from $Set2$.
We see that $\mathcal{B}=s(\{b\})$ where $\{b\}\in Set1$. But suppose
we change the domain of $\mathcal{B}$ to form $\mathcal{B}^{*}=\{a,\{b\}\}$
by swapping out $b$ for some set $a\neq b$. Then $\mathcal{B}^{*}\neq s(\mathcal{A})$
for any structure $\mathcal{A}$ in $Set1$. This means that the function
$u:T\to S$ given by Theorem \ref{Cantor-Bernstien}, may do something
quite different with $\mathcal{B}$ and $\mathcal{B}^{*}$. And it
will do this because $\mathcal{B}^{*}$ is obtained by modifying the
domain of $\mathcal{B}$. Thus, our structuralist norm has been violated. 

Of course, this problem also emerges because our resources for defining
mathematical structures are so powerful. We aren't merely considering
what can be defined over some model of first order logic, we are using
the full resources of our background mathematics. Given that our goal
was to extend our resources in this way, we might wonder if this is
just a side-effect of our generosity. While I think there is something
to that thought, there are also some drawbacks. We agreed to use $ZFC$
above since it provides a well-understood foundation for mathematics,
but we also wanted to remain open to the use of other foundational
frameworks, like those in category theory and type theory. However,
the argument above is tied very closely to a particular set theoretic
perspective in which all sets appear in at a certain point in a cumulative
hierarchy. This perspective is very useful and, indeed, we'll use
it below. However, given that our goal is to provide a model of the
practice of working mathematicians and physicists, it seems preferable
to minimize the effect of such interventions in our framework. Furthermore,
recall that our goal is to compare mathematical structures, not the
way in which those structures are situated within a set theoretic
foundation. As such, I think maintaining the structuralist norm above
is a good idea. We want a framework that has the full definability
resources of our background mathematics, but which only has access
to information that is structurally significant. This will be one
of the main challenge in the development of our preferred approach.

Before we outline our framework, let's reflect on what we've learned
from sticks equivalence. On the positive side, we saw that sticks
equivalence gave us a plausible way of modeling familiar interdefinability
arguments. Moreover, the proofs required no modification; they could
be torn straight out of the textbook. We then observe that sticks
equivalence isn't entirely trivial. In particular, it can distinguish
theories with different cardinalities of structures. But as we pushed
further in the effort to find further inequivalences, the negative
side of things emerged. We found that equivalences were too easy to
obtain and that many of them were counterintuitive. In particular,
we were concerned about our inability to prove sticks inequivalence
for any pair of theories with the same cardinality. Finally, we observed
that when we combined our means of representing structures with our
generous definability resources, we were able to access information
beyond what we thought legitimately belonged to those structures.
Our goal now is to develop a framework that preserves that benefits
of stick equivalence while addressing its defects. In a nutshell,
we want a framework that maximizes our definability resources while
restricting its access to nonstructural information.

\section{$V$-logic and the $\mathcal{V}^{*}$ framework\label{sec:V*}}

We are ready for the final pig. While the underlying idea is very
simple, we need to go through quite a lot of technical material in
order to properly define our framework. As such, it will be helpful
to give a simpler motivating example that will capture the main intuition
behind our proposal. In the last section, we stated our goal to characterize
a notion of definability that provided a good model of our background
mathematics while remaining faithful to our structuralist ideal. We
noted that while interpretations between theories in first order logic
give a relative weak notion of definability, they do respect our structuralist
ideal of invariance under domain changes. Given that sticks equivalence
employed greater definability powers but failed on the structuralist
score, we might wonder if some kind of trade-off is in play. It will
be helpful to scotch this concern by considering interpretability
in the context of second order logic. 

First we do a quick recall of some basic definitions and notation.
Suppose we have a language $\mathcal{L}$ consisting of function and
relation symbols. The language of second order logic based on $\mathcal{L}$
is obtained by expanding the first order language with variables $X_{0}^{n},X_{1}^{n}$
for each $n\in\omega$, where such variables are intended to range
of $n$-ary relations. Now let $\mathcal{M}=\langle M,\sigma\rangle$
be a (first order) model of $\mathcal{L}$ where $M$ is the domain
and $\sigma$ is the interpretation of $\mathcal{L}$. A \emph{full}
second order $\mathcal{L}$-model $\mathcal{M}^{+}$ can be obtained
from $\mathcal{M}$ by expanding it with new domains $P^{n}=\mathcal{P}(M^{n})$
for all $n\in\omega$. These are the domains over which the second
order quantifiers are intended range.\footnote{See \citep{Shapiro} or \citep{VaanSOST} for more details.}
The essential idea is that every possible relation on $M$ is available
to be quantified over. The expressive and definability powers of second
order logic are much greater than first order logic. For example,
we may formulate a second order version of Peano arithmetic, $PA^{2}$,
which is categorical in the sense that there is only one model of
$PA^{2}$ up to isomorphism. In contrast, we know that there are many
pairwise non-isomorphic models of first order $PA$. Moreover, within
such a second order model of $PA$, may define sets of natural numbers
that are not definable in first order $PA$.\footnote{See, for example, Theorem 4.8 in \citep{Shapiro}.}
For example, we can define the set of natural numbers coding computable
well-orderings. While first order $PA$ can define the natural numbers
coding computable relations, it cannot define those that deliver the
well-orderings even if we restrict our attention to the standard model
of arithmetic.\footnote{The set of naturals coding computable well-orderings is a $\Pi_{1}^{1}$-complete
set, while we can only define $\Sigma_{n}^{0}$ sets of the standard
model. See, for example, Theorem 4.9 in \citep{mansfield1EDST}.}

So second order logic is more powerful that first order logic. But
is definability in second order logic compatible with out structuralist
ideal? It turns out that it is. Moreover, the theory of relative interpretation
along with standard notions like definitional equivalence and bi-interpretability
transfer directly over to the context of second order logic without
the need for modification. As such, our problem with $Set1$ and $Set2$
will not occur in this context. To see this, we'll start by giving
a gentle proof of a simple folk proposition that nicely encapsulates
a technique we'll continue to exploit and generalize throughout this
paper.
\begin{prop}
Let $\mathcal{M}=\langle M\rangle$ be a full second order model of
the empty theory in the empty language (so we just have the identity
relation). Then the only subsets of $M$ that are definable over $M$
are $M$ itself and $\emptyset$.\label{prop:protoAuto} 
\end{prop}

\begin{proof}
Suppose toward a contradiction that $\mathcal{M}$ is a such a model
and that it can define a set other than $M$ or $\emptyset$. Then
clearly, $|M|\geq2$. Otherwise, $M$ and $\emptyset$ would be the
only subsets of the domain. Now let $P\subseteq M$ be such that $P\neq\emptyset$
and $P\neq M$ where $P$ is definable by some formula $\varphi_{P}(x)$
in the language of second order logic. Thus, we have 
\[
x\in P\ \Leftrightarrow\ \mathcal{M}\models\varphi_{P}(x).
\]
Now we may fix $m_{0},m_{1}\in M$ with $m_{0}\in P$ and $m_{1}\notin P$.
Then let $\sigma:M\to M$ be the permutation that switches $m_{0}$
and $m_{1}$ and leaves everything else alone. Now we come to the
crucial point. Noting that $\sigma$ is an isomorphism on $\mathcal{M}$,
it can be seen by induction on the complexity of formulae that for
all $n\in M$, 
\[
\mathcal{M}\models\varphi_{P}(n)\ \Leftrightarrow\ \mathcal{M}\models\varphi_{P}(\sigma n).
\]
Thus, we see that 
\begin{align*}
m_{0}\in P\Leftrightarrow\mathcal{M}\models\varphi_{P}(m_{0}) & \Leftrightarrow\mathcal{M}\models\varphi_{P}(\sigma m_{0})\\
 & \Leftrightarrow\mathcal{M}\models\varphi_{P}(m_{1})\Leftrightarrow m_{1}\in P
\end{align*}
which is a contradiction. 
\end{proof}
Using Proposition \ref{prop:protoAuto}, it is easy to see that $Set1$
and $Set2$ are not sticks equivalent. We note first that $Set1$
and $Set2$ are easily axiomatized in second order logic, and indeed,
first order logic. Indeed, $Set2$ cannot even interpret $Set1$.
By this we mean that there is no formula in second order logic such
that we can take an arbitrary model $\mathcal{M}=\langle M\rangle$
of $Set2$ and use that formula to define a model of $Set1$. Such
a formula would need to define a nonempty proper subset of $M$ which
is impossible. Note also the key move in the proof of Proposition
\ref{prop:protoAuto} is the use of automorphisms. In particular,
we are interested in the fact that nontrivial automorphisms, so to
speak, move objects while leaving definitions alone. We shall make
a lot of use of this idea below. 

Thus, we see that second order logic has greater expressive resources
than first order logic and it respects our structuralist ideal. One
might then wonder whether second order logic can deliver the framework
we are looking for. Perhaps second order logic can be used to characterize
theories in physics and offer a powerful, but natural notion of interdefinability.
There is something to this idea, but we are going to avoid this for
two reasons. First and perhaps controversially, we don't think second
order logic is sufficient to live up our goal of representing the
full definability capacities of our background mathematics. While
it defines a lot more than one might naively expect, definability
in second order logic is still clearly circumscribed within the background
framework of $ZFC$.\footnote{For example, for any well-ordering of type $\alpha$ that is definable
in second order logic, there is a sentence $\varphi_{\alpha}$ of
second order logic that is only satisfied in models $\mathcal{M}$
that are isomorphic to $V_{\alpha}$. See \citep{VaanSOST} for discussion
of this.} Second and perhaps more seriously, while second order logic is certainly
powerful, it is quite fussy to use. In order to represent physical
theories in second order logic, we need to translate the naive language
of mathematics in which they are described into the more constrained
syntax of second order logic. As such, the equivalence results we
highlighted above cannot simply be torn from the book in this context. 

If second order logic isn't enough, where should we turn? We might
generalize and consider third order logic, fourth order logic and
so on. Perhaps we might go to $\omega$-order logic, which can be
thought of as a simple type theory.\footnote{This is close to the framework proposed in \citep{Hudetz2017}.}
But the definability capacities of these options are still comfortably
circumscribed within $ZFC$. We'd like to go all the way, but what
does that mean? If we reflect -- a little crudely -- on second order
logic, we might think of of a full model of second order logic as
adding the powerset of the domain to its range of quantification.
Similarly, third order logic adds the powerset of that new domain
to its range of quantification. Simple type theory does this $\omega$-many
times. But why stop there? Why not take the union of the those $\omega$-many
domain and then take the powerset of that? Why not keep going? This
is essentially the process of generating the cumulative hierarchy
of sets, except that we want to start with a particular structure
rather than the empty set. This is the motivating idea behind our
framework. 
\begin{quote}
We want to define the universe of sets relative to a particular structure. 
\end{quote}
This will give us a notion of definability which is strong enough
to plausibly represent all of our background mathematics. Moreover,
if we define it properly, this notion of definability should also
respect our structuralist ideal since the universe is build built
\emph{on top of }a structure and so -- as with our second order logic
example -- the problems with $Set1$ and $Set2$ will be avoided. 

Now that we've described the idea behind our framework, let's now
consider a final wrong turn. Understanding this problem will give
a much clearer idea why our framework is defined in the way that it
is. Suppose we have a first order model $\mathcal{M}=\langle M,\sigma\rangle$
and we want to build the universe of sets above it. Here is a natural,
but misguided, way of doing this. By recursion, let
\begin{align*}
V_{0}^{\dagger}(M) & =M\\
V_{\alpha+1}^{\dagger}(M) & =\mathcal{P}(V_{\alpha}^{\dagger}(M))\cup V_{\alpha}^{\dagger}(M)\\
V_{\lambda}^{\dagger}(M) & =\bigcup_{\alpha<\lambda}V_{\alpha}^{\dagger}(M)\text{ for limit ordinals }\lambda
\end{align*}
and $V^{\dagger}(M)=\bigcup_{\alpha\in Ord}V_{\alpha}^{\dagger}(M)$.
So the idea is that we start with $M$ as $V_{0}^{\dagger}(M)$ at
the bottom. At successor levels we add the powerset of the previous
level to what we had before. And at limit ordinals, we collect up
what we already have formed. It's essentially the definition of $V$
except that we start with $M$ rather than the empty set. Thus, it
seems like it lines up nicely with our goals. Unfortunately, there
is problem. 
\begin{prop}
For all sets $M$, $V^{\dagger}(M)=V$. 
\end{prop}

\begin{proof}
Recall that $\emptyset\in\mathcal{P}(x)$ for any set $x$. Thus,
$\emptyset=V_{0}\in V_{1}^{\dagger}(M)$. Building on that observation,
it can then be seen using induction that $V\subseteq V^{\dagger}(M)$.
And since it's trivially true that $V^{\dagger}(M)\subseteq V$, we
are done.
\end{proof}
Thus, if we do the obvious thing we end up recreating the universe.
Why is this a problem? If we want to uphold our structuralist ideal,
we want to block access to information about the particular sets that
compose the set $M$. If we don't hide this information, then we'll
end up with the problem we had with $Set1$ and $Set2$. As we've
seen this isn't a problem for first order or second order logic, or
indeed, type theory. The recreating problem seems to be the result
of our lofty goals regarding definability. If we want to have the
definability powers of all of our background mathematics, then the
obvious way of modeling this in $ZFC$ gives us the entire universe
and access to too much information about structures. 

Fortunately, there is a standard way around this problem. If we want
to avoid information about the particular sets that compose the domain
of a structure -- like we did with first and second order logic --
then we could treat the domain as set of \emph{atoms} (or urelements).
An atom is an object that is not a set: while it can be a member of
a set, it has no members itself. So like an atomic proposition or
an atomic element in a Boolean algebra, there is a sense in which
a set-theoretic atom cannot be further analyzed. Of course (with the
possible exception of the emptyset) our preferred set theory $ZFC$
does not admit atoms. To get around this we shall make use of the
well-known variation, known as $ZFCU$ or $ZFCA$, $ZFC$ with atoms,
in order to define our framework.\footnote{We'll follow the axiomatization given in \citep{jechAC}. More information
about this and related systems can be found in \citep{Barwise,YaoZFCUforcing}
and \citep{yao2023}. In this paper, I'm going to adopt the (arguably
awkward) halfway house convention of referring to the non-sets as
atoms rather than urelements, but use the notation $ZFCU$ rather
than $ZFCA$ to denote our theory of these objects. I'll try to explain
why I'm doing this, but I'd like to express my gratitude to Bokai
Yao for arguing so passionately about this issue. As I understand
it, researchers involved in set theory are currently divided over
whether to talk about $ZFCU$ or $ZFCA$. Philosophers who work with
set theory tend to talk about urelements and $ZFCU$ \citep{McGeeST,Menzel2014};
while mathematicians tend to talk about atoms and $ZFCA$ \citep{zapletal2025,blass2025,HowardTachtsis2013,Hall2002}.
The philosophers I have spoken to tend to offer arguments and reasons
for their choice of terminology, while the mathematicians seem to
be following a default convention. Given that philosophers are prone
to offering reasons and arguments for just about everything, I'm not
sure that this tells us much. So let's quickly consider the three
main arguments I've seen. First, the term ``atom'' is wildly overused
in mathematics, logic, science and philosophy: Boolean algebras have
atoms; so does logic; so does physics; and also mereology to isolate
just a few places. It would be good to have a name with less risk
of confusion and ``urelement'' delivers on this. Second, there is
a risk that $ZFCA$ will be confused with the set theory based on
Aczel's anti-foundation axiom \citep{AczelNWFS}. And finally, $ZFCU$
lines up better with generalizations of Quine's $NF$ to obtain a
theory $NFU$ that accommodates urelements and atoms \citep{Jensen1969}.
I think the first of these reasons is the strongest. We might also
consider the respective etymologies of ``atom'' and ``urelement.''
On this axis, they appear to be about even: an atom is an \emph{indivisible}
entity; while an urelement is a \emph{primitive }element. Finally,
we might consider the history of these terms. While non-sets were
present in Zermelo's theory as delivered in \citep{Zermelo1908},
the first use of the term seems to be in \citep{Zermelo1930}. I am
less sure when ``atom'' enters the lexicon although it is obviously
the choice made in \citep{jechAC}. The terminology seems to have
been quite fluid in middle of the twentieth century with some authors
even referring to ``individuals'' \citep{Mostowski1945}. But given
the importance of \citep{jechAC} in the study of the axiom of choice
one might naturally speculate that this led the default convention
that is followed by mathematicians today. So where does this leave
us? Philosophers have a plausible argument for their choice and they're
often quite vehement about it. On the other hand, mathematicians just
seem to be following a familiar linguistic convention. Given this,
I have decided to defer to the passionate and refer to $ZFCU$ rather
than $ZFCA$. However, in ongoing conversations with people about
some proofs I have discovered that I just keep talking about ``atoms''
rather than ``urelements,'' so I will continue to (unofficially)
refer to them in this way below.} Our goal will be to use $ZFCU$ to build universes of sets over mathematical
structures while blocking access to information about how the domains
of those structures are constituted. In a nutshell, we are going to
work in $Ord$-order logic, or what we might call $V$-logic. 

Finally, we are ready to exposit our framework, which we shall do
in stages. We'll start by laying out the background theory and the
intuitive perspective one should take on it. Then we'll describe our
way of representing mathematical structures and theories of those
structures. With this in place, we can then deliver our a generalized
theory of interpretation which will give us the notion of definability
we require. Finally, we shall characterize versions of definitional
equivalence and bi-interpretability that are compatible with our framework
and which naturally generalize the corresponding concepts in first
order logic.

\subsection{Background Theory and Perspective}

Although we are going to make use of $ZFCU$, we are still going to
officially work in $ZFC$. This will be helpful for arguments later
on but it also allows us to work in a standard and well-understood
environment. To make the connection between $ZFC$ and $ZFCU$, we
shall work in $ZFC$ to define (via interpretation) a sufficiently
large playground from which models of $ZFCU$ can be generated over
any structure we care to consider. 

Essentially following \citep{Barwise}, given any set $X$ we define
a hierarchy of sets over $X$ as follows. 
\begin{align*}
V_{0}^{*}(X) & =\{\langle0,x\rangle\ |\ y\in X\}\\
V_{\alpha+1}^{*}(X) & =\{\langle1,Y\rangle\ |\ Y\subseteq V_{\alpha}^{*}(X)\}\cup V_{\alpha}^{*}(X)\\
V_{\lambda}^{*}(X) & =\bigcup_{\alpha<\lambda}V_{\alpha}^{*}(X)\text{ for limit ordinals }\lambda
\end{align*}
Then we let $V^{*}(X)=\bigcup_{\alpha\in Ord}V_{\alpha}^{*}$. Note
that this is quite different from our naive definition of $V^{\dagger}(X)$
above. We might say that the elements of $V^{*}(X)$ have extra information
appended to them. In particular, elements of $V^{*}(X)$ are of the
form $\langle i,y\rangle$ where $i\in\{0,1\}$. If $i=0$, then $\langle0,y\rangle$
is intended to be an atom; and if $i=1$, then $\langle1,y\rangle$
is intended to play the role of a set. Now we want $V^{*}(X)$ to
play the role of the universe of sets over $X$, so we also need a
membership relation. However, since the elements of $V^{*}(X)$ are
ordered pairs, we cannot use the membership relation from the background
universe: we must define it. This is easily done as follows: for $\langle i,y\rangle,\langle j,z\rangle\in V^{*}(X)$,
we let $\langle i,y\rangle\in^{*}\langle j,z\rangle$ iff $j=1$ and
$\langle i,y\rangle\in z$. 

To obtain a sufficiently large playground to work in, we then let
$V^{*}=\bigcup_{X\in V}V^{*}(X)$. Thus, we combine all the universes
built over sets into one big universe in which we can do all our work.
We then may define $\mathcal{V}^{*}=\langle V^{*},\in^{*},A\rangle$
where $A$ is the class of atoms; i.e., the class of those $\langle i,y\rangle\in V^{*}$
such that $i=0$. $\mathcal{V}^{*}$ will be the place in which we
develop our framework. To show that this does what we want, we now
want to show that if we work in $\mathcal{V}^{*}$ the recreation
problems we saw above are avoided. To see this let's work, so to speak,
inside $\mathcal{V}^{*}$. We shall use the language $\mathcal{L}_{\in}(At)$
which expands $\mathcal{L}_{\in}$ with a one-place relation symbol
$At$. Then we shall interpret $\in$ as meaning $\in^{*}$ and $At$
as meaning $A$. Now working in $\mathcal{V}^{*}$ suppose that $X$
is a set of atoms: i.e., every $y\in X$ is such that $At(x)$. We
wish to show that the hierarchy above $X$ is a model of $ZFCU$ where
$X$ forms the atoms of this hierarchy. This will show us that we've
avoided recreation since the elements of $X$ are treated as atoms.
But before we can do this, we need to be more specific about the theory,
$ZFCU$.

We work in the language $\mathcal{L}_{\in}(@)$ that expands $\mathcal{L}_{\in}$
with a constant symbol $@$ that is intended to denote the \emph{set}
of atoms.\footnote{Note this is different to the language we used to talk about $\mathcal{V}^{*}$since
there we use a relation symbol $At$ rather than a constant symbol
$@$. It is easy to see that in $\mathcal{V}^{*}$, the extension
$A$ of $At$ is not a set. There is no set $x\in\mathcal{V}^{*}$
such that for all $y$, $y\in^{*}x$ iff $At(y)$. This means that
$\mathcal{V}^{*}$ is not a model of the standard version of $ZFCU$.
It is, however, a model of a generalization of that theory which admits
the use of proper classes of atoms. See Section 1.2 of \citep{yao2023}
for further discussion.} We then expand our language with a a relation symbol $\mathfrak{S}$
whose extension is defined to be the complement of $@$; i.e., $\mathfrak{S}$
denotes the sets.\footnote{Although $\mathfrak{S}$ is a relation symbol, we shall frequently
abuse notation (and follow a common convention) by writing $x\in\mathfrak{S}$
rather than $\mathfrak{S}x$.} The axioms of $ZFCU$ are much the same as those of $ZFC$ except
that we need to take care of the distinction between sets. First,
we add a new axiom to describe the behavior of atoms. 
\begin{lyxlist}{00.00.0000}
\item [{(Atoms)}] $z\in@\to\forall y\ y\notin z$.
\end{lyxlist}
Thus, we ensure that atoms have no members. When formulating the rest
of the axioms, we need to ensure that we are talking about sets at
the right times. For example, 
\begin{lyxlist}{00.00.0000}
\item [{(Extensionality)}] $x,y\in\mathfrak{S}\wedge\forall z(z\in x\leftrightarrow z\in y)\to x=y$. 
\end{lyxlist}
To see the importance of the first clause, observe that without it,
it would apply to any pair $x,y$ of atoms. Since atoms have no members,
the axiom would tell us that $x=y$ and thus, there would be just
one atom, which would be identical to the empty set. The empty set,
$\emptyset$, itself can be defined in the usual way using Infinity
and Separation. Note, however, that saying $x\neq\emptyset$ does
not entail that $x$ is nonempty, since $x$ could be an atom. This
also affects the natural definition of subset. We shall say that $x$
is a subset of $y$, $x\subseteq y$, iff $x\in\mathfrak{S}$ and
for all $z\in x$, $x\in y$. Without the extra restriction, we end
up with slightly unpleasant feature that every atom is a subset of
every set.\footnote{We note, however, that this isn't a big problem. For example, \citeauthor{yao2023}
uses the ordinary definition of subset and just adds the restriction
to sets when required. For example see his version of the Powerset
axiom on page 4 of \citep{yao2023}.} With these traps for young players described, we leave it to the
reader to adapt the rest of the axioms, although we recommend the
use of \citep{Barwise,jechAC} and \citep{yao2023} for further reference.

Now we return to our goal of showing that $\mathcal{V}^{*}$ provides
a suitably large playground. In particular, we want to show that we
may generate a natural hierarchy of sets of any set of atoms from
$\mathcal{V}^{*}$. Let us now work within $\mathcal{V}^{*}$ and
take a set $X$ of atoms. We then use recursion to define the following
hierarchy where:
\begin{align*}
V_{0}(X) & =X\\
V_{\alpha+1}(X) & =\mathcal{P}(V_{\alpha}(X))\cup V_{\alpha}(X)\\
V_{\lambda}(X) & =\bigcup_{\alpha<\lambda}V_{\alpha}(X)\text{ for limit ordinals }\lambda.
\end{align*}
We then let $V(X)=\bigcup_{\alpha\in Ord}V_{\alpha}(X)$. We then
note that: 
\begin{prop}
($\mathcal{V}^{*}$)\footnote{We write ``($\mathcal{V}^{*}$)'' at the beginning of this proposition
to indicate that while we are using a background set theory like $ZFC$,
the proposition is about the$\mathcal{V}^{*}$-framework. Thus, to
bring things back to our background set theory, we should be understood
as using the interpretation above to translate the statement that
follows back into the ordinary language of set theory. } If $X$ is a set of atoms then: $\langle V(X),\in,X\rangle$ satisfies
$ZFCU$.\label{prop:V*buildsV(X)}
\end{prop}

Why is this important? First, we see that if we work within $\mathcal{V}^{*}$,
we avoid the recreation problem since $V(X)$ is clearly a proper
subclass of $V^{*}$. Second it gives us a way of taking an arbitrary
domain $X$ of atoms and building the full $Ord$-length hierarchy
over $X$. Finally, $\mathcal{V}^{*}$ does this in such a way that
the manner in which the particular atoms are composed is invisible.
Thus, we have a framework that has the potential to meet our goals. 

Before we move on, we note a special case of Proposition \ref{prop:V*buildsV(X)},
where $X$ is the empty set of atoms; i.e., the empty set. Inside
$\mathcal{V}^{*}$ this gives us $V(\emptyset)$, which can be seen
from the outside (in our background universe $V$ where we define
$\mathcal{V}^{*}$) to be isomorphic to the universe. We shall call
$V(\emptyset)$ the \emph{kernel} and when working inside $\mathcal{V}^{*}$
we shall denote it as $V$.\footnote{From the outside, where we define $\mathcal{V}^{*}$, we shall denote
the kernel as $V^{\mathcal{V}^{*}}$. The conceptual perspective,
not to mention the notation, can be a little confusing at first. However,
we will almost never need to pay much attention to these issues. The
exception to that rule of thumb occurs at the end of this paper in
Section \ref{subsec:What-can-we-do?} when we consider some extreme
limitation of our proposal.} Sets in the kernel are sometimes known as \emph{pure sets }since
they are not built upon any atoms. The kernel provides a fixed core
from which all of the usual objects we define in $ZFC$ can be found
within the context of $\mathcal{V}^{*}$. It will play an important
role later on. 

\subsection{Structured Sets and Theories}

We now have a way of defining a universe of sets over a particular
domain, but our main goal is to do this for an arbitrary mathematical
structure. In Section \ref{sec:1}, we used a particularly simplistic
representation of mathematical structure by just treating them as
sets. While this served us well enough for the purposes of illustrating
some naive problems, it is far too simple to give a good account of
interdefinability. A set is merely a domain upon which we may situate
structure. Our task now is to find a good technique for representing
structure on some domain. In first order logic, this was relatively
straightforward. For example, given a domain, $M$: constants are
elements of $M$; and relations are subsets of finite products of
$M$. But as we discussed earlier, not all mathematical structures
are naturally represented in this way. For example, if we consider
a topology $\langle X,\mathcal{T}\rangle$ we see that while $\mathcal{T}$
might resemble a one-place relation, $\mathcal{T}$ is a subset of
$\mathcal{P}(X)$ not $X$. Moreover, if all we have in addition to
the domain $X$ is $\mathcal{T}$ it is not so clear how we might
articulate axioms to ensure that $\mathcal{T}$ is a topology. 

While the framework of $\mathcal{V}^{*}$ comes with some conceptual
subtleties that can take a little while to get used to, it does allow
us to give an exceedingly general representation of a mathematical
structure that is quite easy to understand. First recall that the
transitive closure of a set $x$, $trcl(x)$, is the set of those
$y$ such that there exist $z_{0},...,z_{n}$ such that $y\in z_{0}\in...\in z_{n}\in x$.\footnote{See Section I.6 of \citep{Barwise} for a more nuanced and precise
discussion of transitive closure in the context of set theory with
atoms.}
\begin{defn}
($\mathcal{V}^{*}$) Say that $\mathcal{A}=\langle A,a\rangle$ is
a \emph{structured set }if:\label{def:StrucSet}
\begin{enumerate}
\item $A$ is a set of atoms; and
\item $a$ is a set such that $tr(\{a\})\subseteq A$.
\end{enumerate}
\end{defn}

The idea is that $A$ is the \emph{domain }and $a$ is the \emph{structure
}of $\mathcal{A}$. It's a very simple definition, yet surprisingly
powerful. With our representation of structure to hand, we can now
define the universe over a particular structured set. 
\begin{defn}
($\mathcal{V}^{*}$) Given a structure set, $\mathcal{A}$, let the
\emph{universe over $\mathcal{A}$, $V(\mathcal{A})$,} be such that
\[
V(\mathcal{A})=\langle V(A),\in,A,a\rangle.
\]
\end{defn}

Thus, we take the domain $A$ and then build the hierarchy $V(A)$
above it. By including the membership relation, domain and structure
in the universe we aim to give ourselves access to enough information
to develop a powerful notion of definability that matches our intuitions
about what it means to define one mathematical structure using another.
The following diagram, inspired by those in \citep{Barwise} and \citep{ershov1996definability},
gives a rough illustration of a structured set.

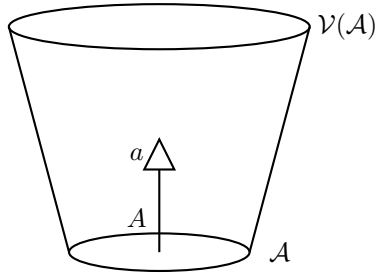
\begin{center}
\begin{tikzpicture}   
\draw[thick] (0,3) ellipse (2 and 0.3);   
\draw[thick] (0,0) ellipse (1.2 and 0.25);      

\draw[thick] (-2,3) -- (-1.2,0);   
\draw[thick] (2,3) -- (1.2,0);    

\draw[thick] (0,0) -- (0,1.1);      

\draw[thick] (0,1.5) -- (-0.2,1.1) -- (0.2,1.1) -- cycle;

\node at (2.5,3) {$\mathcal{V}(\mathcal{A})$};   
\node at (1.6,0) {$\mathcal{A}$};
\node at (-0.3,1.3) {$a$};  
\node at (-0.3,0.45) {$A$};  
\end{tikzpicture}
\end{center}

Our next task on the road to delivering this is to describe how we
can formulate \emph{theories} that will determine collections of structured
sets in a tractable manner. For this purpose, we describe a language
in which to articulate our theories. Since our representation of structures
is simple, so is the language. We let $\mathcal{L_{\in}}(D,d)$ be
the language expanding $\mathcal{L}_{\in}$ with two constant symbols,
$D$ and $d$. Their intended denotation will be the domain and structure
of a structured set respectively. Given that we are hoping for an
extremely general method, one might worry that we haven't included
any relation of function symbols. We'll assuage these concerns with
an example soon, but the essential reason we don't need relation or
function symbols is that we are working in a set theoretic environment
in which relations and functions can be represented by sets which
are naturally denoted by constant symbols. 

A theory in $\mathcal{L}_{\in}(D,d)$ will simply be a sentence $\varphi$
from this language. One might worry that a single sentence will not
suffice for this purpose. For example, when articulating foundational
theories of arithmetic, analysis or sets we usually make use of an
infinite but computable axiomatization. Again the power of our set
theoretic framework makes such devices redundant. This redundancy
already surfaces in the weaker context of second order logic, where
the computable schemata of $PA$ and $ZFC$ can be replaced by second
order counterparts, making their axiomatizations finite. 

Next we need to know when a particular structured set satisfies some
sentence $\varphi$ of $\mathcal{L}_{\in}(D,d)$. First let us recall
how to relativize a formula $\varphi$ to a particular universe $V(\mathcal{A})$
over some structured set $\mathcal{A}=\langle A,a\rangle$. We do
this in the usual way via a recursive translation that restricts quantification
to $V(A)$ and translates $D$ as $A$ and $d$ as $a$.\footnote{See Chapter IV.8 in \citep{KunenST} for more details. However, the
underlying idea is to just use a (very simple) interpretation as it
is usually understood in first order logic \citep{Visser2006}.}
\begin{defn}
($\mathcal{V}^{*}$) Given a structured set $\mathcal{A}$, we say
that $\varphi$ is \emph{satisfied} by $\mathcal{A}$ if the relativization
of $\varphi$ to $V(\mathcal{A})$, $\varphi^{V(\mathcal{A})}$ is
true.\footnote{I should not that while I'm using the word ``satisfied'' we are
not defining a satisfaction relation in the sense that is familiar
from model theory. It is just a translation and it is done in the
metalanguage. This issue is a result of our lofty goals with respect
to definability. If we want to have proper class sized universes,
then we cannot define a satisfaction relation for them in this context.
This follows from Tarski's theorem on the undefinability of truth.
If such a satisfaction relation were definable, we would be able to
define truth for the empty structure $V(\emptyset)$ in $\mathcal{V}^{*}$,
which would in turn, mean that we could define a truth predicate for
$V$ in $V$. }
\end{defn}

To see how this works, let's consider return to our topological example
from earlier. In the language $\mathcal{L}_{\in}(D,d)$ we may axiomatize
topology as follows. Let $Top$ be the sentences of $\mathcal{L}_{\in}(D,d)$
which is the conjunction of the following statements: 
\begin{itemize}
\item $d\subseteq\mathcal{P}(D)$; 
\item $\emptyset,D\in d$; 
\item For all $X,Y\in d,$$X\cap Y\in d$; and
\item for all $X\subseteq d$, $\bigcup Z\in d$. 
\end{itemize}
It is then easy to see that: 
\begin{prop}
($\mathcal{V}^{*}$) If $\mathcal{A}=\langle A,a\rangle$ is a structured
set, then $\mathcal{A}$ satisfies $Top$ iff $\mathcal{A}$ is a
topology with a domain $A$ of atoms.
\end{prop}

Given that the relationship between standard presentations of mathematical
structures like a topologies $\langle X,\mathcal{T}\rangle$ are so
easy to adapt into $\mathcal{L}_{\in}(D,d)$ it will be convenient
to slightly abuse notation and write the more familiar $\langle X,\mathcal{T}\rangle$
-- rather than $\langle A,a\rangle$ -- to denote a structured satisfying
$Top$. We shall leave it to the reader to make the appropriate translations
from the ordinary language of mathematics into $\mathcal{L}_{\in}(D,d)$
below. Moreover, it will also be helpful to streamline our notation
so that we conflate theories and the classes of structures that satisfy
them. Thus, given a theory $T$ and a structured set $\mathcal{A}$,
we shall frequently write $\mathcal{A}\in T$ rather than $\mathcal{A}$
satisfies $T$. 

Returning to the example above, we see that it nicely illustrates
the power of structured sets for representing mathematical structures.
We noted earlier that there was no obvious way to describe a topology
using a theory in first order logic. The obvious problems were that
given a topology $\langle X,\mathcal{T}\rangle$, $\mathcal{T}$ is
not a relation or function on $\mathcal{T}$. Moreover, with just
$\mathcal{T}$ we don't have enough structure to be able to describe
what $\mathcal{T}$ should be like. One might think of adding symbols
to play the role of intersection and union. This works quite well
for closure under finite intersections, but the union clause is more
difficult. The union operator is intended to be able to take infinitely
many arguments while first order logic is only concerned with finitary
operations. These problems simply don't emerge with structured sets.
Given $\langle X,\mathcal{T}\rangle$ satisfying $Top$, it doesn't
matter than $\mathcal{T}$is not a subset of $X$, we just need it
to be, so to speak, built from $X$. Moreover, we don't need to include
intersection and union in the representation of the structure since
they are already available for use in the background universe $V(\mathcal{A})$
that surrounds $\mathcal{A}$. But perhaps most pleasingly of all,
the way in which we are axiomatizing a topology is exactly the way
it is done in a topology book. No revisions or coding is required.
It falls straight off the shelf and into our framework. 

One might also worry that our formulation of structured sets only
allows for a single piece of structure on a particular domain. Suppose
that we want to consider a topology $\langle X,\mathcal{T}\rangle$
endowed with some extra structure in the form of a metric $d$. This
is easily addressed. Let $TopMet$ be the conjunction of the following
statements that can be easily formulated in $\mathcal{L}_{\in}(D,d)$: 
\begin{itemize}
\item $d$ is an ordered pair of the form $\langle d_{0},d_{1}\rangle$; 
\item $\langle D,d_{0}\rangle$ satisfies $Top$; 
\item $d_{1}:D\times D\to\mathbb{R}$ such that: $d_{1}(x,x)=0$; $x\neq y\to d_{1}(x,y)>0$;
$d_{1}(x,y)=d_{1}(y,x)$; and $d_{1}(x,z)\leq d_{1}(x,y)+d_{1}(y,z)$. 
\end{itemize}
It should be easy to see that this technique can be generalized to
meet the needs of ordinary mathematics and more. It is also worth
noting that we used the term $\mathbb{R}$ in the theory above. This
should be understood as a defined term in $\mathcal{V}^{*}$ that
denotes the usual version of $\mathbb{R}$ that resides in the kernel
of $\mathcal{V}^{*}$. In general, we'll leave it to the reader to
make the minor translations required to strictly fit the theories
we describe below into our framework.

\subsection{Interpretations and Isomorphisms}

We now have a way of representing mathematical structures and a linguistic
means of articulating theories that determine classes of those structures.
We are almost ready to discuss interpretation and interdefinability.
However, there is a final bump under the carpet to be addressed. Recall
that, in contrast with the case of $Top$ and $Nei$, we were not
able to show that $Bool$ and $Stone$ were sticks equivalent. The
reason for this was that when we define a Stone space $\mathcal{S}$
over a Boolean algebra $\mathbb{B}$ in the standard way, we define
a domain for that space using the set of ultrafilters on $\mathbb{B}$.
An ultrafilter is a subset of $\mathbb{B}$ and thus, the new domain
will be a subset of $\mathcal{P}(\mathbb{B})$. Moreover, the standard
way to define a Boolean algebra $\mathbb{B}^{*}$ from $\mathcal{S}$
is to take its clopen subsets. These are subsets of $\mathcal{S}$'s
domain and thus, they are also subsets of $\mathcal{P}(\mathbb{B})$.
The resultant Boolean algebra $\mathbb{B}^{*}$ is isomorphic to $\mathbb{B}$
but they are not identical. In particular, if we had started with
$\mathbb{B}$ being a structured set of the form $\langle B,b\rangle$,
then $\mathbb{B}^{*}$ would have a domain that was not a set of atoms
but rather an element of $V_{2}(B)$. This means that $\mathbb{B}^{*}$
is not a structured set. The natural way to treat interpretation in
this context would be as a process taking us from structured sets
to structured sets. But here we have an example of a standard ``interpretation''
where the output is not a structured set. 

This poses a problem with a couple of choices. On the one hand, we
could just say the equivalence between $Bool$ and $Stone$ doesn't
fit out framework. On the other, we could try to make our framework
more flexible. The first option seems draconian and in tension with
our goals to provide a general account of interdefinability in mathematics.
Thus, we shall opt for the second, although we shall do so in a way
that makes the distinction between the $Bool$-$Stone$ and $Top$-$Nei$
pairs visible and more, this distinction will illustrate a pleasing
generalization of a standard distinction in the theory of relative
interpretation in first order logic; i.e., between bi-interpretability
and definitional equivalence.

In order to implement this extra flexibility, we propose a natural
weakening of structured sets that we call \emph{quasi-structured sets}.
They will also be ordered pairs of the form $\mathcal{B}=\langle B,b\rangle$
with a domain $B$ and structure $b$. However, in order to accommodate
cases like we encountered with $Bool$ and $Stone$, we shall not
demand that $B$ is set of atoms. Rather we shall place some technical
conditions on $B$ and $b$ in order that they function the same way
-- for our purposes -- as a domain of atoms with a structure placed
upon it. It will be helpful to make a few preliminary definitions
on the way to our final definition. 
\begin{defn}
($\mathcal{V}^{*}$) (1) For a set $B$, say that $d$ is \emph{below
$B$ }if $d\in trlc(e)$ for some $e\in B$. 

(2) We say that $p$ is a\emph{ path from $d$ to }$b$ if $p=\langle p_{0},...,p_{n}\rangle$
where $n\geq1$, $p_{0}=d$, $p_{n}=b$ and for all $i<n$, $p(i)\in p(i+1)$. 

(3) We say that a path $p=\langle p_{0},...,p_{n}\rangle$ from $d$
to $b$ \emph{passes through} $B$ if there is some $i\leq n$ such
that $p_{i}\in B$. 
\end{defn}

This makes it a little easier to define our weakening of structured
sets. This is probably the fussiest definition in this paper. It may
take a little while to get used to, but the underlying idea is very
simple. We want something close to structured sets that does not require
the use of a strict domain of atoms.
\begin{defn}
($\mathcal{V}^{*})$ We say that $\mathcal{B}=\langle B,b\rangle$
is a \emph{quasi-structured set} if: 
\begin{enumerate}
\item $B$ is such that: 
\begin{enumerate}
\item $\emptyset\notin trcl(B)$; 
\item for all $x,y\in B$, $x\notin trcl(y)$; 
\end{enumerate}
\item $b$ is such that:
\begin{enumerate}
\item if $d$ is an atom with $d\in trcl(\{b\})$, then $d\in trcl(B)$;
and
\item if $d\in trcl(\{b\})$ is below $B$, then every path $p$ from $d$
to $b$ passes through $B$.
\end{enumerate}
\end{enumerate}
\end{defn}

Let's try to describe the information motivation behind each of the
clauses. Clause 1(a) is intended to ensure that the domain, $B$,
does not overlap the kernel. Since every set in the kernel has $\emptyset$
at the bottom of its transitive closure, this condition suffices for
this purpose. Clause 1(b) is intended to avoid elements of the domain
from being tangled up with each other. We do this by ensuring that
no element of the domain is in the transitive closure of any other
element. Clause 2(a) is analogous to the clause 2 of Definition \ref{def:StrucSet}.
It is designed to ensure that the structure, $b$, is built up from
atoms in the domain and not any other atoms. Finally clause 2(b) is
intended to ensure that the structure $b$ respects $B$ as a genuine
domain. The clause aims to do this by making sure that $b$ is, so
to speak, built from elements of $B$ and not from elements below.
The following diagram is intended to assist with the interpretation
of clause 2(b). 

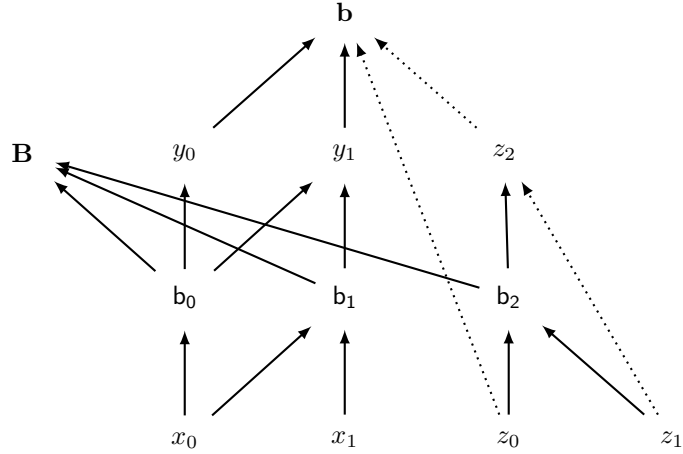
\begin{center}
\begin{tikzpicture}[     
node distance=1.2cm and 1.4cm,     
every node/.style={inner sep=6pt},     
->,     
>=latex,     
thick,     
shorten >=2pt,     
shorten <=0pt 
]
\node (b) {$\bold{b}$};
\node[below left=of b] (y0) {$y_0$}; 
\node[left=of y0] (B2) {$\bold{B}$};
\node[below=of b] (y1) {$y_1$};
\node[below right= of b] (z2) {$z_2$};
\node[below right=of B2] (b0) {$\mathsf{b_0}$}; 
\node[below=of y1] (b1) {$\mathsf{b_1}$}; 
\node[below right=of y1] (b2) {$\mathsf{b_2}$};
\node[below=of b0] (x0) {$x_0$}; 
\node[below=of b1] (x1) {$x_1$}; 
\node[below=of b2] (z0) {$z_0$};
\node[below right=of b2] (z1) {$z_1$};

\draw (y0) -- (b); 
\draw (y1) -- (b);
\draw (b0) -- (B2); 
\draw (b1) -- (B2); 
\draw (b2) -- (B2);
\draw (b0) -- (y1); 
\draw (b1) -- (y1);
\draw (b0) -- (y0);
\draw (x0) -- (b0); 
\draw (x0) -- (b1);
\draw (x1) -- (b1);
\draw[dotted] (z0) -- (b);
\draw[dotted] (z1) -- (z2);
\draw[dotted] (z2) -- (b);
\draw (z1) -- (b2);
\draw (z0) -- (b2);
\draw (b2) -- (z2);
\end{tikzpicture}
\end{center}

The idea is that $\mathcal{B}=\langle B,b\rangle$ where $B=\{b_{0},b_{1},b_{2}\}$
is intended to represent a putative quasi-structure using a graph
where we have an arrow between vertices $u$ and $v$ iff $u\in v$.
Thus, we see that $b_{0},b_{1},b_{2}\in B$ and $b$ is canonical
ordered pair of $b_{0}$ and $b_{1}$. If we ignore the dotted arrows
and suppose that $x_{0},x_{1},z_{0}$ and $z_{1}$ are atoms, then
$\mathcal{B}$ is a quasi-structured set. With regard to clause 2(b),
we see that $x_{0}$ and $x_{1}$ are both below $B$, but every path
from $x_{0}$ and $x_{1}$ to $b$ passes through $B$. On the other
hand, if we didn't ignore the dotted arrows, $\mathcal{B}$ would
not be a quasi-structured set. While $z_{0}$ and $z_{1}$ are both
below $B$, neither of the paths $\langle z_{0},b\rangle$ and $\langle z_{1},z_{2},b\rangle$
pass through $B$. This motivation behind this definition is most
clearly illustrated with the following generalization of Proposition
\ref{prop:V*buildsV(X)}: 
\begin{prop}
($\mathcal{V}^{*}$) If $B$ is the domain of a quasi-structured set
$\mathcal{B}=\langle B,b\rangle$, then $\langle V(B),\in,B\rangle$
satisfies $ZFCU$.
\end{prop}

This tells us that when we build universe of domains $B$ of quasi-structured
sets we get models of $ZFCU$ whose atoms are the domains we started
with. Thus, we again avoid the recreation problem and obtain a natural
hierarchy over a mathematical structure. Note, however, that just
because $B$ is a set of atoms relative to the $\langle V(B),\in B\rangle$
this does not mean that $B$ is a set of atoms in $\mathcal{V}^{*}$.
Indeed that is the point. 

Finally, we can describe how one theory may interpret another. First,
let us recall how constant symbols can be defined in a definitional
expansion of a theory. This will be the underlying mechanism for our
generalized notion of interpretation. If we are working a theory $T$,
we may introduce a new constant symbol, $t$, if there is some formula
$\varphi_{t}(x)$ in the language of $T$ such that $T$ proves that
there is a unique $x$ such that $\varphi_{t}(x)$. In such cases,
we shall write $t=x$ and we shall call $t$ a \emph{term}. For a
classic example from $ZFC$ and $ZFCU$, $\emptyset$ is a term. Our
interpretations will take the form of terms in universes $V(\mathcal{A})$
over quasi-structured sets $\mathcal{A}$. We can describe them now
as follows:
\begin{defn}
($\mathcal{V}^{*}$) Let $\tau_{d}(x)$ and $\tau_{s}(x)$ be formulae
of $\mathcal{L}_{\in}(D,d)$. We say that they form a $T$-interpretation,
$t$, if whenever $\mathcal{A}$ is a quasi-structured set satisfying
$T$, then the following statements are satisfied in $V(\mathcal{A})$; 
\begin{itemize}
\item there is a unique $B$ such that $\tau_{d}(B)$; 
\item there is a unique $b$ such that $\tau_{s}(b$); and
\item $\mathcal{B}=\langle B,b\rangle$ is a quasi-structured set.
\end{itemize}
Thus $\tau_{d}$ and $\tau_{s}$ determine terms $t_{d}$ and $t_{s}$
over $T$ such that $t_{d}=B$ iff $\tau_{d}(B)$ and $t_{s}=b$ iff
$\tau_{s}(b)$. We then say that $t=\langle t_{s},t_{d}\rangle$ is
a $T$\emph{-interpretation} and we write $t(\mathcal{A})$ to denote
$\mathcal{B}$.\label{def:T-interpt}
\end{defn}

Informally speaking, we can think of the terms $t_{d}$ as delivering
a new domain and $t_{s}$ as delivering a structure on that domain.
We may then describe what it means for one theory to interpret another.
The following diagram is intended to provide a rough illustration
of the situation described in the definition above.

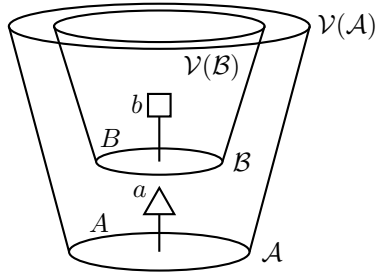
\begin{center} 
\begin{tikzpicture}   
\draw[thick] (0,3) ellipse (2 and 0.3);   
\draw[thick] (0,0) ellipse (1.2 and 0.25);      

\draw[thick] (-2,3) -- (-1.2,0);   
\draw[thick] (2,3) -- (1.2,0);      

\draw[thick] (0,0) -- (0,0.5);      

\draw[thick] (0,0.85) -- (-0.2,0.5) -- (0.2,0.5) -- cycle;      

\draw[thick] (0,3) ellipse (1.4 and 0.21);   
\draw[thick] (0,1.2) ellipse (0.84 and 0.175);      

\draw[thick] (-1.4,3) -- (-0.84,1.2);   
\draw[thick] (1.4,3) -- (0.84,1.2);      

\draw[thick] (0,1.2) -- (0,1.8);      

\draw[thick] (-0.15,1.8) rectangle (0.15,2.1);     

\node at (2.5,3) {$\mathcal{V}(\mathcal{A})$};   
\node at (1.5,0) {$\mathcal{A}$};   
\node at (-0.26,0.8) {$a$};  
\node at (-0.8,0.38) {$A$};  

\node at (0.7,2.45) {$\mathcal{V}(\mathcal{B})$}; 
\node at (1.1,1.2) {$\mathcal{B}$};  
\node at (-0.3,1.97) {$b$};  
\node at (-0.65,1.52) {$B$};   

\end{tikzpicture}
\end{center}
\begin{defn}
($\mathcal{V}^{*}$) We say that $T$ interprets $S$ if there is
a $T$-interpretation $t$ such that for all structured sets $\mathcal{A}$
satisfying $T$, $t(\mathcal{A})$ satisfies $S$. We abbreviate this
by writing $t:T\to S$.
\end{defn}

This is directly analogous to the standard notion of interpretation
for models of first order logic. Indeed, it's a little easier to state
since we have a simple common language for structured sets. 

We are just about to describe our promised interdefinability relations,
but first, we have one more task to accomplish. We need to define
a suitable notion of isomorphism between quasi-structured sets. Once
this is done, the final definitions will be almost trivial. We start
by defining what we call the \emph{lift} of a function between the
domains of a pair of quasi-structured sets. The key point is that
there is enough information in the action of a function on the domain
to determine its behavior on the structure above it. First let us
say that the \emph{field} of a quasi-structure $\mathcal{A}=\langle A,a\rangle$,
abbreviated $field(\mathcal{A})$, is $(A\cup trcl(\{a\}))\cap V(A)$.
Intuitively speaking, $field(\mathcal{A})$ is the collection of sets
that occur, so to speak, in between (and including) $A$ and $a$. 
\begin{defn}
($\mathcal{V}^{*}$) Let $\mathcal{A}=\langle A,a\rangle$ and $\mathcal{B}=\langle B,b\rangle$
be quasi-structured sets and suppose that $f:A\to B$. The \emph{lift}
of $f$ to $f_{\mathcal{A}}^{+}:field(\mathcal{A})\to field(\mathcal{B})$
is defined by $\in$-recursion on $field(\mathcal{A})$ so that for
all $x\in field(\mathcal{A})$\label{def:lift}
\[
f_{\mathcal{A}}^{+}(x)=\begin{cases}
x & \text{if }x\in V\\
f(x) & \text{if }x\in A\\
\{f_{\mathcal{A}}^{+}(y)\ |\ y\in x\} & \text{otherwise.}
\end{cases}
\]
\end{defn}

Using the lift we may then define an isomorphism between quasi-structured
sets as follows: 
\begin{defn}
($\mathcal{V}^{*}$) Given quasi-structured sets $\mathcal{A}=\langle A,a\rangle$
and $\mathcal{B}=\langle B,b\rangle$, we say that $f:A\to B$ is
an isomorphism, abbreviated $f:\mathcal{A}\cong\mathcal{B}$, if $f$
is a bijection and $f_{\mathcal{A}}^{+}(a)=b$.\label{def:auto}
\end{defn}

The definition is very simple and perhaps looks too simple to be an
appropriate definition for isomorphism. However, the key point is
that the structure $a$ can be very complex entity in relation to
$A$. In particular, the sets contained in the field of $\mathcal{A}$
can be made to do very intricate work. To illustrate a classic case
of this, observe that: 
\begin{prop}
($\mathcal{V}^{*}$) Suppose $\langle X,\mathcal{T}\rangle$ and $\langle Y,\mathcal{S}\rangle$
are quasi-structured sets satisfying $Top$; i.e., they are topologies
with domains consisting of atoms. Then $\langle X,\mathcal{T}\rangle$
and $\langle Y,\mathcal{S}\rangle$ are homeomorphic iff they are
isomorphic as structured sets. 
\end{prop}

Having a notion is isomorphism available, also allows us to illustrate
how our notion of interpretation generalizes the theory of relative
interpretation in first order logic. To see this, recall that in first
order logic, if we have a function $t:mod(T)\to mod(S)$ between first
order models of theory $T$ and models of theory $S$, then whenever
$\mathcal{M}\cong\mathcal{N}$ are models of $T$, then we also have
$t(\mathcal{M})\cong t(\mathcal{N})$. We might say that first order
interpretations preserves isomorphism. It is easy to see that this
is also the case for our generalized notion of interpretation. 
\begin{prop}
If $t:T\to S$ is an interpretation, then for all $\mathcal{A}\cong\mathcal{B}\in T$,
$t(\mathcal{A})\cong t(\mathcal{B})$. 
\end{prop}

Before we give our main definitions, we observe a crucial fact about
these isomorphisms that will allow us to generalize the proof technique
used for Proposition \ref{prop:protoAuto}. First, we observe that
whenever $f:\mathcal{A}\cong\mathcal{B}$ we can lift the bijection
$f:A\to B$ not just to the $field(\mathcal{A})$, but to the entirety
of $V(\mathcal{A})$. To see this, we simply continue lifting $f$
beyond $f^{+}$ by deploying the same technique as used in Definition
\ref{def:lift}. We denote this function as $f^{*}:V(\mathcal{A})\to V(\mathcal{B})$
and observe that it gives a full isomorphism between the universes
associated with $\mathcal{A}$ and $\mathcal{B}$. Of course, $f^{*}$
is a proper class and not a set, which makes it inconvenient for some
purposes. However, the following lemma is very helpful indeed. 
\begin{lem}
Let $\mathcal{A}$ and $\mathcal{B}$ be structured sets and $f:A\to B$
be a bijection witnessing that $f:\mathcal{A}\cong\mathcal{B}$. Then
for all formula $\varphi(x_{0},...,x_{n})$ from $\mathcal{L}_{\in}(D,d)$
and objects $c_{0},...,c_{n}$ from $V(A)$ we have 
\[
V(\mathcal{A})\models\varphi(c_{0},...,c_{n})\ \Leftrightarrow\ V(\mathcal{B})\models\varphi(f^{*}(c_{0}),...,f^{*}(c_{n})).
\]
\end{lem}

In the particular case where $\mathcal{A}$ is $\mathcal{B}$ and
thus, $f$ is an automorphism, we see that while $f^{*}$ is often
able to move elements of $V(\mathcal{A})$, properties definable by
formulae like $\varphi$ are not affected by $f^{*}$. This will be
very useful below.

\subsection{Definitional equivalence and Bi-interpretability\label{subsec:Definitional-equivalence-and}}

Finally, we can define our preferred interdefinability relations.
Like the notion of interpretability described above, they very obviously
generalize equivalence relations used when comparing theories in first
order logic. 
\begin{defn}
($\mathcal{V}^{*}$) Say that theories $T$ and $S$ are \emph{definitionally
equivalent }if there are interpretations $t:T\leftrightarrow S:s$
such that:
\begin{itemize}
\item $s\circ t(\mathcal{A})=\mathcal{A}$ for all $\mathcal{A}\in T$;
and
\item $t\circ s(\mathcal{B})=\mathcal{B}$ for all $\mathcal{B}\in S$. 
\end{itemize}
\end{defn}

Note that aside from the shift in logical framework, this is exactly
the same as ordinary definitional equivalence in first order logic.
\begin{defn}
Say that $T$ and $S$ are \emph{bi-interpretable} if there exist
interpretations $t:T\leftrightarrow S:s$ and functions $\eta$ and
$\nu$ that are uniformly definable over $T$ and $S$ respectively
such that:
\begin{itemize}
\item $\eta^{V(\mathcal{A})}:\mathcal{A}\cong s\circ t(\mathcal{A})$ for
all $\mathcal{A}\in T$; and
\item $\nu^{V(\mathcal{B})}:\mathcal{B}\cong t\circ s(\mathcal{B})$ for
all $\mathcal{B}\in S$.
\end{itemize}
\end{defn}

Again, this is just the usual definition of bi-interpretability adapted
to the $\mathcal{V}^{*}$ -framework. We then note that if $T$ and
$S$ are definitionally equivalent, they are also (vacuously) bi-interpretable. 

Of course, these are proper generalizations of the ordinary notions
of definitional equivalence and bi-interpretability, so some care
about nomenclature is required. Since this is a paper about the $\mathcal{V}^{*}$-framework,
we'll depart from convention and use the words of the definitions
above.\footnote{Outside the context of this paper, it might be helpful to say that
theories are $\mathcal{V}^{*}$-definitionally equivalent and $\mathcal{V}^{*}$-bi-interpretable.} To avoid confusion, we'll then be careful to note when we are talking
the standard versions of these definitions that work in the context
of first order logic.

Before we move onto applications of the framework, it will be helpful
to observe a useful fact about definitionally equivalent theories
that also generalizes the standard first order characterization.
\begin{defn}
Let us say that an interpretation $t=\langle t_{d},t_{s}\rangle:T\to S$
\emph{preserves domains} if whenever $\mathcal{A}=\langle A,a\rangle\in T$,
then $t_{d}^{V(\mathcal{A})}=A$.
\end{defn}

\begin{prop}
If $t:T\leftrightarrow S:s$ are interpretations witness definitional
equivalence, then $t$ and $s$ preserve domains.
\end{prop}

We leave the proof to the reader, but the idea is very simple. Just
as in the first order logic case, if we are to translate back and
forth and end up exactly where we started, then neither of those interpretations
can afford to lose an element of the domain. 

In the next section, we'll start applying this framework to a representative
collection of examples, but before we do this, let's do a quick stock
take. After doing a lot of quite technical work, I hope the reader
sees that we've ended up in a relatively familiar place. In first
order logic, we interpret theories by defining models of the target
theory over models of some base theory. In the framework above, we
interpret theories by defining structured sets in the target theory
over universes generated from structures in the base theory. If we
are willing to grant these generous definability resources, then I
submit that this gives us a reasonable account of interdefinability
that satisfies the desiderata set out at the beginning of this paper.
Thus, as we shall see, we can harness the power of the category theoretic
approach to theories of complex mathematical structures while maintaining
a plausible story about interdefinability. Let's put it to work!

\section{Applications\label{sec:Applications}}

In this final, quite large section, our goal is to take the framework
outlined above for test drive that will demonstrate its many strengths
and also some of its limitations. Our main objectives will be to show
that standard equivalence proofs can be torn right out the book, while
inequivalence is far more common and happens for, what we might think
are, the right reasons.

This final part of the paper is the longest and so warrants its own
little overview. We begin in Section \ref{subsec:Equivalence} with
a short overview of some equivalence results. These are very straightforward
and, as such, we do not spend much time here. In Section \ref{subsec:Inequivalence},
we begin the more difficult task of establishing inequivalence between
theories. This is where the naive approaches of Section \ref{sec:1}
fell down. Here, we set out a basic methodology based on automorphisms
for establishing that theories fail to be definitionally equivalent
or bi-interpretable. We begin with our simple example of $Set1$ and
$Set2$ from Section \ref{subsec:Counterintuitive-Cantor-Bernstei}
and then consider some more interesting cases outside the reach of
the first order theory of interpretation. In Section \ref{subsec:Counterintuitive-Cantor-Bernstei},
we consider a more natural pair of theories that are not bi-interpretable
and show that by introducing more structure to those theories they
can become definitionally equivalent. We then discuss how this relates
to the distinction between material and structural set theory. Sections
\ref{subsec:Theories-of-rigid} and \ref{subsec:Comparing-particular-rigid}
then deal with theories where the automorphism methodology breaks
down: theories of rigid structures. We show that forcing can help
when comparing theories of multiple rigid structures, but in the case
of particular rigid structure, the $\mathcal{V}^{*}$-framework hits
a wall. Section \ref{subsec:Comparing-particular-rigid} aims to precisely
delineate this limitation and concludes in Section \ref{subsec:What-can-we-do?},
by considering how the $\mathcal{V}^{*}$-framework could be simply
revised to address this.

\subsection{Equivalence\label{subsec:Equivalence}}

Let's return to where we started. Clearly both $Top$ and $Nei$,
from Definitions \ref{def:Top} and \ref{def:Nei}, can be articulated
as theories in $\mathcal{L}_{\in}(D,d)$. It is then easy to see that:
\begin{prop}
$Top$ and $Nei$ are definitionally equivalent. 
\end{prop}

The proof of Proposition \ref{prop:TopNei} suffices here as well:
we can tear proofs out of the book. Next, let us return to $Bool$
and $Stone$, which can also clearly be articulated as theories in
$\mathcal{L}_{\in}(D,d)$. Once again the result is straightforward,
except this time we have bi-interpretability rather than definitional
equivalence. 
\begin{prop}
$Bool$ and $Stone$ are bi-interpretable.\label{prop:BoolStoneBiint}
\end{prop}

\begin{proof}
The proof of Proposition \ref{thm:BoolStone} can be transferred directly,
however, we have one extra task to complete in order to demonstrate
bi-interpretability: we must define the required isomorphisms. We
shall pick up where the proof of Proposition \ref{thm:BoolStone}
leaves off. First, we want to give a uniform definition of a function
$\eta$ such that for any Boolean algebra $\mathbb{B}$, $\eta^{V(\mathbb{B})}$
is an isomorphism between $\mathbb{B}$ and $s\circ t(\mathbb{B})$.
Recalling that $X$ is the set of ultrafilters on $\mathbb{B}$, we
let $\eta(b)$ for $b\in\mathbb{B}$ be 
\[
\{U\in X\ |\ b\in U\}.
\]
Second, we want to define a uniform definition of a function $\eta$
such that for any stone space $\langle X,\mathcal{T}\rangle$, $\eta^{V(\langle X,\mathcal{T}\rangle)}:\langle X,\mathcal{T}\rangle\cong t\circ s(\langle X,\mathcal{T}\rangle)$.
Let us write $\langle Y,\mathcal{S}\rangle$ to denote $t\circ s(\langle X,\mathcal{T}\rangle)$.
Then recall that $Y$ is the set of ultrapowers on $\mathbb{B}$ where
$\mathbb{B}$ is the Boolean algebra obtained by taking the field
of sets given by the clopen sets in $\mathcal{T}$. We then let $\eta^{V(\langle X,\mathcal{T}\rangle)}:X\to Y$
be such that for all $x\in X$ 
\[
\eta^{V(\langle X,\mathcal{T}\rangle)}(x)=\{b\in\mathbb{B}\subseteq\mathcal{T}\ |\ x\in b\}.
\]
It is easy to see that this is an ultrafilter on $\mathbb{B}$ as
required. We leave the rest of the details to the reader.\footnote{Although as above, we recommend the reader consult \citep{HalvLogPhilSci}
for a comprehensive proof.}
\end{proof}
The proposition above is clearly a version of what is commonly known
as the Stone Duality Theorem. This theorem tells us that natural categories
associated with Boolean algebras and Stone spaces are equivalent to
each other. It will be helpful to explain this in more detail. Let
$Bool^{gpd}$ be the category whose objects are Boolean algebras and
whose arrows are isomorphisms. Let $Stone^{gpd}$ be the category
whose objects are Stone spaces and whose arrows are homeomorphisms.
\footnote{By using the canonical isomorphisms in these categories, we are actually
defining groupoids. Usually the arrows would be continuous maps and
homomorphisms respectively. This way of framing thing hides the ``duality''
aspect of their relationship. But by characterizing things this way,
it makes it easier for us to understand the relationship between our
framework and category-theoretic approaches. We'll discuss this further
soon. } Recall the definition of an equivalence between categories. 
\begin{defn}
Suppose $\mathcal{C}$ and $\mathcal{D}$ are categories. We say that
$\mathcal{C}$ and $\mathcal{D}$ are \emph{equivalent }if there are
functors $F:\mathcal{C}\Rightarrow\mathcal{D}$ and $G:\mathcal{D}\Rightarrow\mathcal{C}$
along with natural transformations $\eta,\nu$ such that:\footnote{See Section 7.8 of \citep{AwodeyCat} for further discussion.}
\begin{itemize}
\item $\eta_{A}:A\cong G\circ F(A)$ for all objects $A$ in $\mathcal{C}$;
and
\item $\nu_{B}:B\cong F\circ G(B)$ for all objects $B$ in $\mathcal{D}$. 
\end{itemize}
\end{defn}

The family resemblance with bi-interpretability should be manifest.
Moreover, it should be clear that the proof of Proposition \ref{prop:BoolStoneBiint}
can also be used to establish that $Bool^{gpd}$ and $Stone^{gpd}$
are equivalent as categories. Moreover, similar examples are easily
plucked from the literatures. For a classic mathematical example we
have:
\begin{thm}
(Gelfand) The theories of commutative $C^{*}$-algebras and compact
Hausdorff spaces are bi-interpretable.\footnote{See Theorem 1.16 in \citep{folland1994course} for more details.}
\end{thm}

And for an example from physics we have:
\begin{thm}
\citep{RosenstockGRvsEA} The theory of general relativity is bi-interpretable
with the theory of Einstein algebras.
\end{thm}

In both of these cases, we are taking duality theorems from the literature
and just directly importing the proofs into our framework. This is
something we should expect to be able to keep on doing in our proposed
system and this speaks in favor of my claim that we can take proofs
straight out of the book. However, there is a sense in which I'm just
stealing these proofs from what is arguably a competitor framework:
category theory. Moreover, there is also a sense in the results, as
stated above are weakenings of their standard statements. This warrants
some discussion that will help us better understand that the current
framework offers that category theory does not. 

The usual way to set up the Stone duality is to start with categories
$Bool^{cat}$ and $Stone^{cat}$. These categories have the same objects
as $Bool^{gpd}$ and $Stone^{gpd}$ but they use weaker arrows. In
particular, $Bool^{cat}$ uses homomorphisms while $Stone^{cat}$
uses continuous maps. In this setting, we do not get equivalence,
but rather, duality. Recall that the opposite category of some category
$\mathcal{C}$ is obtained by simply reversing the direction of all
arrows in that category. It can then be seen, using essentially the
same argument sketched above, that these categories are dual to each
other in the sense that $Bool^{cat}$ is equivalent to the opposite
category of $Stone^{cat}$. This extra information about reversing
arrows is not part of the information detected by our notion of bi-interpretability.
As such, we might worry that the equivalence relations defined above
are unable to recapture important information that mathematicians
care about. This is a natural worry, but one that is easily addressed.

The first thing to note is that although definitional equivalence
and bi-interpretability do not capture the content of duality, there
is nothing stopping us from working in the $\mathcal{V}^{*}$ framework
to define equivalence relations that do. Indeed, I think it is likely
that investigation into this problem would yield a better understanding
of the relationship between set theoretic and category theoretic approaches
and attitudes toward equivalence in mathematics. However, if we think
back to our initial motivations for developing this framework, we
will recall that we had some concerns about categorical equivalence
and its philosophical implications for theory comparison. In particular,
we we worried that novel conceptual elements of arrows and functors
placed significant hurdles in front of any story about equivalence
and interdefinability. As such, I would like to suggest that while
the introduction of arrows between structures reveals interesting
mathematical information, it has no obvious place in a story about
how the language of one theory user may be translated into the language
of another. By contrast, definitional equivalence and bi-interpretability,
as defined above, make no essential use of arrows between structured
sets. Moreover, theories and the translations between them are treated
as linguistic objects that fit very neatly into a story about translation.
It seems very natural to say that $Top$ and $Nei$ are interdefinable
variants of each other; and the fact that they are definitionally
equivalent provides some explanation of this intuition. So they key
point here is that, yes, duality becomes invisible, but this appears
to be the cost of obtaining a framework just focused on translation. 

The second thing to note is that in the $\mathcal{V}^{*}$ framework
functors do play a silent role. So while the equivalence relations
described above are distinct from their category theoretic cousins,
they are still very closely related. To illustrate this, suppose we
have theories $T$ and $S$. Then we may easily obtain theory groupoids
$T^{gpd}$ and $S^{gpd}$ from $T$ and $S$ by letting the objects
be the structured sets satisfying the theory and letting the arrows
between them be the isomorphisms between them. Now if the theories
$T$ and $S$ were bi-interpretable as witnessed by interpretations
$t:T\leftrightarrow S:s$, then it is not difficult to see that $t$
and $s$ can be used to derive functors $F_{t}:T^{gpd}\Rightarrow S^{gpd}$
and $G_{s}:S^{gpd}\Rightarrow T^{gpd}$ that witness the equivalence
of those categories.\footnote{In particular, for objects $\mathcal{A}$ from $T^{gpd}$, we let
$F_{t}(\mathcal{A})=t(\mathcal{A})$. And for isomorphisms $h:\mathcal{A}\cong\mathcal{B}$
from $T^{gpd}$, we let $F_{t}(h):t(\mathcal{A})\cong t(\mathcal{B})$
be the bijection $h^{*}:t_{d}(A)\to t_{d}(B)$ such that $h^{*}=h_{\mathcal{A}}^{+}\restriction t_{d}(A)$.
$G_{s}$ is defined similarly. Moreover, the required natural isomorphisms
are then easily defined form the definable isomorphisms witnessing
the bi-interpretation.\label{fn:IntToFunct}} Thus, we see that for theories, as defined above and then represented
as groupoids, bi-intepretability implies definitional equivalence.
And more, the converse is also ``practically'' true. By this we
mean that, in practice, the arguments used to establish equivalence
between categories of these kinds are quite constructive: they prove
that the functor exists by defining the functor. I'm not aware of
any proof of such an equivalence that begins by supposing toward a
contradiction that there was no such functor. As such, these arguments
also tend to slide immediately into the framework we've described
above.

So far, so good. We have a framework that gives us easy equivalence
proofs that can be taken from textbooks with little or no change.
The interdefinability definitions are closely related to category
theoretic relations, but they remain compatible with a simple story
about translation and equivalence. But we always knew this would be
the easy part. The rubber will only hit the road when we discuss inequivalence.

\subsection{Inequivalence\label{subsec:Inequivalence}}

We continue our test drive by returning to the example that stumped
sticks equivalence: $Set1$ and $Set2$. We'll do better now, but
let's take a moment to set things up. Let us now construe $Set1$
as the the theory in $\mathcal{L}_{\in}(D,d)$ of singleton sets of
atoms with no structure. Thus, structured sets satisfying $Set1$
can be regarded as being of the form $\mathcal{A}=\langle\{e\}\rangle$.\footnote{Strictly, we should include a structure to go with the domain. We'll
omit it here and below, by $\emptyset$ would suffice for this purpose.} Similarly, $Set2$ will be the theory of pairs of atoms and so, structured
sets satisfying $Set2$ will be regarded as being of the form $\langle\{b,c\}\rangle$.
It is not difficult to see that $Set1$ and $Set2$ are not definitionally
equivalent. We include a proof since it seems like a good idea to
see a simple case written out in some detail. It also provides an
opportunity to illustrate that, despite some of the technical fussiness
of Section \ref{sec:V*}, the proofs tend to flow in quite a natural
fashion. 
\begin{prop}
($\mathcal{V}^{*}$) $Set1$ and $Set2$ are not definitionally equivalent. 
\end{prop}

\begin{proof}
Suppose toward a contradiction that $t:Set1\leftrightarrow Set2:s$
are interpretations witnessing that $Set1$ and $Set2$ are definitionally
equivalent. Let $\mathcal{A}=\langle\{b,c\}\rangle\in Set2$. Then
$s(\mathcal{A})=\langle\{e\}\rangle$ for some $e\in V(\mathcal{A})$
where $\{e\}=t_{d}^{V(\mathcal{A})}$. Note that at most one of $b$
and $c$ can be an element of $V(s(\mathcal{A}))=V(\langle\{e\}\rangle)$.
Then we see that $t\circ s(\mathcal{A})=\langle\{i,j\}\rangle$ for
$i\neq j\in V(s(\mathcal{A}))$ where $\{i,j\}=s_{d}^{V(s(\mathcal{A}))}$.
Since we've assumed that $\mathcal{A}=t\circ s(\mathcal{A})$ we must
have: $b=i$ and $b=j$; or $b=j$ and $c=i$. But this is clearly
impossible since either $b\notin V(s(\mathcal{A}))$ or $c\notin V(s(\mathcal{A}))$. 
\end{proof}
In fact, using an argument similar to the proof of Proposition \ref{prop:protoAuto},
we can show that $e$ in the proof above cannot be either $b$ or
$c$; i.e., we use an automorphism argument to show that no proper
subset of $\{b,c\}$ other than $\emptyset$ can be defined in $V(\mathcal{A}$).
Nonetheless, in contrast to our discussion of second order logic,
$Set2$ can interpret $Set1$. In particular, given $\mathcal{A}=\langle\{b,c\}\rangle\in Set2$,
we may let $s(\mathcal{A})=\langle\{\{b,c]\}\rangle\in Set1$ which
is clearly definable in $V(\mathcal{A})$. Thus, we have some parallels
with the case of second order logic, but our notion of interpretation
is notably stronger in the $\mathcal{V}^{*}$ framework. This extra
power gives rise to a natural question: are $Set1$ and $Set2$ perhaps
bi-interpretable? We shall see that they are not. This can be proved
in a direct manner, but it will be helpful to use this simple example
to unpack a proof template that can be used very generally. The following
lemma is the key.
\begin{lem}
($\mathcal{V}^{*}$) Suppose $T$ and $S$ are bi-interpretable as
witnessed by interpretations $t:T\leftrightarrow S:s$. Then for all
$\mathcal{A}\in T$\label{lem:AutoTemplate}
\[
Aut(\mathcal{A})\cong Aut(t(\mathcal{A}))
\]
where $Aut(\mathcal{A})$ is the automorphism group of the structured
set $\mathcal{A}$.
\end{lem}

While it is quite possible to provide a direct proof of this, it will
be more compact to borrow a couple of standard facts from category
theory. 
\begin{fact}
Let $\mathcal{C}$ and $\mathcal{D}$ be locally small categories.
Suppose $\mathcal{C}$ and $\mathcal{D}$ are equivalent as witnessed
by functors $F:\mathcal{C}\Leftrightarrow\mathcal{D}:G$. Then\label{fact:CatFacts}
\begin{enumerate}
\item $F$ is full, faithful and essentially surjective; and
\item For all objects $A$ from $\mathcal{C}$, $Aut(A)\cong Aut(F(A))$. 
\end{enumerate}
\end{fact}

\begin{proof}
For (1) see Proposition 7.28 of \citep{AwodeyCat}.\footnote{Note that in the context we are working, we cannot prove the converse
unless we have some form of global choice. For example, assuming $V=HOD$
would suffice. Given a functor $F$ as described in (1), we use the
choice principle to pick one of what could be many objects $A$ from
$\mathcal{C}$ such that $F(A)=B$ for some object $B$ in $\mathcal{D}$. } For (2), we see from (1) that $F$ is full, faithful and essentially
surjective. Now let $A$ be an object from $\mathcal{C}$. We let
$F^{*}:Aut(A)\to Aut(F(A))$ be defined such that for all $g:A\cong A$
in $\mathcal{C}$, $F^{*}(g)=F(g)$; i.e., $F^{*}$ is the restriction
of $F$ to automorphisms. We claim that $F^{*}$ is itself an automorphism.
Injectiveness follows from the faithfulness of $F$; and surjectiveness
follows from the fullness of $F$. Preservation of identity and composition
follow by functoriality as does the preservation of inverses.
\end{proof}
The proof of our lemma then follows easily.
\begin{proof}
(of Lemma \ref{lem:AutoTemplate}) Let $T^{cat}$ and $S^{cat}$ be
the categories obtained from $T$ and $S$ by using: the structured
sets that satisfy $T$ and $S$ as objects; and isomorphisms between
them as arrows. One can then obtain functors $F_{t}:T^{cat}\Rightarrow S^{cat}$
and $G_{s}:S^{cat}\Rightarrow T^{cat}$ from the interpretations $t$
and $s$ that witness that $T^{cat}$ and $S^{cat}$ are equivalent
categories.\footnote{See footnote \ref{fn:IntToFunct} for a sketch of how such a functor
may be obtained.} The result then follows from Fact \ref{fact:CatFacts}(2).
\end{proof}
This gives us a very simple, and very common strategy, for showing
a failure of bi-interpretability. It suffices, for example, to just
find a structured set $\mathcal{B}\in S$ whose automorphism group
is not isomorphic to any structured set $\mathcal{A}\in T$. Let's
put this strategy to work. 
\begin{prop}
($\mathcal{V}^{*}$) $Set1$ and $Set2$ are not bi-interpretable. 
\end{prop}

\begin{proof}
Note that both of these theories are categorical in the sense that
they both contain just one structure up to isomorphism. For $\mathcal{A}$
in $Set1$, we obviously have $|Aut(\mathcal{A})|=1$ since the only
isomorphism of a singleton is the identity. Similarly, for $\mathcal{B}$
in $Set2$, we have $|Aut(\mathcal{B})|=2$ since the only isomorphism
of a pair other than the identity is the switching permutation. Clearly,
$Aut(\mathcal{A})$ and $Aut(\mathcal{B})$ are not isomorphic and
so, Lemma \ref{lem:AutoTemplate} tells us that $Set1$ and $Set2$
are not bi-interpretable. 
\end{proof}
This is the main device behind most of our proofs that bi-interpretability
fails. Once again, the proof relies on familiar techniques from category
theory that have been re-situated in our $\mathcal{V}^{*}$ framework.
However, the example above is so simple that is unlikely to help us
develop better intuitions about when theories are equivalent and inequivalent
in our framework. With that in mind, let us explore something a little
more concrete. 

Let $Reals_{Top}$ be the $\mathcal{L}_{\in}(D,d)$-theory of the
standard topology on the reals. More formally, $Reals_{Top}$ will
say that a structured sets $\langle X,\mathcal{T}\rangle$ satisfies
$Top$ and that $\langle X,\mathcal{T}\rangle$ is homeomorphic to
the to the standard interval topology on $\mathbb{R}$, where $\mathbb{R}$
denotes the reals in the kernel as they are usually defined in $ZFC$.
Note that we cannot let $X$ simply be $\mathbb{R}$ since then $\langle X,\mathcal{T}\rangle$
would not be a structured set as its domain would not consist of atoms.
We shall see that this move plays a crucial role in upholding our
motivating intuitions about generalizing second order logic. In particular,
automorphisms of $V(\langle X,\mathcal{T}\rangle)$ can often move
elements of $X$, but elements of the kernel like elements of $\mathbb{R}$
are fixed. Next, let $Reals_{Met}$ be the $\mathcal{L}_{\in}(D,d)$-theory
of the Euclidean Metric on the reals. More formally, $Reals_{Met}$
say that a structured set $\langle X,d\rangle$ is isomorphic to $\mathbb{R}$
(from the kernel) equipped with the Euclidean metric. Intuitively
speaking, it seems obvious that $Reals_{Top}$and $Reals_{Met}$ can't
be interdefinable. Let's show that the $\mathcal{V}^{*}$ framework
upholds that intuition. 
\begin{prop}
($\mathcal{V}^{*}$) $Reals_{Top}$ and $Reals_{Met}$ are not definitionally
equivalent.\label{prop:RealTopNMet=0000ACdefEq}
\end{prop}

\begin{proof}
Suppose toward a contradiction that we have interpretations $t:Reals_{Top}\leftrightarrow Reals_{Met}:s$
witnessing definitional equivalence. Then given some $\mathcal{A}=\langle X,\mathcal{T}\rangle$
satisfying $Reals_{Top}$ we see that since $s\circ t(\mathcal{A})=\mathcal{A}$
that $\mathcal{A}$ and $t(\mathcal{A})$ must have the same domain;
or more formally, we must have $t(\mathcal{A})=\langle X,d\rangle$.
This means that $d$ is definable in $V(\mathcal{A})$ and so for
all automorphisms $\sigma:\mathcal{A}\cong\mathcal{A}$ and $x,y\in X$
we have 
\[
d(x,y)=\epsilon\ \Leftrightarrow\ d(\sigma(x),\sigma(y))=\epsilon.
\]
We claim this is impossible. To see this, we'll consider automorphisms
that stretch the distance between points in $X$. Let's work through
the details. 

Since $\langle X,\mathcal{T}\rangle$ satisfies $Reals_{Top}$ we
may fix a homeomorphism $h:\langle\mathbb{R},\mathcal{T}_{int}\rangle\cong\langle X,\mathcal{T}\rangle$.
To make things a little less busy on the page, let us write $\dot{z}$
instead of $h(z)$ when $z\in\mathbb{R}$.  Next for all $c\in\mathbb{R}$,
let $\sigma_{c}:X\to X$ be such that for all $x\in X$, $\sigma_{c}(x)=h\circ(x\mapsto cx)\circ h^{-1}$.
Thus, informally, we take $x$ into the kernel's $\mathbb{R}$, multiply
it by $c$, and then bring the result back to $X$. But then we see
that for very $c\in\mathbb{R}$
\begin{align*}
d(\dot{0},\dot{1})=\epsilon & \Leftrightarrow d(\sigma_{c}(\dot{0}),\sigma_{c}(\dot{1}))=\epsilon\\
 & \Leftrightarrow d(\dot{0},\dot{c})=\epsilon.
\end{align*}
This means that every point other than $\dot{0}$ is the same $d$-distance
away from $\dot{0}$; i.e., for all $y,z\in X$ if $y\neq\dot{0}\neq z$,
then $d(\dot{0},y)=d(\dot{0},z)$. Thus, there is obviously no isometry
between $\langle X,d\rangle$ and $\langle\mathbb{R},d_{Euc}\rangle$
where $d:\mathbb{R}\times\mathbb{R}\to\mathbb{R}$ is the Euclidean
metric. So $t(\mathcal{A})$ does not satisfy $Real_{Met}$ and we
have our contradiction. 
\end{proof}
I think this little proof aligns with our intuition that $Reals_{Top}$
and $Reals_{Met}$ are not interdefinable. It draws out the idea that
placing a metric on the reals is a genuine constraint over and above
its natural topology. We might say that $Real_{Met}$ possesses more
structure than $Reals_{Top}$. Indeed, the proof above can be easily
generalized to show us that the only metrics definable on $Reals_{Top}$
are trivial.\footnote{Strictly, a trivial metric should make the distance between any pair
of points equal to $1$, while we are just saying that they all distinct
points are the same distance apart. We'll bump into this issues again
soon. } But perhaps $Reals_{Top}$ and $Reals_{Met}$ are interdefinable
in a weaker sense. We might worry that they are bi-interpretable. 
\begin{prop}
($\mathcal{V}^{*}$) $Real_{Top}$ and $Reals_{Met}$ are not bi-interpretable.\label{prop:RealTopNMet=0000ACbi}
\end{prop}

\begin{proof}
Both of these theories are categorical, so it will suffice to show
that the unique (up to isomorphism) structured sets that satisfy them
have different automorphism groups. This means that we can work with
the standard representations of these structures that live in the
kernel. Let $\langle\mathbb{R},d_{Euc}\rangle$ and $\langle\mathbb{R},\mathcal{T}_{int}\rangle$
be the canonical structures\footnote{Note that these are not structured sets although, for our current
purposes, this doesn't matter.} witnessing the required isomorphisms of $Reals_{Met}$ and $Reals_{Top}$
respectively. 

Next observe that every isometry $f:\langle\mathbb{R},d_{Euc}\rangle\cong\langle\mathbb{R},d_{Euc}\rangle$
is of the form $f(x)=a+x$ or $f(x)=a-x$; i.e., translations and
reflections. Now let $\mathcal{I}$ be the isometry group on $\langle\mathbb{R},d_{Euc}\rangle$
obtained by removing the nontrivial automorphisms of order $2$; i.e.,
we remove all the reflections. This leaves us with the translations
which gives us an Abelian group. Now let $\mathcal{H}$ be the homeomorphism
group on $\langle\mathbb{R},\mathcal{T}_{int}\rangle$ obtained by
removing the non-trivial automorphism groups of order $2$. $\mathcal{H}$
is not Abelian.\footnote{To spare the weary reader a moment of scribbling, let $f(x)=x+1$
and $g(x)=2x$.} Thus the automorphism groups on $\langle\mathbb{R},d_{Euc}\rangle$
and $\langle\mathbb{R},\mathcal{T}_{int}\rangle$ cannot be isomorphic;
and so Lemma \ref{lem:AutoTemplate} tells us that $Reals_{Met}$
and $Reals_{Top}$ are not bi-interpretable.
\end{proof}
In comparison to the previous proposition, the proof above is a little
indirect.\footnote{But also note that the proof above also suffices to show that $Reals_{Top}$
and $Reals_{Met}$ are not definitionally equivalent. Nonetheless,
we think it is valuable to see how a direct proof of the failure of
definitional equivalence can work.} Nonetheless, this proposition tells us something interesting. Even
if we allow our selves to define structured sets with new domains,
we cannot obtain interpretations between $Reals_{Met}$ and $Reals_{Top}$
that return us to structured sets that are detectably isomorphic to
the ones we started with. For an application of this technique in
philosophy of physics, we have the following:
\begin{thm}
($\mathcal{V}^{*}$, \citealp{HudetzDefEquiv}) Euclidean Geometry
and Minkowski Geometry are not bi-interpretable.\label{prop:EucMink=0000ACbi}
\end{thm}

\begin{proof}
Descriptions of the required structures can be found in \citep{HudetzDefEquiv}.
We leave the simple task of adapting them into the $\mathcal{V}^{*}$
framework to the reader. Both theories are categorical. The automorphism
group of Euclidean geometry is called the Euclidean group, while the
automorphism group of Minkowski geometry is known as the Poincar�
group. They are not isomorphic. Thus, these theories are not bi-interpretable. 
\end{proof}
Thus far, we have only obtained inequivalences between categorical
theories. Given this we might say that we've only learned how to show
when a particular structure is not interdefinable with another. Let
us go beyond categorical theories by generalizing the example above.
Let $Metr$ be the $\mathcal{L}_{\in}(D,d)$-theory of metric spaces
and let $Metr^{ble}$ be the $\mathcal{L}_{\in}(D,d)$-theory of metrizable
spaces. More formally, $Metr$ is the theory of structured sets $\langle X,d\rangle$
where $d:X\times X\to\mathbb{R}$ is a metric. And $Metr^{ble}$ is
the theory of structured sets $\langle X,\mathcal{T}\rangle$ satisfying
$Top$ such that there is a metric $d:X\times X\to\mathbb{R}$ that
generates $\mathcal{T}$. Unlike $Reals_{Top}$ and $Reals_{Met}$,
$Metr^{ble}$ and $Metr$ are both satisfied by many different structured
sets that are not isomorphic to each other. Intuitively, we are probably
inclined to think that they should not be equivalent. 
\begin{prop}
$Metr$ and $Metr^{ble}$ are not definitionally equivalent.\label{prop:Metr=000026Metrble=0000ACDefEq}
\end{prop}

\begin{proof}
Suppose toward a contradiction that we have interpretations $t:Metr\leftrightarrow Metr^{ble}:s$
witnessing their definitional equivalence. Let $\mathcal{A}=\langle A,d\rangle\in Metr$
be such that $A=\{u,v\}$ where $u$ and $v$ are atoms. Since we
have definitional equivalence, the domains of $\mathcal{A}$ and $t(\mathcal{A})$
must be the same. Thus, $t(\mathcal{A})=\langle A,\mathcal{T}\rangle$.
Now since $\langle A,\mathcal{T}\rangle$ is metrizable, we must have
$\mathcal{T}=\mathcal{P}(A)$.\footnote{Any finite metric space must be discrete since we can just find open
balls just containing any particular point by using a distance smaller
than the distance between any of the finitely many pairs of points
in the space.} Let $d_{0}$ be the trivial metric that says the distance between
distinct points is $1$; and let $d_{1}$ be the metric that says
that distance between distinct points is $2$. Then $\langle A,d_{0}\rangle$
and $\langle A,d_{1}\rangle$ both satisfy $Metr$ and we must have
\[
t(\langle A,d_{0}\rangle)=\langle A,\mathcal{P}(A)\rangle=t(\langle A,d_{1}\rangle).
\]
This means $t:Metr\to Metr^{ble}$ is not an injection and so $t$
cannot witness definitional equivalence. 
\end{proof}
Again we seem to have some evidence that our framework is lining up
with intuitions about interdefinability. However, I think there might
be a natural feeling that the proof above is a bit of a cheat. We
are exploiting the fact that there is only one metrizable topology
on any finite set, while there are continuum many metrics that are
compatible with that topology. But in the example above, we might
be tempted to say that there aren't ``really'' continuum many different
metrics since they are isometric up to some scale. More specifically,
given any pair of metrics $d_{0}$ and $d_{1}$ on $A$, there will
be some $c\in\mathbb{R}$ such that $d_{0}(u,v)=c\cdot d_{1}(u,v)$.
There are a couple of things we might say about this. The first is
that we shouldn't be surprised. In moving from $Reals_{Met}$ and
$Reals_{Top}$ to $Metr$ and $Metr^{ble}$ we have taken a step away
from the concrete into a more abstract question. As such, we have
a much wider class of counterexamples available and in such situations
it is not unusual to find that the obvious counterexample feels like
a strange edge case. The second thing to note is that there is some
value in taking the ``cheat'' intuition seriously. Sometimes when
we find an answer to a question, we realize it was the wrong question
and that better questions are available. For example, we might wish
to, so to speak, wash away the effects of the scale issue described
above. It will be instructive to see how this might be done. 

If we want to wash away the scale issue, then one way of doing this
is to reformulate our theory of metric spaces in such a way that we
aren't using a particular metric, but rather, a metric up to scale.
To achieve this, an obvious thing to do is to just use all the metrics
that are identical up to scale. With this in mind, let $Metr^{*}$
be the $\mathcal{L}_{\in}(D,d)$-theory of structured sets $\mathcal{A}=\langle A,\mathcal{D}\rangle$
that says that: for all $d\in\mathcal{D}$, $d:A\times A\to\mathbb{R}$
is a metric; and for all $d_{0}\in\mathcal{D}$ and metrics $d_{1}:A\times A\to\mathbb{R}$,
\[
d_{1}\in\mathcal{D}\ \Leftrightarrow\ \exists c\in\mathbb{R}_{\geq0}\forall x,y\in A\ d_{1}(x,y)=c\cdot d_{0}(x,y).
\]

Thus given $\mathcal{A}=\langle A,\mathcal{D}\rangle$ satisfying
$Metr^{*}$, we don't have access to a particular metric but rather
a set of them that are all the same up to some scale. We might wonder
then if $Metr^{*}$ is interdefinable with $Metr^{ble}$. For quite
obvious reasons, we see that:
\begin{prop}
$Metr^{*}$ and $Metr^{ble}$ are not definitionally equivalent.
\end{prop}

\begin{proof}
Suppose toward a contradiction that we have interpretations $t:Metr^{*}\leftrightarrow Metr^{ble}:s$
witnessing their definitional equivalence. Let $\mathcal{A}_{0}=\langle A,\mathcal{D}_{0}\rangle\in Metr^{*}$
be such that $A=\{u,v,w\}$ where $u,v$ and $w$ are atoms. Since
we have definitional equivalence, the domains of $\mathcal{A}_{0}$
and $t(\mathcal{A}_{0})$ must be the same. Moreover, since $t(\mathcal{A}_{0})=\langle A,\mathcal{T}\rangle$
is metrizable, we again see that have $\mathcal{T}=\mathcal{P}(A)$.
As in the proof of Proposition \ref{prop:Metr=000026Metrble=0000ACDefEq},
we are going to show that $t$ cannot be an injection. 

To see this suppose that $d_{0}\in\mathcal{D}_{0}$ is such that $d_{0}(u,v)=1$,
$d_{0}(u,w)=1$ and $d_{0}(w,v)=1$. Then let $d_{1}:A\times A\to\mathbb{R}$
be the metric that is the same as $d_{0}$ except that $d_{1}(u,v)=2$.
Let $\mathcal{D}_{1}$ be set of metrics that can be obtained from
$d_{1}$ by scaling. It is clear that $d_{0}$ cannot be obtained
from $d_{1}$ by scaling, so $\mathcal{D}_{0}\neq\mathcal{D}_{1}$.
If we $\mathcal{A}_{1}=\langle A,\mathcal{D}_{1}\rangle$ we see that
$t(\mathcal{A}_{1})=t(\mathcal{A}_{0})$ and so $t$ is not an injection. 
\end{proof}
As in the proof of Proposition \ref{prop:Metr=000026Metrble=0000ACDefEq},
we can use the reasoning above to show that there are again continuum
many metrics compatible with $\langle A,\mathcal{T}\rangle$. However
in contrast to that case, this proof doesn't appear to be cheating
with a strange edge case. After removing the scaling issue, the proof
above seems to isolate a significant difference between these theories.
In particular, we are seeing that there are structural differences,
in the form of the distances between points, that cannot be recovered
from the topology. This might give us a better reason to say that
metric spaces and metrizable spaces are not interdefinable. Of course,
this still might inspire us to ask even sharper question. For example,
the counterexamples above both rely on finite spaces which are hardly
the norm in topology. We might then ask what happens if we restrict
our attention to structured sets with infinite domains. Or we might
wonder what happens if we liberalize our equivalence criteria and
bring our second notion of interdefinability. 
\begin{problem}
Are $Metr$ and $Metr^{ble}$ bi-interpretable?
\end{problem}

I suspect they are not, but I do not know. Perhaps there is a metric
space with a sufficiently small and specific isometry group that it
cannot be replicated by the homeomorphism group of any metrizable
space. Or perhaps, as with $Bool$ and $Stone$, there is some way
of defining metric spaces from metrizable spaces that make use of
more complex domains. It seems like an interesting question. 

This concludes the initial phase of our test drive, which provides
an opportune moment to reflect on a deeper question: What is the philosophical
significance of all this? We've now seen a few examples that give
a taste of how this framework may be applied to concrete examples.
But what does it all mean? One thing I think our test drive has illustrated
is that deriving philosophical conclusions from these results is not
immediate or altogether straightforward. It's not a simple matter
of seeing that two theories are, say, definitionally equivalent and
then concluding that they are thus, interdefinable in some intuitive
sense. There is also value and insight to be gained from the proofs
and the specific interpretations they exploit. Sometimes a result
will come too easily and this may reveal something awry in our articulation
of a theory, or that some subtlety is invisible from a certain perspective.
I don't think these are weaknesses of our framework but rather, they
reveal that we have an investigative tool rather than an oracle. We
have an instrument that can be used to sharpen our intuitions and
communicate them more clearly to others. But the task of interpreting
of these results remains in the hands of their philosophically minded
users.

\subsection{A more concrete example of bi-interpretation failure and the coordinate
effect\label{subsec:A-more-concrete}}

So far our examples where bi-interpretability has failed have just
used categorical theories. In this section, we will consider a pair
of theories that are neither categorical nor bi-interpretable with
each other. Beyond filling a gap left open above, we shall also take
this example a step further and use it to illustrate what we might
call the ``coordinate effect.'' In particular, we will consider
what happens when in introduce more background structure and show
that this can have the dramatic effect of changing an inequivalence
into an equivalence. 

Let's first introduce our theories. Let $DiLo$ be the $\mathcal{L}_{\in}(D,d)$-theory
of discrete linear orders without end points. More formally, $DiLO$
is the theory of structured sets of the form $\mathcal{A}=\langle A,\prec\rangle$
such that $\prec\subseteq A\times A$ is a discrete linear so that
every point $a\in A$ has an immediate $\prec$-successor and $\prec$-predecessor.
Let $pDiLO$ be the $\mathcal{L}_{\in}(D,d)$- theory of pointed discrete
linear orders. More precisely, $pDiLO$ is the theory of structured
sets of the form $\mathcal{A}=\langle A,\prec,p\rangle$ where $p\in A$
and $\langle A,\prec\rangle$ satisfies $DiLO$. An automorphism argument
can be used to establish that:
\begin{prop}
$DiLO$ and $pDiLO$ are not bi-interpretable. 
\end{prop}

\begin{proof}
We show that there is some $\mathcal{A}\in pDiLO$ whose automorphism
group is not isomorphic to the automorphism group of any $\mathcal{B}\in DiLO$.
Let $\mathcal{A}=\langle A,\prec,p\rangle$ be isomorphic $\langle\mathbb{Z},<,0\rangle$,
which is defined in the kernel in the usual way that we define the
integers. Clearly $\langle\mathbb{Z},<,0\rangle$ and $\mathcal{A}$
have no nontrivial automorphisms and are thus, rigid. But it is not
difficult to see that every $\mathcal{B}=\langle B,\prec\rangle\in DiLO$
has many nontrivial automorphisms. For example, we can simply send
every element $b\in B$ to its $\prec$-successor. 
\end{proof}
Presumably, this shouldn't be too surprising. The proposition above
seems to confirm that by adding a point to a discrete linear order,
we are adding nontrivial structure that cannot be recovered once lost.
So far so good. But in mathematical logic, questions like these are
often considered in the context of some background structure like
the natural numbers. This is particularly common in computability
theory and descriptive set theory where, for example, one might consider
the set of ill-founded trees on the natural numbers and show that
the real numbers coding them form a $\boldsymbol{\Sigma_{1}^{1}}$-complete
set.\footnote{See Theorem 27.1 in \citep{kechrisCDST} for more details. The $\boldsymbol{\Sigma_{1}^{1}}$-completeness
of this set, $IF$, shows that every other $\boldsymbol{\Sigma_{1}^{1}}$
structure on the natural numbers is can be continuously reduced to
$IF$. This notion of reduction is related to our notion of interpretation,
but is simultaneously weaker and stronger (in two different senses).
Recall that a set of reals is $\boldsymbol{\Sigma_{1}^{1}}$ if it
is definable by a $\Sigma_{1}^{1}$-formula with a real number as
a parameter. This is weaker in the sense that we are using just a
small subset of the formulae available in our framework; but it is
stronger in the sense that we are permitted to use a parameter. Similar
remarks apply to continuous reduction which can be understood as computable
in a real parameter.} We wish to find an analogue to this approach in our $\mathcal{V}^{*}$
framework. To achieve this, we shall expand augment the structures
on our structured sets in such a way as to replicate the background
scaffolding of natural numbers. We now describe the required modifications. 

First let $DiLO_{\mathbb{N}}$ be the $\mathcal{L}_{\in}(D,d)$-theory
of discrete linear orders without end points on the natural numbers.
More formally, $DiLO_{\mathbb{N}}$ is the theory of structured sets
of the form $\mathcal{A}=\langle A,\prec,\dot{0},<_{N}\rangle$ that
says $\prec,<_{N}\subseteq A\times A$ and $\dot{0}\in A$ are such
that: $\langle A,\prec\rangle$ satisfies $DiLO$; and $\langle A,\dot{0},<_{N}\rangle$
is isomorphic to $\langle\omega,\emptyset,\in\rangle$ where $\langle\omega,\emptyset,\in\rangle$
is formed from the usual objects defined in $ZFC$ which reside in
the kernel. Let $pDiLO_{\mathbb{N}}$ be the $\mathcal{L}_{\in}(D,d)$-theory
of pointed discrete linear orders without end points on the natural
numbers. More formally, $pDiLO_{\mathbb{N}}$ is the theory of structured
sets of the form $\mathcal{A}=\langle A,\prec,p,\dot{0},<_{N}\rangle$
that says $\prec,<_{N}\subseteq A\times A$ and $\dot{0},p\in A$are
such that: $\langle A,\prec,p\rangle$ satisfies $pDiLO$; and $\langle A,\dot{0},<_{N}\rangle$
is isomorphic to $\langle\omega,\emptyset,\in\rangle$. 

We see that $DiLO_{\mathbb{N}}$ and $pDiLO_{\mathbb{N}}$ are the
same a $DiLO$ and $pDiLO$ except that we demand that the domains
of these structured sets are scaffolded by the natural numbers. We
will now show that $DiLO_{\mathbb{N}}$ and $pDiLO_{\mathbb{N}}$
are not only bi-interpretable, but definitionally equivalent. While
a direct proof can be given, it will be much easier if we make use
of the following Lemma, which will be discussed further below. It
is a generalization of the Cantor-Bernstein argument into the $\mathcal{V}^{*}$
framework that is very helpful in establishing definitional equivalence. 
\begin{lem}
($\mathcal{V}^{*})$ If interpretations $t:T\to S$ and $t:S\to T$
preserve domains and are injective, then $T$ and $S$ are definitionally
equivalent.\label{lem:(V*)CB} 
\end{lem}

\begin{proof}
The proof is essentially the same as that sketched for Theorem \ref{Cantor-Bernstien}.
So given a structured set $\mathcal{A}=\langle A,a\rangle$ satisfying
$T$, we just need to establish that that we can a define class functions
taking us to $s^{-1}(\mathcal{A})$, $t^{-1}\circ s^{-1}(\mathcal{A})$,
$s^{-1}\circ t^{-1}\circ s^{-1}(\mathcal{A})$ and so on for arbitrarily
many iterations such that this is well-defined. Let's start with $s^{-1}(\mathcal{A})$
and suppose that $\mathcal{B}\in S$ is such that $s(\mathcal{B})=\mathcal{A}$.
Then we want $s^{-1}(\mathcal{A})=\mathcal{B}$. To see how this works,
recall that $s(\mathcal{B})=\mathcal{A}$ when $A=t_{d}^{V(\mathcal{B})}$
and $a=t_{s}^{V(\mathcal{B})}$. Since $s$ is domain-preserving,
we see that $A=B$ and so $V(A)=V(B)$ and we only need to focus on
the structures rather than domains. Then we see that in $\mathcal{V}(\mathcal{A})$
\[
s^{-1}(\mathcal{A})=\mathcal{B}\ \Leftrightarrow\ \exists b\ (\mathcal{B}=\langle A,b\rangle\wedge a=t_{s}^{V(\mathcal{B})}).
\]
Since $s$ is injective there can be at most one such $b$, so this
gives us the definition of a partial function. A similar argument
establishes that we can define a partial function $t^{-1}$. In order
to complete the argument of Theorem \ref{Cantor-Bernstien}, we need
to find the first point at which the sequence 
\[
s^{-1}(\mathcal{A}),(s\circ t)^{-1}(\mathcal{A}),s^{-1}\circ(s\circ t)^{-1}(\mathcal{A}),(s\circ t)^{-2}(\mathcal{A}),\dots
\]
fails to be defined if there is such a point. For this purpose, we
just use transfinite recursion on the natural numbers and the class
functions $s^{-1}$ and $t^{-1}$. This gives us the ability to classify
structures $\mathcal{A}\in T$ as in the proof of Theorem \ref{Cantor-Bernstien},
and the rest of the argument is routine.
\end{proof}
\begin{prop}
$DiLO_{\mathbb{N}}$ and $pDiLO_{\mathbb{N}}$ are definitionally
equivalent.\label{prop:DilLO_NdefEqpDiLO_N}
\end{prop}

\begin{proof}
We plan to use Lemma \ref{lem:(V*)CB} to obtain interpretations $t:DiLO_{\mathbb{N}}\leftrightarrow pDiLO_{\mathbb{N}}:s$
witnessing definitional equivalence. The trick is to use the underlying
scaffolding of natural numbers as coordinates to keep track of things.
First let us define $t^{*}:DiLO_{\mathbb{N}}\to pDiLO_{\mathbb{N}}$
by starting with some $\mathcal{A}=\langle A,\prec,\dot{0},<_{N}\rangle\in DiLO_{\mathbb{N}}$.
We then let $t^{*}(\mathcal{A})=\langle A,\prec,p,\dot{0},<_{N}\rangle$.
We let $p$ be the $<_{N}$-least element of $A$.  Next we define
$s^{*}:pDiLO_{\mathbb{N}}\to DiLO_{\mathbb{N}}$ by starting with
some $\mathcal{B}=\langle B,\triangleleft,p,\dot{0},<_{N}\rangle\in pDiLO_{\mathbb{N}}$.
We want to define $s^{*}(\mathcal{B})=\langle B,\prec,\dot{0},<_{N}\rangle\in$.
For $a,b\in A$, we let $a\prec b$ iff either: 
\begin{itemize}
\item $a$ is even and $b$ is odd; 
\item $a=2n$ and $b=2m$ for some $n,m$ and $m\triangleleft n$; or
\item $a=2n+1$ and $b=2m+1$ for some $n,m$ where $m\triangleleft n$. 
\end{itemize}
Now $t^{*}$ and $s^{*}$ are not inverses of each other, however,
it is not difficult to see that they are both injections that preserve
domains. Thus, we may use Lemma \ref{lem:(V*)CB} to obtain interpretations
$t:DiLO_{\mathbb{N}}\leftrightarrow pDiLO_{\mathbb{N}}:s$ that are
inversions of each other and thus, witness the definitional equivalence
of these theories. 
\end{proof}
Thus, we see that by adding ``coordinates'' in the form of the natural
numbers, a failure of bi-interpretability can be turned into a success
for definitional equivalence. This phenomenon is not uncommon but
warrants further discussion and investigation. I've used the term
``coordinate'' above to highlight a parallel with similar results
in physics. For example, in contrast to Theorem \ref{prop:EucMink=0000ACbi}
we see that: 
\begin{thm}
($\mathcal{V}^{*}$, \citealp{HudetzDefEquiv}) The theories of Euclidean
geometry and Minkowski geometry when endowed with coordinates are
definitionally equivalent. 
\end{thm}

This raises a natural question: what effect should the addition of
coordinates have on our interpretation of interdefinability results?
This is a big question that I don't propose to answer here. A satisfying
investigation would require a patient analysis involving a deft balance
of mathematical and philosophical research. Nonetheless, this phenomenon
has an interesting parallel with a well-known distinction often made
by category theorists: structural vs material set theory. Very broadly,
\emph{material} set theory is generally associated with set theories
like $ZFC$ that are articulated in a first order language with a
single two-place relation symbol, $\in$. By contrast, \emph{structural
}set theory is generally associated with set theories based on generalizations
of category theory like Lawvere's $ETCS$ \citep{Lawvere1506}. What's
the difference? Very briefly, as we've seen above, in a framework
like $ZFC$, a set is understood to have an ancestral structure of
members and members of members etc, which is known as its transitive
closure. When we established that $Set1$ and $Set2$ were sticks
equivalent in Proposition \ref{prop:Set1Set2}, we saw that permitting
free access to information about ancestral structure caused problems
in the form of counterintuitive results. By contrast, in a structural
set theory while we can talk about the members of some set, it makes
no sense to talk about the members of those members.\footnote{In $ETCS$, we say that $x$ is a member of $A$ if $x$ is an arrow
$x:1\to A$ from the terminal object into $A$. Thus, sets and their
members below to different sorts in this framework: sets are objects;
and their members are arrows. See \citep{MeadowsForRev} for more
discussion of this and how it affects the theory of forcing.} Thus, in the context of structural set theory, it doesn't make sense
to talk about the way the elements of the domain of some structure
are built up. The domain is just a set, in much the same way that
a set of atoms in $ZFCU$ is just a set. Given this, we might say
that the approach taken in this paper is a structural one, or at least,
that it occupies an interesting place between material and structural
set theory. We developed the $\mathcal{V}^{*}$ framework with the
goal of hiding material information about the domains of the structures.
We might think then of that material information as being akin to
coordinates. In particular, when we work in the ordinary $ZFC$ universe
-- as set theorists do -- we might think of ourselves as working
within the \emph{maximal} \emph{coordinate} \emph{frame }where every
set has a unique transitive closure distinguishing from any other.
This goes some way to explaining why our initial attempts at characterizing
interdefinability where so close to the brink of triviality. If we
want to understand the definability relationships between theories
and structures, then we need to be able to strip away the idiosyncratic
coordinates that we often use to scaffold them. This much is obvious
in the case of a maximal coordinate frame like the universe of $ZFC$.
But what of more modest coordinates like the natural numbers used
above? On this, I have less to say, so let's make a provocative suggestion.
Perhaps in housing structures within the natural numbers, computability
theory and descriptive set theory are missing out on a host of questions
and problems about mathematical structures that arguably get closer
to their essence. Regardless, the $\mathcal{V}^{*}$ framework may
be useful in providing a bridge between seemingly disparate research
programs conducted in material and structural set theory. 

The reader may also be concerned about our use of Cantor-Bernstein
reasoning in Proposition \ref{prop:DilLO_NdefEqpDiLO_N}. Given that
our move into the $\mathcal{V}^{*}$ framework was motivated by a
counterintuitive result based on similar reasoning, we may wonder
if we've ended up back in the same mess. I don't think that's right
but it warrants discussion. When we showed that $Set1$ and $Set2$
were sticks equivalent, we used a Cantor-Bernstein argument, but that
was not the underlying problem. As we've been discussing above, we
blamed this result on our ability to access material information about
the domains of structures. Moving to the $\mathcal{V}^{*}$ framework
addressed that problem, but -- as we see -- it did not rule out
Cantor-Bernstein arguments. Is this a bad thing? I claim that it is
not. To see this, we note that -- although it is not well-known --
Cantor-Bernstein reasoning can already be applied to ordinary relative
interpretation arguments in the familiar context of first order logic.
Here is a particularly nice example:
\begin{thm}
\citep{GalzerCBDefEq} $ZFC$ with global choice is definitionally
equivalent to $ZFC$ with global well-ordering.\footnote{The proof given on Mathoverflow just claims to give bi-interpretability,
but a quick inspection of the proof reveals that it yields more. Note
that $ZFC$ with global choice is articulated in an expansion of $\mathcal{L}_{\in}$
by a 1-place function symbol, $F$, that selects an element of every
nonempty set; and $ZFC$ with global choice is articulated in $\mathcal{L}_{\in}$
expanded by a 2-place relation symbol, $\prec$, that well-orders
the universe. Very briefly, the proof proceeds by working in an arbitrary
model of $ZFC$ and defining (in that model) injections back and forth
between the choice functions and global well-orderings on that model.
The required bijection is then obtained using a Cantor-Bernstein argument. }
\end{thm}

Assuming that we think that the theory of definitional equivalence
over models of first order logic provides a plausible analysis of
intertranslation between theories, results like the one above tell
us that Cantor-Bernstein reasoning was already part of our tool kit.
As such, there seems to be no particular reason to worry about its
use in the $\mathcal{V}^{*}$ framework. 

Before we move on, it is worth noting that the restriction to domain
preserving interpretations is necessary in Lemma \ref{lem:(V*)CB}.
To see this consider $Set1$ and $pSet2$ where $pSet2$ is the theory
of pointed pairs; i.e., structured sets of the form $\mathcal{A}=\langle A,a\rangle$
where $a\in A$ and $|A|=2$.
\begin{prop}
$Set1$ and $pSet2$ are not definitionally equivalent, but there
are injective interpretations $t:Set1\leftrightarrow pSet2:s$.
\end{prop}

\begin{proof}
Clearly, there cannot be domain preserving interpretations between
$Set1$ and $pSet2$ so they are not definitionally equivalent. We
define $t:Set1\to pSet2$ by taking some $\mathcal{A}=\langle\{u\}\rangle\in Set1$
and letting $t(\mathcal{A})=\langle B,b\rangle$ where we let $v=\{u\}$,
$B=\{\{u,v\},\{v\}$\} and $b=\{v\}$.\footnote{Note that we cannot just let $t(\mathcal{A})=\langle\{u,\{u\}\},u\rangle$
as we did earlier. This is because $u\in trcl(u)$ and so $\{u,\{u\}\}$
cannot be the domain of a quasi-structured set. } $t(\mathcal{A})$ is a quasi-structured set satisfying $pSet2$ and
$t$ is clearly injective. We define $s:pSet2\to Set1$ by taking
$\mathcal{B}=\langle\{m,n\},m\rangle$ and letting $s(\mathcal{B})=\langle\{\{m,n\}\}\rangle$.
$s(\mathcal{B})$ is also a quasi-structured and $s$ is obviously
injective.
\end{proof}

\subsection{Theories of rigid structures and the return of $HOD$ \label{subsec:Theories-of-rigid}}

Thus far, our test drive has gone quite smoothly. We've shown that
standard equivalence argument fit naturally into our framework and
we've developed a basic suite of tools for establishing inequivalence.
I think the framework has a lot to recommend it. Now we are going
to explore its limitations. Up until now, we have mostly be concerned
with comparing theories where at least one of those theories is satisfied
by a structured set that isn't rigid. Our goal now is to consider
theories that are only satisfied by structured sets with no nontrivial
automorphisms. We shall see that in certain circumstances this leads
to counterintuitive results. In particular, if $V=HOD$, then some
surprising equivalences hold. There is an obvious parallel to the
problems we faced in Section \ref{subsec:WhenV=00003DHOD}, however,
we shall also sketch a way around these problems. 

First, let us introduce the leading theories of this section. Let
$Sub\omega$ be the $\mathcal{L}_{\in}(D,d)$-theory of sets of natural
numbers. More precisely, we let $Sub\omega$ be the theory of structured
sets of the form $\mathcal{A}=\langle A,\dot{0},<_{\mathbb{N}},P\rangle$
where $\dot{0}\in A$, $<_{\mathbb{N}}\subseteq A\times A$, $P\subseteq A$
and $\langle A,\dot{0},<_{\mathbb{N}}\rangle$ is isomorphic to $\langle\omega,\emptyset,\in\rangle$.
Let $WO_{\mathfrak{c}}$ be the $\mathcal{L}_{\in}(D,d)$-theory of
nonempty well-orderings shorter than the continuum. More precisely,
$WO_{\mathfrak{c}}$ is the theory of structured sets of the form
$\mathcal{B}=\langle B,\prec\rangle$ where $\prec$ is a well-ordering
of $B$ whose order type is some $\alpha$ with $0<\alpha<\mathfrak{c}=2^{\aleph_{0}}$. 

Observe that $Sub_{\omega}$ and $WO_{\mathfrak{c}}$ are both such
that every structured set that satisfies them is rigid. Thus, if we
want to show that they are inequivalent, we cannot exploit the automorphism
argument strategy we've been using above. Indeed we see that under
certain conditions these theories are interdefinable. 
\begin{prop}
Suppose $V=HOD$. Then in $\mathcal{V}^{*}$, $Sub\omega$ and $WO_{\mathfrak{c}}$
are bi-interpretable.\label{prop:SubOmegaWOc=0000ACbint} 
\end{prop}

\begin{proof}
We aim to show that there are interpretations $t:Sub\omega\leftrightarrow WO_{\mathfrak{c}}:s$
and suitably definable isomorphisms $\mu,\nu$ such that: 
\begin{itemize}
\item $\mu^{V(\mathcal{A})}:\mathcal{A}\cong s\circ t(\mathcal{A})$ for
all $\mathcal{A}\in Sub\omega$; and
\item $\nu^{V(\mathcal{B})}:\mathcal{B}\cong t\circ s(\mathcal{B})$ for
all $\mathcal{B}\in WO_{\mathfrak{c}}$. 
\end{itemize}
For $t$, let us start with some $\mathcal{A}=\langle A,\dot{0},<_{\mathbb{N}},P\rangle$.
We want to define $\mathcal{B}=\langle B,\prec\rangle$ satisfying
$WO_{\mathfrak{c}}$. Fix $P^{*}\subseteq\omega$ such that $\mathcal{A}$
is isomorphic to $\langle\omega,\emptyset,\in,P^{*}\rangle$ in the
kernel. Since $V=HOD$, there is a definable enumeration of the subsets
of $\omega$ of length $\mathfrak{c}$. Let $P^{*}$ be the $\alpha^{th}$
subset of $\omega$ in this ordering. Then $\langle\alpha,\in\rangle$
is a well-ordering of order type $<\mathfrak{c}$. However, $\langle\alpha,\in\rangle$
is not a structured set since its domain is in the kernel. To address
this, we ``convert'' $\langle\alpha,\in\rangle$ into a structured
set. Recall that $\dot{0}\in A$ is an atom. We next define a sequence
$\langle B_{\beta}^{*}\rangle_{\beta\leq\alpha}$ of sets by transfinite
recursion such that: 
\begin{align*}
B_{0}^{*} & =\{\dot{0}\}\\
B_{\beta+1}^{*} & =B_{\beta}^{*}\cup\{B_{\beta}^{*}\}\\
B_{\lambda}^{*} & =\bigcup_{\beta<\lambda}B_{\beta}^{*}\text{ for limit }\lambda.
\end{align*}
Then $\langle B_{\alpha}^{*},\in\rangle\cong\langle\alpha,\in\rangle$.
However, while $\emptyset\notin trcl(B_{\alpha}^{*})$, $trcl(B_{\alpha}^{*})$
cannot be the domain of a quasi-structured set since, for example,
$\dot{0},\{\dot{0}\}\in B_{\alpha}^{*}$ while $\dot{0}\in trcl(\{\dot{0}\})$.
We address this as follows. Let $g$ be the function on $B_{\alpha}^{*}$
such that for all $x\in B_{\alpha}^{*}$
\[
g(x)=\begin{cases}
\{x\} & \text{if }x\neq\dot{0}\\
B_{\alpha}^{*} & \text{otherwise.}
\end{cases}
\]
Then we let $B=g[B_{\alpha}^{*}]$. $B$ can be the domain of a quasi-structured
set. Finally, for $b,c\in B$, we let $b\prec c$ iff either: $b=B_{\alpha}^{*}$;
or $b=\{x\}$, $c=\{y\}$ and $x\in y$ for some $x,y$. 

For $s$, we start with some $\mathcal{B}=\langle B,\prec\rangle\in WO_{\mathfrak{c}}$.
We aim to define $\mathcal{A}=\langle A,\dot{0},<_{\mathbb{N}},P\rangle$.
Let $\alpha<\mathfrak{c}$ be the order type of $\mathcal{B}$. Analogously
to the previous case, let $P^{*}\subseteq\omega$ be the $\alpha^{th}$
subset of $\omega$ in the definable well-ordering given by $V=HOD$.
Then $\langle\omega,\emptyset,\in,P^{*}\rangle$ is like a quasi-structured
set satisfying $Sub\omega$ except that $\emptyset\in trcl(\omega)$.
This can be addressed in much the same way we as we did above. Since
$B$ is nonempty, we may fix the $\prec$-least element of $B$ and
call it $\dot{0}$. Now as in the previous case we build the ordinals
up to $\omega$ over $\dot{0}$ instead of $\emptyset$. Call this
$A_{\omega}^{*}$. Then following the argument of the previous case,
we may define $A$ and $<_{\mathbb{N}}$ by sending appropriate elements
of $A_{\omega}^{*}$ to their singletons. This gives us a quasi-structured
set $\langle A,\dot{0},<_{\mathbb{N}}\rangle$ that isomorphic to
$\langle\omega,\emptyset,\in\rangle$. Moreover, such an isomorphism
is unique. Call it $h$ and let $P=h``\omega\subseteq A$. Then $\mathcal{A}=\langle A,\dot{0},<_{\mathbb{N}},P\rangle$
satisfying $Set\omega$ as required. 

Finally, it is not difficult to see that $\mathcal{A}\cong s\circ t(\mathcal{A})$
and $\mathcal{B}\cong t\circ s(\mathcal{B})$ for all $\mathcal{A}\in Sub\omega$
and $\mathcal{B}\in WO_{\mathfrak{c}}$. Since these isomorphisms
are unique, they are clearly definable and so the proof is complete. 
\end{proof}
Let's discuss the proof first. The main idea is the same as that used
in the proof of Proposition \ref{prop:HODsticks}. We use the definable
well-ordering of the kernel to assign an ordinal to every structure
in the kernel. However, since we are in the $\mathcal{V}^{*}$ framework,
a little more care is required. The structures in the kernel are not
genuine structured sets. Nonetheless, we may use use structures in
the kernel -- as we have many times above -- to define classes of
structured sets and thus, obtain theories. In the proof above, we
can use the ordinals associated with structures in order to obtain
kernel structures associated with each theory that correspond to each
other. But we then need to take the structures from the kernel and
turn them into structured sets. This is where most of the work goes
into the proof above. 

What should we make of this result? Perhaps this does not line up
well with our intuitions.\footnote{For an arguably more counterintuitive result, we might compare $Sub\omega$
with $WO_{\aleph_{1}}$ where the latter is the theory of well-orderings
below $\aleph_{1}$. Under the assumption that $V=L$, essentially
the argument above establishes that they are bi-interpretable. However,
such a bi-interpretation would clearly witness that the continuum
hypothesis holds.} $WO_{\mathfrak{c}}$ is a theory of well-orderings below a certain
length and $Sub\omega$ is the theory of arbitrary subsets of the
naturals. While they have the same cardinality, it seems odd to think
that they are interdefinable. Of course, this is an artifact of our
assumption that $V=HOD$. But just as we saw in Section \ref{subsec:WhenV=00003DHOD},
this means that if we develop our $\mathcal{V}^{*}$ framework using
$ZFC$, then we cannot prove that $Sub\omega$ and $WO_{\mathfrak{c}}$
are not interdefinable. Perhaps this isn't such a big problem. Perhaps
it's reasonable to think they are interdefinable when so much is definable
under the assumption $V=HOD$ and we just have to accept that $ZFC$
doesn't rule out the possibility that $V=HOD$. But then we'd still
want to know that $Sub\omega$ and $WO_{\mathfrak{c}}$ could also
fail to be interdefinbable. Or in other words, we'd want to know that
the bi-interpretability of these theories in $\mathcal{V}^{*}$ is
independent of $ZFC$. A well-known forcing argument establishes that
this is the case.
\begin{prop}
If $ZFC$ is consistent then it's consistent in $\mathcal{V}^{*}$
that $Sub\omega$ and $WO_{\mathfrak{c}}$ are not bi-interpretable.\label{prop:SubWO=0000ACbi}
\end{prop}

\begin{proof}
First let $G$ be $Add(\omega,1)$-generic over $V$ and let us work
in $V[G]$. Suppose toward a contradiction that $Sub\omega$ and $WO_{\mathfrak{c}}$
are bi-interpretable in $\mathcal{V}^{*}$. Then we can use the interpretations
to show that there is a bijection $f:\mathcal{P}(\omega)\to\mathfrak{c}$
that is definable $V[G]$. Let $\Phi(x,y)$ be a formula in $\mathcal{L}_{\in}$
such that for $X\subseteq\omega$ and $\alpha<\mathfrak{c}$ 
\[
f(X)=\alpha\ \Leftrightarrow\ V[G]\models\Phi(X,\alpha).
\]

Now let 
\[
\dot{X}=\{\langle\check{n},p\rangle\in\check{\omega}\times Add(\omega,1)\ |\ p(n)=1\}.
\]
Then it is not difficult to see that $\Vdash\dot{X}\subseteq\omega\wedge x\notin\check{V}$.
Then there must be some $\alpha<\mathfrak{c}$ such that $f(\dot{X}_{G})=\alpha$
and so we may fix $p\in G$ such that: 
\[
p\Vdash\Phi(\dot{X},\check{\alpha}).
\]
Let us now fix an automorphism $\sigma:\mathbb{P}\cong\mathbb{P}$
such that $\sigma(p)=p$ and $p\Vdash\dot{X}\neq\sigma\dot{X}$.
Then we see that\footnote{We rely here on the fact that automorphisms of any poset $\mathbb{P}$
naturally lift to something like an automorphism on $\mathbb{P}$-names,
$\dot{x}\in V^{\mathbb{P}}$. Moreover, it can be seen that $p\Vdash\psi(\dot{x},\dot{y})$
iff $\sigma(p)\Vdash\psi(\sigma\dot{x},\sigma\dot{y})$. See Lemma
5.13 in \citep{jechAC} and Lemma VII.7.13 in \citep{KunenST} for
more details. } 
\[
\sigma(p)\Vdash\Phi(\sigma\dot{X},\sigma\check{\alpha}).
\]
But $\sigma(p)=p$ and $\sigma\check{\alpha}=\check{\alpha}$, so
in $V[G]$ if we let $X=\dot{X}_{G}$ and $X^{*}=(\sigma\dot{X})_{G}$,we
see that: $f(X)=\alpha$; $f(X^{*})=\alpha$; and $X\neq X^{*}$,
which means that $f$ is not an injection.
\end{proof}
Thus, we see that although we cannot prove that $Sub\omega$ and $WO_{\mathfrak{c}}$
are not bi-interpretable, we cannot prove they are bi-interpretable
either. If it seems unpalatable that such questions remain undecidable,
we might propose a provocative response. If we think our foundation
of mathematics -- understood as something like a coarse model for
mathematical practice -- should give be able to answer straightforward
questions about interdefinability and, in particular, that it should
answer the one above negatively, then perhaps we shouldn't use a foundational
theory that assumes so much of the mathematical world can be defined.
Perhaps we should reject that $V=HOD$. 

We should also discuss the proof. We noted earlier that there was
no hope of using our automorphism strategy to show that $Sub\omega$
and $WO_{\mathfrak{c}}$ aren't bi-interpretable but it's plain to
see that an automorphism strategy was nonetheless employed. We might
say that in the absence of the ability to find automorphisms of our
target structures in our current universe, we expanded the universe
and exploited automorphisms of the universe itself. Even in the context
of rigid structures, automorphisms still seem to play a key role in
this framework. The argument used above is a very simple one in this
context. For instance, it relies on the fact that ordinals -- as
residents of the ground universe cannot be changed by automorphisms
of a forcing poset. It would be interesting to see more sophisticated
examples of theories that can be forced to not be bi-interpretable.

\subsection{Comparing particular rigid structures\label{subsec:Comparing-particular-rigid}}

We finish the test drive with a crash test where we'll explore some
deeper limitations of the $\mathcal{V}^{*}$-framework that cannot
be addressed without modifications, which we'll discuss below. In
the examples of the previous section, we investigated theories of
rigid structures and -- with the aid of forcing -- we were able
to recover some form of inequivalence and thus a non-trivial equivalence
relation. In this section, we will discuss categorical theories of
rigid structures; i.e., theories that are satisfied by exactly one
rigid structure up to isomorphism. We might understand this as a way
of comparing a pair of mathematical structures, rather than theories.
We shall see that our framework is all but trivial here. We'll start
with a pair of simple, but counterintuitive, examples concerning bi-interpretability
and definitional equivalence. Then we'll discuss hows these counterintuitive
results can be generalized. Finally, we'll talk about how our framework
might be modified to address this. 

Let's start with a pair of familiar theories. Let $Nat$ be the the
$\mathcal{L}_{\in}(D,d)$-theory of the natural numbers. More formally,
we let $Nat$ be the theory of structured sets of the form $\mathcal{A}=\langle A,\dot{0},<_{\mathbb{N}}\rangle$
where $\mathcal{A}$ is isomorphic to $\langle\omega,\emptyset,\in\rangle$.
Let $Aleph1$ be the theory of the least uncountable ordinal. More
formally, we let $Aleph1$ be the theory of structured sets of the
form $\mathcal{B}=\langle B,\prec\rangle$ where $\mathcal{B}$ is
isomorphic to $\langle\aleph_{1},\in\rangle$. 

Note that $Nat$ and $Aleph1$ are both categorical theories of a
single rigid structure up to isomorphism. Given that $Nat$ talks
about a countable structure and $Aleph1$ talks about one that is
uncountable, it doesn't seem like they should be interdefinable. Nonetheless,
we have the following:
\begin{prop}
($\mathcal{V}^{*}$) $Nat$ and $Aleph1$ are bi-interpretable. 
\end{prop}

\begin{proof}
We define interpretations $t:Nat\leftrightarrow Aleph1:u$ and isomorphisms
$\eta,\nu$ such that:
\begin{enumerate}
\item $\eta_{\mathcal{A}}:\mathcal{A}\cong u\circ t(\mathcal{A})$ for all
$\mathcal{A}\in Nat$; and
\item $\nu_{\mathcal{B}}:\mathcal{B}\cong t\circ u(\mathcal{B})$ for all
$\mathcal{B}\in Aleph1$. 
\end{enumerate}
Let $\mathcal{A}=\langle A,\dot{0},<_{\mathbb{N}}\rangle\in Nat$.
Let $\mathcal{W}=\langle\omega_{1},\in\rangle$ be the canonical model
of the least uncountable well-ordering located in the kernel. This
is is not a quasi-structure. Let $f$ be defined by transfinite recursion
on $\omega_{1}$ be such that 
\begin{align*}
f(0) & =A\\
f(\alpha+1) & =f(\alpha)\cup\{f(\alpha)\}\\
f(\lambda) & =\bigcup_{\alpha<\lambda}f(\alpha).
\end{align*}

Let $B^{*}=f[\omega_{1}]$ and note that $\langle\omega_{1},\in\rangle\cong\langle B^{*},\in\rangle$,
but $\langle B^{*},\in\rangle$ is still not a quasi-structure. Let
$g$ be a function with domain $B^{*}$ which is such that for all
$x\in B^{*}$, $g(x)=x\cup\{B^{*}\}$. Let $B=g[B^{*}]$. Then let
$\prec\subseteq B^{2}$ be such that for all $x,y\in B$
\[
x\prec y\ \Leftrightarrow\ g^{-1}(x)\in g^{-1}(y).
\]
We then see that $\mathcal{B}=\langle B,\prec\rangle$ is a quasi-structure
and $\mathcal{B}\cong\mathcal{W}$. Let $t(\mathcal{A})=\mathcal{B}$.

Now let $\mathcal{B}=\langle B,\prec\rangle\in Aleph1$. Let $\mathcal{N}=\langle\omega,\emptyset,\in\rangle$
be the standard model of the natural numbers. Note that it is not
a quasi-structure. Define $h$ by induction on $\omega$ as follows:
\begin{align*}
h(0) & =B\\
h(n+1) & =h(n)\cup\{h(n)\}.
\end{align*}
Let $A^{*}=h[\omega]$ and note that $\langle A^{*},B,\in\rangle\cong\langle\omega,\emptyset,\in\rangle$
but $\langle A^{*},B,\in\rangle$ is still not a quasi-structure.
Let $i$ be a function with domain $A^{*}$ which is such that for
all $x\in A^{*}$, $i(x)=\{x,\{A^{*}\}\}$. Then let $A=i[A^{*}]$,
$\dot{0}=\{B,\{A^{*}\}\}$ and for $\{x,\{A^{*}\}\},\{y,\{A^{*}\}\}\in A$,
let
\[
\{x,\{A^{*}\}\}<_{\mathbb{N}}\{y,\{A^{*}\}\}\ \Leftrightarrow\ x\in y.
\]
Then note that $\mathcal{A}=\langle A,\dot{0},<_{\mathbb{N}}\rangle$
is a quasi-structure and $\mathcal{A}\cong\mathcal{N}$. Let $u(\mathcal{B})=\mathcal{A}$. 

It is easy to see that $\mathcal{A}\cong s\circ t(\mathcal{A})$ for
any $\mathcal{A}\in Nat$. Moreover, since $\mathcal{A}$ and $s\circ t(\mathcal{A})$
are rigid there is only one function witnessing this isomorphism and
so it is definable. This establishes (1) and a similar argument establishes
(2). 
\end{proof}
Parts of the proof above resemble that of Proposition \ref{prop:SubWO=0000ACbi},
however, note that we did not need to assume that $V=HOD$. This is
an outright theorem. We just find the corresponding structure in the
kernel and then, so to speak, copy and paste it outside the kernel
to form a quasi-structured set. Next we observe that if we do assume
$V=HOD$, then even definitional equivalence can occur in odd places. 

Let's introduce a couple of theories to illustrate this. Let $Reals$
be the $\mathcal{L}_{\in}(D,d)$-theory of the real numbers. More
formally, we let $Reals$ be the theory of structured sets of the
form $\mathcal{A}=\langle A,a\rangle$ where $\mathcal{A}$ is isomorphic
to $\langle\mathbb{R},r\rangle$ where this is defined in the kernel
and $r$ is some standard tuple of constants and operations that pins
down the real numbers. Let $Card\mathfrak{c}$ be the theory of the
continuum length well-ordering. More specifically, $Card\mathfrak{c}$
is the theory of structured sets of the form $\mathcal{B}=\langle B,b\rangle$
where $\mathcal{B}$ is isomorphic to $\langle\mathfrak{c},\in\rangle$. 

Note that both $Reals$ and $Card\mathfrak{c}$ are both categorical
theories of a single rigid structure up to isomorphism. Moreover,
the cardinality of any structured set satisfying $Reals$ will be
the same as the cardinality of any structured set satisfying $Card\mathfrak{c}$.
Despite the fact that their instantiations have the same cardinality,
it would seem -- I think -- odd to say they are interdefinable.
Nonetheless, if $V=HOD$ they meet our highest standard of equivalence. 
\begin{prop}
($\mathcal{V}^{*}$) Suppose $V=HOD$. Let $Reals$ be the theory
of the reals in the usual signature and $Card\mathfrak{c}$ be the
theory of the continuum length well-order. Then $Reals$ and $Card\mathfrak{c}$
are definitionally equivalent. 
\end{prop}

\begin{proof}
We define interpretations $t:Reals\leftrightarrow Card\mathfrak{c}:s$
witnessing definitional equivalence. For the rest of this proof we
fix the canonical model of the reals $\mathcal{R}=\langle\mathbb{R},r\rangle$
and the canonical model of the continuum well-ordering $\mathcal{W}=\langle2^{\aleph_{0}},\in\rangle$.
We suppose without loss of generality that $r$ is an $n$-ary relation
on $\mathbb{R}$. Both of these structures are elements of the kernel.
Moreover, since $V=HOD$ we may fix the $HOD$-least bijection $g:2^{\aleph_{0}}\to\mathbb{R}$.

Let $\mathcal{A}=\langle A,a\rangle\in Reals$. Then there is a unique
-- and thus, definable -- isomorphism $f:\mathcal{R}\to\mathcal{A}$.
Now we let $t(\mathcal{A})=\langle A,\prec\rangle$ where $\prec\subseteq A^{2}$
is such that for all $x,y\in A$
\[
x\prec y\ \Leftrightarrow\ (f\circ g)^{-1}(x)\in(f\circ g)^{-1}(y).
\]
Now let $\mathcal{B}=\langle B,\prec\rangle\in Card\mathfrak{c}$.
Then there is a unique and definable bijection $h:\mathcal{W}\to\mathcal{B}$.
We then let $t(\mathcal{B})=\langle B,a\rangle$ where $a\subseteq B^{n}$
is such that for all $\langle x_{1},...,x_{n}\rangle\in B^{n}$
\[
\langle x_{1},...x_{n}\rangle\in a\Leftrightarrow\langle(g\circ h^{-1})(x_{1}),...,(g\circ h^{-1})(x_{n})\rangle\in r.
\]
\end{proof}

\subsubsection{Generalizing a little}

The examples above can be generalized greatly. However, it will be
easier to illustrate this phenomena using a restricted but very natural
class of structures and after that, we'll discuss how the full generalization
works. 
\begin{defn}
($\mathcal{V}^{*}$) Let us say that a \emph{simple} \emph{structure}
is a set of the form $\mathcal{M}=\langle M,m\rangle$ when $M$ is
an element of the kernel and $m\subseteq M^{n}$ for some $n\in\omega$.
\end{defn}

Thus, a simple structure is domain with an $n$-ary relation on it,
or a first order model for a language with a single $n$-ary predicate.
This should is the ordinary situation in mathematics. The following
lemma aims to generalize the move -- we've seen above -- that allows
us to export structures in the kernel to form genuine quasi-structured
sets in $\mathcal{V}^{*}$. We call this process \emph{atomization}.
\begin{lem}
($\mathcal{V}^{*}$, Atomization) Let $\mathcal{M}=\langle M,m\rangle$
be a simple structure of arity $n$ and let $B$ be a quasi-domain.
Then there is a quasi-structured set $\mathcal{A}=\langle A,a\rangle$
that is definable in $V(B)$ from $\mathcal{M}$ such that $\mathcal{A}\cong\mathcal{M}$.\label{lem:(Simple-atomization-lemma)}
\end{lem}

\begin{proof}
Our plan is to push $M$ out of the kernel and then turn it into a
quasi-domain; and then we'll $m$ accordingly. We start by defining
an injection $f$ by recursion on the transitive closure of $A$ by
letting 
\begin{align*}
f(\emptyset) & =B\\
f(x) & =\{f(y)\ |\ y\in x\}.
\end{align*}
Then we let $A^{*}=f[trcl(A)]$. This recreates the ancestral structure
of $M$ over $B$ instead of the empty set. It doesn't yet satisfy
the conditions for being a quasi-domain, so we then let we let 
\[
A=\{\{x,\{A^{*}\}\}\ |\ x\in A^{*}\}.
\]
It can be seen that $A$ satisfies the conditions for being a quasi-domain.
Finally, we let $a\subseteq A^{n}$ be such that for all $\{x_{1},\{A^{*}\}\},...,\{x_{n},\{A^{*}\}\}\in A$
\[
\langle\{x_{1},\{A^{*}\}\},...,\{x_{n},\{A^{*}\}\}\rangle\in a\Leftrightarrow\langle f^{-1}(x_{1}),...,f^{-1}(x_{n})\rangle\in m.
\]
\end{proof}
It should be clear how to generalize this process to arbitrary models
in first order logic and beyond that to more complex structures like
topologies. Next we isolate a kind of theory that we've used many
times above and, in particular, with the pathological examples of
the current section. 
\begin{defn}
($\mathcal{V}^{*}$) Let us say that $T$ is a \emph{simple rigidly
anchored categorical theory }if there is a definable simple structure,
$\mathcal{M}$, such that $\mathcal{A}\in T$ iff $\mathcal{A}\cong\mathcal{M}$.
We call $\mathcal{A}$ the \emph{anchor} of $T$.
\end{defn}

The underlying idea is very simple. We set out a theory by defining
a canonical structure in the kernel, the anchor, with regard to which
which every structured set in that theory is isomorphic. With the
abstraction moves out of the way, generalization of the examples above
are easily obtained. 
\begin{thm}
($\mathcal{V}^{*}$) Let $T$ and $S$ be simple rigidly anchored
categorical theories. Then $T$ and $S$ are bi-interpretable.\label{thm:SimpRigCatBiInt}
\end{thm}

\begin{proof}
We show that there exist $t:T\leftrightarrow S:s$ and $\eta,\nu$
such that:
\begin{enumerate}
\item $\eta_{\mathcal{A}}:\mathcal{A}\cong s\circ t(\mathcal{A})$ for all
$\mathcal{A}\in T$; and
\item $\eta_{\mathcal{B}}:\mathcal{B}\cong t\circ s(\mathcal{B})$ for all
$\mathcal{B}\in S$. 
\end{enumerate}
Let $\mathcal{A}=\langle A,a\rangle\in T$. Now fix a simple structure
$\mathcal{\mathcal{M}}$ in the kernel such that $\mathcal{C}\in S$
iff $\mathcal{C}\cong\mathcal{\mathcal{M}}$. If we apply Lemma \ref{lem:(Simple-atomization-lemma)}
to $\mathcal{M}$ and $A$, we then obtain a quasi-structured set
$\mathcal{B}$ that is definable in $V(\mathcal{A})$. We let $t(\mathcal{A})=\mathcal{B}$.
For $\mathcal{B}\in S$, we obtain $s(\mathcal{B})$ be the same method.
It should be clear that there exist isomorphisms witnessing (1) and
(2). Moreover, since these structures are rigid, these isomorphisms
are clearly definable. 

\end{proof}
\begin{thm}
($\mathcal{V}^{*}$) Suppose $V=HOD$. Let $T$ and $S$ be simple
rigidly anchored categorical theories whose anchors have the same
cardinality. Then $T$ and $S$ are definitionally equivalent.\label{thm:SimpRigCatSameCardDefEq}
\end{thm}

\begin{proof}
We show that there exist $t:T\leftrightarrow S:s$ such that:
\begin{enumerate}
\item $\mathcal{A}=s\circ t(\mathcal{A})$ for all $\mathcal{A}\in T$;
and
\item $\mathcal{B}=t\circ s(\mathcal{B})$ for all $\mathcal{B}\in S$. 
\end{enumerate}
First we let $\mathcal{M}=\langle M,m\rangle$ and $\mathcal{N}=\langle N,n\rangle$
be the anchors of $T$ and $S$ respectively. Then let $f$ be the
$HOD$-least bijection $f:M\to N$. First we define $t$ by starting
with some $\mathcal{A}=\langle A,a\rangle$ satisfying $T$. Then
we let $g:\mathcal{M}\cong\mathcal{A}$ and note that since it is
unique it is also definable. We let $t(\mathcal{A})=\langle B,b\rangle$
where $B=A$ and $b$ is such that for all $x_{1},...,x_{n}\in A^{n}$
\[
\langle x_{1},...,x_{n}\rangle\in b\ \Leftrightarrow\ \langle f\circ g^{-1}(x_{1}),...,f\circ g^{-1}(x_{n})\rangle\in n.
\]
A similar argument gives us $s(\mathcal{B})$ for $\mathcal{B}\in S$.
We leave to the reader to establish that $t$ and $s$ witness that
(1) and (2) are satisfied. 
\end{proof}
Thus, if we are interested in simple rigidly anchored categorical
theories, the framework we've described above will not make enough
distinctions to be practically useful. However, this collection of
theories has a somewhat technical definition, and so we might hope
that there are other ways of describing categorical theories of simple,
rigid structures that are not exposed to these problems. This is not
the case. 
\begin{prop}
($\mathcal{V}^{*}$) If $T$ is a categorical theory of a simple
rigid structure (up to isomorphism) then it is a simple rigidly anchored
categorical theory. 
\end{prop}

\begin{proof}
Let $\alpha$ be least such that there is a simple structure in $\mathcal{M}\in V_{\alpha}$
(i.e., the kernel) such that $\mathcal{M}\cong\mathcal{A}$ for some
structured set $\mathcal{A}$ satisfying $T$. Our goal is to find
an anchor $\mathcal{M}^{*}$ for $T$. We start by letting $X$ be
the set of simple structures $\mathcal{N}$ in $V_{\alpha}$ that
are isomorphic to $\mathcal{M}$. We then define our anchor $\mathcal{M}^{*}=\langle M^{*},m^{*}\rangle$
by letting $M^{*}$ be the set of functions $f$ with domain $X$
such that for all $\mathcal{N}_{0},\mathcal{N}_{1}\in X$ and $j:\mathcal{N}_{0}\cong\mathcal{N}_{1}$,
$j(f(\mathcal{N}_{0}))=f(\mathcal{N}_{1})$. This makes sense since
$\mathcal{M}$ is rigid and thus, there is exactly one isomorphism
between any pair of simple structures in $X$. Finally, we let $m^{*}$
be such that for $f_{0},...,f_{n}\in M^{*}$
\[
\langle f_{0},...,f_{n}\rangle\in m^{*}\Leftrightarrow\langle f_{0}(\mathcal{M}),...,f_{n}(\mathcal{M})\rangle\in m.
\]
It should be clear that $\mathcal{M}^{*}$ is an anchor for $T$. 
\end{proof}
We note that our requirement that $T$ is a definable class plays
a crucial role in the proof above. For example, if $V\neq HOD$ there
is an isomorphism class of directed graphs that has no definable member.
Thus, the theory of those graph cannot be defined via an anchor.\footnote{To see this, use our assumption that $V=HOD$ to fix $x$ of minimal
rank such that $x\notin OD$. Then it is clear that $\langle trcl(\{x\}),\in\rangle$
is also not ordinal definable Moreover, no directed graph $\mathcal{G}$
with $\mathcal{G}\cong\langle trcl(\{x\}),\in\rangle$ can be ordinal
definable either otherwise we could define $\langle trcl(\{x\},\in\rangle$
by collapsing $\mathcal{G}$.} With regard to definitional equivalence, we might also wonder whether
the assumption that $V=HOD$ is required in Theorem \ref{thm:SimpRigCatBiInt}.
We see that it is. 
\begin{prop}
If $ZFC$ is consistent, then it's consistent in $\mathcal{V}^{*}$
that there are simple categorical theories of rigid structures with
the same cardinality, but they are not definitionally equivalent.
\end{prop}

\begin{proof}
Assume toward a contradiction that the claim is false. Let us work
in $\mathcal{V}^{*}$ of $V[G]$ where $G$ is $Add(\omega,1)$-generic
over $V$. Let $T$ be the theory of structured sets $\mathcal{A}$
that are isomorphic to $\langle2^{\aleph_{0}},\in\rangle$; and let
$S$ be the theory of structured sets $\mathcal{B}$ that are isomorphic
to $\langle V_{\omega+1},\in\rangle$. These are obviously simple
categorical theories of rigid structures with the same cardinality.
Thus, our assumption tells us that $T$ and $S$ are definitionally
equivalent. Now let $\mathcal{A}$ satisfy $T$ and $\mathcal{B}=t(\mathcal{A})\in S$.
There is a unique -- and thus, definable -- isomorphism $f:\mathcal{A}\cong\langle2^{\aleph_{0}},\in\rangle$
and similarly, a unique isomorphism $g:\mathcal{B}\cong\langle V_{\omega+1},\in\rangle$.Thus,
$f\circ g^{-1}:V_{\omega+1}\to2^{\aleph_{0}}$ is a definable bijection.
The proof of Proposition \ref{prop:SubWO=0000ACbi} shows that this
is impossible. 
\end{proof}

\subsubsection{Fully generality\label{subsec:Fully-generality}}

How do we go beyond simple structures? We briefly discuss a hurdle
and then a simple way to get around it. In a nutshell, the triviality
results still apply. Here is the problem. Given a simple structured
sets $\mathcal{A}=\langle A,a\rangle$ where $a\subseteq A^{n}$ for
some $n\in\omega$, it is easy to find a counterpart in the kernel.
We just need some $X$ in the kernel with the same cardinality as
$A$. Given a bijection witnessing this, it is then easy to find a
counterpart $x$ to $a$. This makes it easy to find and use anchors
for simple theories. Indeed, one can easily extend this to more complex
structured sets like topologies. However, when considering a structured
set $\mathcal{A}=\langle A,a\rangle$ in general, it is not so obvious
that it will have a natural counterpart in the kernel. The reason
for this is that the transitive closure of $a$ does not have to be
a subset of the domain $A$ of atoms; it may include elements from
the kernel. If we try to pull such an $\mathcal{A}$ back into the
kernel in the way we did above, the distinction between atoms and
kernel elements will be lost. This causes problems. 

Fortunately, there is a relatively easy way to address this, that
we'll merely sketch.\footnote{The details get fussy quickly and I think the main points of limitation
have been illustrated relatively clearly above. Nonetheless, it seems
important to gesture at how the anchoring method can be generalized.} The idea is simple but, conceptually speaking, it probably takes
a moment or two to digest. Our goal is find natural counterparts to
structured sets from $\mathcal{V}^{*}$ within the kernel. Recall
that the kernel is $V(\emptyset)$ as constructed with in $\mathcal{V}^{*}$.
But from the perspective of our background set theory, the universe
itself is isomorphic to the kernel. Thus, within the kernel we can
define another version of $\mathcal{V}^{*}$. So working in $\mathcal{V}^{*}$,
we let $(\mathcal{V}^{*})^{V(\emptyset)}$ be the big playground,
$\mathcal{V}^{*}$, as defined in there kernel, $V(\emptyset)$. This
means that for any structured set $\mathcal{A}$, there will be a
natural counterpart $\mathcal{A}^{\dagger}$ to $\mathcal{A}$ located
in $(\mathcal{V}^{*})^{V(\emptyset)}$. Moreover, there is a natural
notion of isomorphism between $\mathcal{A}$ and $\mathcal{A}^{\dagger}$
that can be defined in $\mathcal{V}^{*}$. This allows us to use the
structured sets occurring in $(\mathcal{V}^{*})^{V(\emptyset)}$ as
anchors for theories. We may then generalize the results above to
show that: every pair of categorical theories of rigid structures
are bi-interpretable; and if $V=HOD$, then every pair of categorical
theories of rigid structures whose domains have the same cardinality
are definitionally equivalent.

Putting the observations of this section together, we see that there
are significant hurdles to providing a satisfying account of the interdefinability
of theories when the structured sets satisfying them are rigid. In
particular, the $\mathcal{V}^{*}$-framework all but trivially identifies
categorical theories of rigid structures. 

\subsubsection{What can we do?\label{subsec:What-can-we-do?}}

The first thing we should do is concede that this is a genuine limitation
of the $\mathcal{V}^{*}$-framework. While we've done very well with
theories with no rigid structures and quite well with theories having
many rigid structures, we hit a wall when it came to categorical theories
of a rigid structure. Of course, this is just how the system works.
The $\mathcal{V}^{*}$-framework is intended to give us access to
all of our background mathematics when we come to define one mathematical
structure using another. Informally speaking, there just aren't that
many ways a particular rigid structure can be manifested, so it should
not be too surprising that our generous powers of definition allow
us to isolate them. The proofs above attest to how easily and naturally
this can be done. Nonetheless, there is still something counterintuitive
about these results. It seems -- in some sense -- wrong to say that
the theories of $\omega$ and $\aleph_{1}$ are interdefinable. But
just what this sense is, is invisible to our current analysis. 

What should we do? One option is to just bite the bullet. The framework
described above provides a powerful tool for understanding equivalence
in many mathematical instances. But it doesn't cover all of them.
So we have an instrument that works well in many cases. Moreover,
the $\mathcal{V}^{*}$-framework can be understood as pinpointing
a particular sense or meaning of ``interdefinability'' that can
be distinguished from that being used in cases that do not fit within
its confines. But what about these other senses or meanings of interdefinability?
This is work for a future date, however, I think there is value in
sketching one way of generalizing the $\mathcal{V}^{*}$-framework
to deliver more intuitive results in these cases. Besides shedding
a little more light on interdefinability, it will also help us see
that our framework sits at the top of a kind of hierarchy of notions
of definability used in mathematics. 

What should we change? We suggested above that the counterintuitive
results above were caused by our overpowered notion of definability.
So why don't we weaken this and how do we weaken it? Note that as
we weaken our powers of definition, we generally strengthen our equivalence
relations. This is because weaker definability powers generally tie
fewer structures together and thus, make fewer theories equivalent
to each other, meaning that we have a stronger equivalence relation.
So let us consider which interdefinability relations are stronger
than definitional equivalence and bi-interpretability in the $\mathcal{V}^{*}$-framework.
Essentially all of them are. In the context of first order logic and
the standard theory of relative interpretation, definitional equivalence
and bi-interpretability are stronger than their $\mathcal{V}^{*}$counterparts
\citep{Visser2006}. So is Morita equivalence \citep{bHMoritaEquiv}.
So are the second order logic counterparts of definitional equivalence
and bi-interpretability. So are type-theoretic relations considered
by Hudetz \citeyearpar{Hudetz2017,HudetzDefEquiv}.\footnote{Strictly speaking, some care is required around whether the functors
used there are definable, however, the examples considered in those
papers certainly fit into the $\mathcal{V}^{*}$-framework as we've
seen above.} In a nutshell, if two theories are interdefinable according to some
reasonable criteria, then they will be definitionally equivalent or
bi-interpretable in the $\mathcal{V}^{*}$-framework. The landscape
of weaker definability relations is vast and we cannot hope to tame
it here at the end of this long paper. Instead, I propose to consider
two particularly interesting weakenings that naturally generalize
the $\mathcal{V}^{*}$-framework and give us some helpful perspective.
We'll consider something close to the weakest useful notion of definability
and a natural halfway house between it and the full powers of the
framework above. These will be based on ideas from \emph{computability
}and \emph{constructibility} respectively.

Computability theory is one of the most developed parts of mathematical
logic, however, it generally focuses on computation over the natural
numbers (or similar structures). In the spirit of the discussion above,
we'd like to offer an analysis of what it means to say that one structured
set $\mathcal{A}$ can be \emph{computed}, rather than defined, from
another. Let's make things a little easier for ourselves and only
consider structures $\mathcal{A}=\langle A,a\rangle$ that are simple
in the sense that $a\subseteq A^{n}$ for some $n\in\omega$.\footnote{The story below can be generalized to arbitrary structured sets, but
I'm not sure it lines up well with intuitions about computability.
This requires further work.} There is no thought that $\mathcal{A}$ resembles the natural numbers
in any salient way, so a significant generalization of ordinary computability
theory is required. Fortunately, there is a canonical approach that
fits well very naturally into the discussion above. Our remarks below
will be based on \citep{Barwise},\footnote{In particular, see Section II.2 of \citep{Barwise} and for a more
thorough treatment of computability in this framework see \citep{ershov1996definability}.} however, we should note that in the early 1960s this kind of generalized
theory of computation enjoyed interest in a number of camps \citep{MontRecThAsModTh,Mosch1969aAbsComp,Fraisse,Lacombe1,KREISEL2014190}.
Moreover, it was discovered -- somewhat surprisingly -- that each
of these generalized theories were equivalent modulo some minor specializations
\citep{Gordon1970,Lacombe2}. Perhaps -- as with the standard version
of the Church-Turing thesis -- this convergence speaks to the naturalness
of these accounts. We'll leave that for the reader to judge. 

The basic idea is then a variation on the $\mathcal{V}^{*}$ framework.
In particular, we will restrict our attention to interpretations that
are computable in the sense we are about to describe. First, we define
a space over a structured set $\mathcal{A}=\langle A,a\rangle$ in
which we might understand computation as taking place. Working within
the $\mathcal{V}^{*}$-framework, we let 
\begin{align*}
\mathbb{HF}_{0}(A) & =A\\
\mathbb{HF}_{n+1}(A) & =\{X\subseteq\mathbb{HF}_{n}(A)\ |\ |A|<\omega\}\cup\mathbb{HF}_{n}(A)
\end{align*}
and we let $\mathbb{HF}(A)=\bigcup_{n\in\omega}\mathbb{HF}_{n}(A)$,
which we call the \emph{hereditarily finite sets above $A$}. Note
that if $A$ is finite, then $\mathbb{HF}(A)=V_{\omega}(A)$, but
the converse does not hold. Let us write $\mathbb{HF}(\mathcal{A})$
to denote $\langle\mathbb{HF}(A),\in,A,a\rangle$. Thus, we see that
$\mathbb{HF}(\mathcal{A})$ is, so to speak, a smaller version of
$\mathcal{V}^{*}(\mathcal{A})$. Recalling the L�vy hierarchy,\footnote{See Chapter 13 of \citep{JechST}.}
we say that $X\subseteq A^{m}$ for some $m\in\omega$ is \emph{computably
enumerable over $\mathcal{A}$ }if there is a $\Sigma_{1}$ formula
of $\mathcal{L}_{\in}(D,d)$ that defines $X$ over $\mathbb{HF}(\mathcal{A})$.
We say that $X$ is \emph{computable over $\mathcal{A}$} if there
is a $\Sigma_{1}$ formula and a $\Pi_{1}$ formula that both define
$X$ over $\mathbb{HF}(\mathcal{A})$. To illustrate the link with
ordinary computability theory, let $\mathcal{N}$ be a standard model
of arithmetic.\footnote{By this we mean that $\mathcal{N}$ is a structured set that is isomorphic
to the standard model of arithmetic as defined in the kernel. Thus,
$\mathcal{N}$ will be of the form $\langle N,\dot{0},\dot{s},\dot{+},\dot{\times}\rangle$
where there is bijection $f:\omega\to N$ that preserves the structural
elements in the obvious way.} Then if $X$ is computable over $\mathcal{N}$ iff $X$ is computable
in the ordinary sense.\footnote{See Theorem II.2.5 in \citep{Barwise}. Note that $N$ is not a set
in $\mathbb{HF}(\mathcal{N})$ but is computable over $\mathbb{HF}(\mathcal{N})$.
Also see Chapter IV in \citep{kunenFoM} for a detailed treatment
of computability theory in the context of $\mathbb{HF}$ rather than
$\omega$.}

Using this, we may then refine our definition of a $T$-interpretation
from Definition \ref{def:T-interpt}. In particular, we shall say
that $t$ is a \emph{computable $T$-interpretation} if $t$ is a
$T$-interpretation where the $\mathcal{L}_{\in}(D,d)$-formulae $\tau_{d}$
and $\tau_{s}$ that compose $t$ are $\Sigma_{1}$ relativized to
$\mathbb{HF}(D)$ and there are corresponding formulae which are $\Pi_{1}$
over $\mathbb{HF}(D)$ that are equivalent to $\tau_{d}$ and $\tau_{s}$
over quasi-structured sets $\mathcal{A}$ that satisfy $T$. The rest
of the theory can then be modified accordingly by using computable
interpretations rather than their more powerful cousins. Thus, we
end up with a \emph{computable definitional equivalence} and \emph{computable
bi-interpretability}. When a pair of theory are computably definitionally
equivalent, we are saying that there is a uniform way of computing
structures from one theory back and forth from the other that returns
us to exactly where we started. This seems like a natural equivalence
relation to investigate. Moreover, it is one of the strongest equivalence
relations that mathematicians would ordinarily consider. It is obviously
a much stronger equivalence relation than the full definitional equivalence
in the $\mathcal{V}^{*}$-framework, but it is also much stronger
than ordinary interpretation. To see this, let's consider a somewhat
artificial, but very simple example. 

Let $Arith$ be the $\mathcal{L}_{\in}(D,d)$ the theory of structured
sets of the form $\mathcal{N}=\langle N,\dot{0},\dot{s},\dot{+},\dot{\times}\rangle$
where $\mathcal{N}$ is isomorphic to the standard model of arithmetic
in the kernel. Then recall that in computability theory, we often
let $K$ be the set of those $e\in\omega$ such that that $e^{th}$
partial computable function $\Phi_{e}$ halts when given $e$ as its
input.\footnote{See, for example, the definition on page 62 of \citep{RogRecT}.}
Let $Arith+K$ be the $\mathcal{L}_{\in}(D,d)$ be theory of structured
sets of the form $\mathcal{K}=\langle K,\dot{0},\dot{s},\dot{+},\dot{\times},\dot{K}\rangle$
where $\mathcal{K}$ is isomorphic to the standard model expanded
with $K$. Next, we recall that $K$ is computably enumerable but
not computable.\footnote{For a proof, see Theorem VI in Section 5.3 of \citep{RogRecT}.}
Thus, given a structured set $\mathcal{N}$ satisfying $Arith$, the
appropriate version of $K$ for $\mathcal{N}$ can be given a $\Sigma_{1}$
but no $\Pi_{1}$ definition over $\mathbb{HF}(\mathcal{N})$. This
means that we can define a model of $\mathcal{K}$ in any model of
$\mathcal{N}$, which can then be used to show that $Arith$ and $Arith+K$
are definitionally equivalent in the standard sense used in relative
interpretability \citep{Visser2006}. However, no model of $\mathcal{N}$
can compute a suitable version of $K$, so $Arith$ and $Arith+K$
are not computably definitionally equivalent.

For our halfway house between the computable and full versions of
definitional equivalence and bi-interpretability, we look to G�del's
notion of constructibility \citep{GodelConCH}. We don't pretend that
this is the most salient place in this vast space, but it does seem
like a good place to stake out a milestone. Moreover, it fits very
naturally into the $\mathcal{V}^{*}$-framework. We start by defining
a framework to house a structured set $\mathcal{A}$ that will be
somewhere in between $\mathbb{HF}(\mathcal{A})$ and $\mathcal{V}^{*}(\mathcal{A})$.
For this purpose, we recall G�del's $\mathcal{D}$ function which
takes a set $X$ are returns the closure of $X$ under a set of operations
($\mathcal{F}_{1}$-$\mathcal{F}_{12}$) that are intended to simulate
the subsets of $X$ that are simply constructible from $X$.\footnote{It is possible to just define our target structure using definability
instead of the G�del operations and the reader will not lose much
in thinking about it like this. However, it will make comparison to
the short $\mathbb{HF}$structures described above more difficult,
so we opt for a hierarchy that grows a little slower. See the beginning
of Section II.5 in \citep{Barwise} for a discussion of this issue;
and see see Section II.6 for a proper description of the operations
mentioned above.} Working in $\mathcal{V}^{*}$, for a structured set $\mathcal{A}=\langle A,a\rangle$,
we then let
\begin{align*}
L_{0}(A) & =A\\
L_{\alpha+1}(A) & =\mathcal{D}(L_{\alpha}(A))\cup L_{\alpha}(A)\\
L_{\lambda}(A) & =\bigcup_{\alpha<\lambda}L_{\alpha}(A)\text{ for limit ordinals }\lambda
\end{align*}
and we let $L(A)=\bigcup_{\alpha\in Ord}L_{\alpha}(A)$. We then let
$L(\mathcal{A})=\langle L(A),\in,A,\in\rangle$ and call this the
\emph{constructible hierarchy over $\mathcal{A}$}. We then observe
that $L_{\omega}(\mathcal{A})=\mathbb{HF}(\mathcal{A})$ and $\mathbb{HF}(\mathcal{A})\subsetneq L(\mathcal{A})\subseteq\mathcal{V}^{*}(\mathcal{A})$,
so we have a halfway house in terms of definability power. Moreover,
the final $\subseteq$ is also strict if we assume that $V\neq L$.
We then revise our notion of $T$-interpretation in much the same
way as we did to obtain computable $T$-interpretations. In particular,
we shall say that $t$ is a \emph{constructible $T$-interpretation}
if $t$ is a $T$-interpretation where the $\mathcal{L}_{\in}(D,d)$-formula
$\tau_{d}$ and $\tau_{s}$ that compose $t$ are $\Sigma_{1}$ relativized
to $L(D)$ and there are corresponding formulae which are $\Pi_{1}$
over $L(D)$ that are equivalent to $\tau_{d}$ and $\tau_{s}$ over
quasi-structured sets $\mathcal{A}$ that satisfy $T$. As with computable
interpretations, the rest of the theory can be straightforwardly modified
to give us \emph{constructible definitional equivalence} and \emph{constructible
bi-interpretability}. 

Let us close by giving a simple (to state) but somewhat contrived
example to establish that constructible definitional equivalence is
in between computable definitional equivalence and full definitional
equivalence in $\mathcal{V}^{*}$. First, we observe that $Arith$
and $Arith+K$ are constructibly definitionally equivalent and we
already know they are not computably definitionally equivalent. To
break the other two apart, we need to assume that $V\neq L$, so let's
suppose there is a measurable cardinal and recall that this implies
the existence of $0^{\#}$ which is a definable set of natural numbers
that is not a member of $L$.\footnote{See Chapter 18 of \citep{JechST} for more details.}
Analogously to $Arith+K$, let $Arith+0^{\#}$ be the theory of structured
sets isomorphic to the standard model of arithmetic expanded by $K$.
Then we see that $Arith$ and $Arith+0^{\#}$ are definitionally equivalent
in the full $\mathcal{V}^{*}$-framework, but they are not constructibly
definitionally equivalent, since $0^{\#}\notin L(N)$ for any $\mathcal{N}$
satisfying $Arith$.

This is just a brief glimpse of the zoo that lies below the $\mathcal{V}^{*}$-framework.
While it's perhaps disappointing that the $\mathcal{V}^{*}$-framework
handles categorical theories of rigid structures poorly, it's encouraging
to see that the framework can naturally revised in ways that avoid
the problems of Section \ref{subsec:Fully-generality}. This seems
like a good place for future investigation. But even at this early
juncture, I think the $\mathcal{V}^{*}$-framework offers us a new
vantage point from which the full landscape of definability can be
comprehended. 

\section{Conclusion}

So that's the $\mathcal{V}^{*}$-framework. We started with a goal
of analyzing what we mean when we say that two theories consist of
interdefinable structures. We observed that first order logic and
the ordinary theory of relative interpretation face hurdles with theories
in physics and, in particular, those involving topology. While category
theoretic approaches can go some way to answering the question of
when two such theories are equivalent, the simple story about definability
and translation falls away and, with it, the easy argument about the
philosophical significance of such equivalence results. The $\mathcal{V}^{*}$-framework
brings definability back to the conversation by proving a model of
what it means for one structure to be definable from another when
we have the full resources of mathematics behind us. The proposed
framework adopts a structuralist attitude that is brought to life
using set theory with atoms. With it, we are able to take standard
equivalence results straight out of the book and into the $\mathcal{V}^{*}$-framework.
Moreover and more importantly, we are able to establish non-trivial
inequivalences with theories that often line up with results in category
theory but retain a story built upon translation and definability.
I think this paper, while long, is just a first step into a world
of new questions in logic and mathematics. I look forward to exploring
it further myself and seeing what other might do with it.

\end{document}